\documentclass[12pt]{amsart}

\usepackage{amsmath}
\usepackage{amsthm}
\usepackage{amssymb}
\usepackage{amsfonts}
\usepackage{latexsym, graphicx}
\usepackage[top=1in, bottom=1.5in, left=1.5in, right=1.5in]{geometry}
\usepackage{tikz}
\usetikzlibrary{shapes.multipart,shapes,calc,decorations.pathreplacing}

\usepackage{varwidth}
\usepackage{ytableau}
\usepackage{mathdots}
\usepackage{thm-restate}
\usepackage{enumitem}
\usepackage{fancybox}
\usepackage{dsfont}
\usepackage{multicol}
\usepackage{url}

\newtheorem{thm}{Theorem}[section]
\newtheorem{lemma}[thm]{Lemma}
\newtheorem{prop}[thm]{Proposition}
\newtheorem{defn}[thm]{Definition}

\newtheorem{cor}[thm]{Corollary}

\theoremstyle{remark}
\newtheorem{remark}[thm]{Remark}
\newtheorem{obs}[thm]{Observation}
\newtheorem{ex}[thm]{Example}

\newcommand{\wt}{\widetilde}
\newcommand{\wh}{\widehat}

\newcommand{\fq}{\mathfrak{q}}

\newcommand{\cA}{\mathcal{A}}
\newcommand{\cB}{\mathcal{B}}
\newcommand{\cC}{\mathcal{C}}
\newcommand{\cD}{\mathcal{D}}
\newcommand{\cE}{\mathcal{E}}
\newcommand{\cF}{\mathcal{F}}
\newcommand{\cM}{\mathcal{M}}

\newcommand{\cP}{\mathcal{P}}
\newcommand{\cW}{\mathcal{W}}

\newcommand{\ve}{\varepsilon}

\newcommand{\FC}{\operatorname{FC}}
\newcommand{\fc}{\operatorname{fc}}

\newcommand{\Subsk}{\operatorname{Subsk}}
\newcommand{\Std}{\operatorname{Std}}

\newcommand{\abs}[1]{\left\lvert#1\right\rvert}

\title{Double Rim Hook Cluster Algebras}
\author{Michael Chmutov, Pakawut Jiradilok, and James Stevens}

\address{\,}
\email[M.~Chmutov]{mchmutov@gmail.com}

\address{Department of Mathematics, Massachusetts Institute of Technology, Cambridge, Massachusetts}
\email[P.~Jiradilok]{pakawut@mit.edu}

\address{Department of Mathematics, University of Chicago, Chicago, Illinois}
\email[J.~Stevens]{jts@math.uchicago.edu}

\begin{document}

\begin{abstract}
We describe an infinite family of non-Pl\"ucker cluster variables inside the double Bruhat cell cluster algebra defined by Berenstein, Fomin, and Zelevinsky. These cluster variables occur in a family of subalgebras of the double Bruhat cell cluster algebra which we call Double Rim Hook (DRH) cluster algebras. We discover that all of the cluster variables are determinants of matrices of special form. We conjecture that all the cluster variables of the double Bruhat-cell cluster algebra have similar determinant form. We notice the resemblance between our staircase diagram and Auslander-Reiten quivers.
\end{abstract}

\maketitle

\tableofcontents

\section{Introduction}

Cluster algebras have been introduced by Fomin and Zelevinsky \cite{ca1}, and have since developed into a major area of mathematics. Cluster algebras are subalgebras of rational functions in many variables. They are defined via dynamical systems; starting with a seed, which consists of a quiver and a rational function for each vertex of the quiver (called a cluster variable), one performs mutations to the seed to get new collections of cluster variables. The cluster algebra is the algebra generated by all the cluster variables obtained by mutations from the original seed. It turns out that many well algebras have cluster algebra structure; in particular coordinate rings of many algebraic varieties have this property, thus allowing one to use combinatorial methods to study geometry. More information on the current state of the field can be found in \cite{CAP}.

This paper is concerned with a particular cluster algebra, namely the coordinate ring of the open double Bruhat cell $G^{w_0,w_0}$ of $GL_n(\mathbb{C})$. This cluster algebra is generally known to be of wild type, namely the number of cluster variables is infinite and even the number of possible quivers obtained by mutations from the initial seed is infinite. As such, a full combinatorial understanding of it is somewhat hopeless. The main result of the report is a description of a family of subalgebras of finite type, which we refer to as the double rim hook algebras, and an explicit description of all the cluster variables of these algebras. As far as we know this is the first explicit description of infinite families of non-Pl\"ucker cluster variables in the double Bruhat cell cluster algebra.

It is known that all the minors of a matrix with distinct variables for entries are cluster variables of the double Bruhat cell algebra. The cluster variables of the double rim-hook subalgebras are not minors of such a matrix, but they are all determinants of a particular form. We suspect that all cluster variables of the double Bruhat cell algebra are of this form, but we do not know how one would approach a proof of this statement.

\bigskip

\section{Open double Bruhat cell cluster algebra}

As shown in \cite{ca3}, the coordinate ring of the open double Bruhat cell $G^{w_0,w_0}$ of $GL_n(\mathbb{C})$ has the structure of a cluster algebra (in fact, that was shown for all double Bruhat cells). Following \cite{fomin-reading}, we define this cluster algebra via a combinatorial gadget called double wiring diagrams. An example is shown in Figure \ref{fig:wd}. Such a diagram consists of two collections of piecewise straight lines, denoted by two colors, with the property that each pair of lines of the same color intersects precisely once. Thus, in each color we end up with a diagram for a reduced expression of the longest element of the symmetric group. The lines of each color are numbered so that the left endpoints of the lines end up numbered bottom to top. The \emph{chambers} of a double wiring diagram are the areas of each level between, or to the side of, crosses. Each of these is labeled by the indices of lines below it (see Figure \ref{fig:wd with labels}). 

\begin{figure}
\centering
\begin{minipage}{0.47\textwidth}
\centering
\resizebox{\textwidth}{!}{\input{wd.pspdftex}}
\caption{A double wiring diagram on four strands.}
\label{fig:wd}
\end{minipage}\hfill
\begin{minipage}{0.47\textwidth}
\centering
\resizebox{\textwidth}{!}{\input{wd_with_labels.pspdftex}}
\caption{Chamber labels for the wiring diagram.}
\label{fig:wd with labels}
\end{minipage}
\end{figure}

\begin{remark}
In the literature, the labels on the red lines usually go top to bottom. Our convention was chosen to work better for the correspondence with the double rim hook algebras. It will, however, result in using axis-coordinates when talking about matrices; for example we will label the bottom left entry of an $n\times n$ matrix $X$ by $x_{11}$ and the top left by $x_{1n}$.
\end{remark}

Now we describe a seed for this cluster algebra, starting with the quiver. Fix a positive integer $n$. Choose a double wiring diagram $D$ on $n$ strands. The vertices of the quiver are the chambers of $D$, except for the bottom chamber labeled by $(\varnothing,\varnothing)$. The arrows of the quiver are placed according to the local rules in Figure \ref{fig:wd quiver rules}; these ensure that the mutation is as described in \cite[Lemma 4.23]{fomin-reading}. The first column just claims that every red crossing has an arrow from the chamber on its left to the chamber on its right, and the opposite is true of the blue crossings. The second and third columns deal with the cases when one chamber is completely above or below another. The last two columns deal with the cases when the two chambers are on adjacent levels but are not completely above or below each other. All the cases that are unaccounted for do not have arrows. See Figure \ref{fig:quiver_example} for the quiver corresponding to our example of the double wiring diagram.

\begin{figure}
$$\resizebox{.6\textwidth}{!}{\input{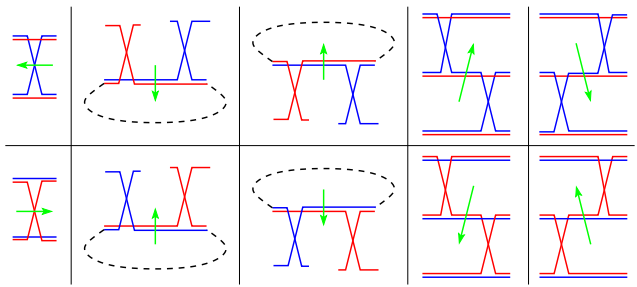}}$$ 
\caption{Local rules for arrows of a quiver coming from a double wiring diagram.}
\label{fig:wd quiver rules}
\end{figure}

\begin{figure}
$$\resizebox{.7\textwidth}{!}{\input{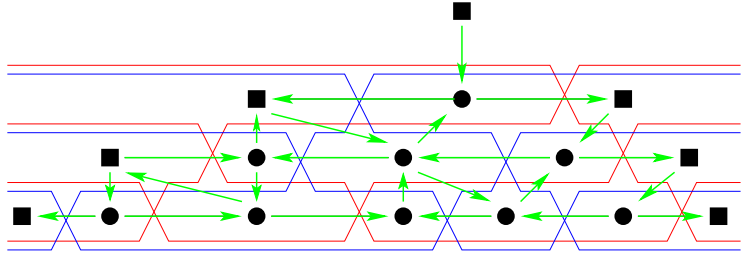}}$$ 
\caption{An example of a quiver corresponding to a double wiring diagram.}
\label{fig:quiver_example}
\end{figure}

Now we describe the initial cluster. Let $X = (x_{ij})$ be an $n\times n$ matrix with distinct variables as entries (recall that we are using axis-coordinates to denote matrix entries). The variables attached to vertices are the minors of $X$ whose columns are the red labels of a chamber and whose rows are the blue labels of a chamber. 

\bigskip

\section{Double rim hooks}
\subsection{Double Rim Hooks} \label{ss:introboa}

A {\em North-East lattice path} $\lambda$ is a finite word in the letters N and E. We think of $\lambda$ as a zigzag segment, starting from a point on the plane, going up one unit when a letter N is read, and going right one unit when a letter E is read, with the reading order of the word $\lambda$ being from the left to the right. For example, we consider the path NNE as the zigzag path starting from a point on the plane, going up two units and then going right one unit. 

The lines $X = k$ for all $k \in \mathbb{Z}$ and $Y = \ell$ for all $\ell \in \mathbb{Z}$ make the plane an infinite chessboard. For a North-East lattice path $\lambda$ of length $l = l(\lambda)$, we define the {\em $\lambda$-array} as a subset of the infinite chessboard as follows. Consider the path $\lambda$ starting at the vertex $(1,1)$. The {\em $\lambda$-array} is defined to be the union of all the $2 \times 2$-squares whose centers are lattice points on $\lambda$. Evidently, the $\lambda$-array contains $2\cdot l(\lambda)+4$ connected unit squares (cells). We denote the set of these cells by $\cC(\lambda)$. We may associate a rational function to each cell of $\cC(\lambda)$. A {\em $\lambda$-Double Rim Hook} (or {\em $\lambda$-DRH}) is the data of the $\lambda$-array together with the $2 \cdot l(\lambda) + 4$ rational functions. When we put the indeterminate variable $a_{ij}$ into the cell whose upper right vertex is $(i,j)$ for all cells in $\cC(\lambda)$, the resulting $\lambda$-DRH will be called the {\em initial $\lambda$-DRH}.

For example, the initial NNE-DRH is shown below. The NNE path is shown in orange.

\medskip

\begin{center}
\begin{tikzpicture}
\ytableausetup{notabloids}
\ytableausetup{mathmode, boxsize=2.0em}
\node (n) {\begin{ytableau}
a_{14} & a_{24} & a_{34} \\
a_{13} & a_{23} & a_{33} \\
a_{12} & a_{22} & \none \\
a_{11} & a_{21} & \none \\
\end{ytableau}};
\draw[very thick,orange] (-0.42,-0.84) -- (-0.42,0.84) -- (0.42,0.84);
\end{tikzpicture}

\medskip

\text{The initial NNE-DRH}
\end{center}

\subsection{Initial DRH quiver} \label{ss:iniquiv} In this section, we shall describe the construction of the {\em initial $\lambda$-DRH quiver}. In the initial $\lambda$-DRH, let $c_{ij} \in \cC(\lambda)$ denote the cell with $a_{ij}$.

First, we freeze the lower left corner cell, $c_{11}$, and the upper right corner cell, $c_{pq}$, where $p$ is the number of E's in $\lambda$ plus $2$, and $q$ is the number of N's in $\lambda$ plus $2$. Then, we define a subset $\cM(\lambda) \subseteq \cC(\lambda)$ algorithmically as follows. Suppose initially there is a bug in the cell $c_{12}$, which is the one above the lower-left frozen cell. The goal of the bug is to jump one unit at a time, north or east, until he reaches a cell adjacent to the upper right frozen cell, at which point the process terminates. Starting in the cell $c_{12}$, as long as he is not next to the upper right frozen cell, he goes right as much as possible until he hits a wall, then goes up as much as possible, then right, then up, and so on, and he eventually stops in a cell adjacent to the upper right frozen cell. The set $\cM(\lambda)$ is defined to be the set of $l+1$ cells the bug has been in.

The quiver $Q_{\lambda}$ is specified as follows. The mutable vertices are precisely the $l(\lambda)+1$ cells in $\cM(\lambda)$. There are $l(\lambda)+3$ frozen vertices, including the $l(\lambda)+1$ connected $2 \times 2$-squares in the $\lambda$-array, and the two previously frozen cells ($c_{11}$ and $c_{pq}$). The variables associated with the vertices are given as follows. To each mutable vertex $c_{ij} \in \cM(\lambda)$, we associate the variable $a_{ij}$. To $c_{11}$ and $c_{pq}$, we associate $a_{11}$ and $a_{pq}$, respectively. To each $2 \times 2$-frozen vertex, we associate the $2 \times 2$-determinant of the square.

The arrows between the mutable vertices in $Q_{\lambda}$ are defined as follows. Two mutable vertices $V$ and $V'$ are connected by an arrow if the corresponding cells share an edge. If $V'$ is to the right of $V$, the arrow is from $V$ to $V'$. If $V'$ is above $V$, the arrow is from $V'$ to $V$.

In order to describe the arrows between the frozen vertices and the mutable vertices, we first look at some determinantal identities. Consider a $2 \times 2$ array with variables inside as shown below.

\medskip

\begin{center}
\ytableausetup{notabloids}
\ytableausetup{mathmode, boxsize=2.0em}
\begin{ytableau}
a & b \\
c & d \\
\end{ytableau}
\end{center}

\medskip

Here, we have the first determinantal identity 
\begin{itemize}
\item (D1) $a \cdot d =  \begin{vmatrix} a & b \\ c & d \end{vmatrix} + b \cdot c$. 
\end{itemize}
Next, consider a $2 \times 3$ array with variables inside as shown below.

\medskip

\begin{center}
\ytableausetup{notabloids}
\ytableausetup{mathmode, boxsize=2.0em}
\begin{ytableau}
a & b & c \\
d & e & f \\
\end{ytableau}
\end{center}

\medskip

Here, we have two determinantal identities: 
\begin{itemize}
\item (D2) $b \cdot \begin{vmatrix} a & c \\ d & f \end{vmatrix} = a \cdot \begin{vmatrix} b & c \\ e & f \end{vmatrix} + c \cdot \begin{vmatrix} a & b \\ d & e \end{vmatrix}$, and 
\item (D3) $e \cdot \begin{vmatrix} a & c \\ d & f \end{vmatrix} = d \cdot \begin{vmatrix} b & c \\ e & f \end{vmatrix} + f \cdot \begin{vmatrix} a & b \\ d & e \end{vmatrix}$.
\end{itemize}

Finally, consider a $3 \times 2$ array with variables inside as shown below.

\medskip

\begin{center}
\ytableausetup{notabloids}
\ytableausetup{mathmode, boxsize=2.0em}
\begin{ytableau}
a & d \\
b & e \\
c & f \\
\end{ytableau}
\end{center}

\medskip

We have two more determinantal identities:
\begin{itemize}
\item (D4) $b \cdot \begin{vmatrix} a & d \\ c & f \end{vmatrix} = a \cdot \begin{vmatrix} b & e \\ c & f \end{vmatrix} + c \cdot \begin{vmatrix} a & d \\ b & e \end{vmatrix}$, and 
\item (D5) $e \cdot \begin{vmatrix} a & d \\ c & f \end{vmatrix} = d \cdot \begin{vmatrix} b & e \\ c & f \end{vmatrix} + f \cdot \begin{vmatrix} a & d \\ b & e \end{vmatrix}$.
\end{itemize}

The arrows between the frozen vertices and the mutable vertices are defined locally at each mutable vertex, so that the resulting exchange relation is one of the five determinantal identities above.

For example, if $\lambda = E$, then the initial quiver $Q_{\lambda}$ is shown in Figure \ref{fig:Q_E}. In this quiver, the two determinantal identities involved are
\begin{align*}
& a_{12} \cdot \ovalbox{$a_{21}$} = \begin{vmatrix}
a_{12} & a_{22} \\
a_{11} & a_{21} 
\end{vmatrix} + a_{11} \cdot a_{22} \\
& a_{22} \cdot \ovalbox{$\begin{vmatrix}
a_{12} & a_{32} \\
a_{11} & a_{31}
\end{vmatrix}$} = a_{12} \cdot \begin{vmatrix}
a_{22} & a_{32} \\
a_{21} & a_{31}
\end{vmatrix} +
a_{32} \cdot \begin{vmatrix}
a_{12} & a_{22} \\
a_{11} & a_{21}
\end{vmatrix}.
\end{align*}
The boxed expression in each equation is the cluster variable into which the original mutable variable transforms when we mutate the quiver at the corresponding node.

The $\lambda$-DRH cluster algebra $\cA_{\lambda}$ is defined as the type-A cluster algebra generated from the quiver $Q_{\lambda}$. Our main result is the explicit description of all the cluster variables of $\cA_{\lambda}$, for every North-East lattice path $\lambda$.

\begin{center}
\begin{figure}
\begin{tikzpicture}
  [scale=2.5,auto=left,every node/.style={circle,fill=cyan!20,minimum width=2 em, draw = black}]
  \node[rectangle] (D1) at (1,3) {$\begin{vmatrix} a_{12} & a_{22} \\ a_{11} & a_{21} \end{vmatrix}$};
  \node[rectangle] (D2) at (2,3) {$\begin{vmatrix} a_{22} & a_{32} \\ a_{21} & a_{31} \end{vmatrix}$};
  \node[fill = none] (M1) at (1,2)  {$a_{12}$};
  \node[fill = none] (M2) at (2,2)  {$a_{22}$};
  \node[rectangle] (F1) at (1,1) {$a_{11}$};
  \node[rectangle] (F2) at (2,1) {$a_{32}$};
  
  \draw[line width = 0.1 em,->,>=stealth] (D1) -- (M1);
  \draw[line width = 0.1 em,->,>=stealth] (M1) -- (F1);
  \draw[line width = 0.1 em,->,>=stealth] (D2) -- (M2);
  \draw[line width = 0.1 em,->,>=stealth] (M2) -- (F2);
  \draw[line width = 0.1 em,->,>=stealth] (M1) -- (M2);
  \draw[line width = 0.1 em,->,>=stealth] (M2) -- (D1);
\end{tikzpicture}
\caption{the initial quiver of the E-DRH. The mutable vertices are drawn as circles and the frozen ones as rectangles.}\label{fig:Q_E}
\end{figure}
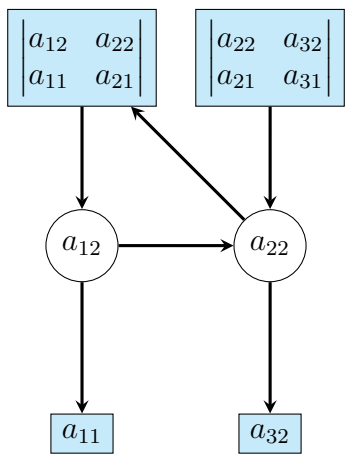
\end{center}

\bigskip

\section{Constructing a double wiring diagram from a DRH}

It is sufficient to construct a double wiring diagram for a DRH whose skeleton has as many E's as it has N's (since we can extend any DRH to one of these). 

The bottom layer is formed as follows. Extend the worm to the southwest cell and to the northeast cell. Following this extended worm from southwest to northeast, place a blue crossing for each N-step, and a red crossing for each each E-step (see Figures \ref{fig:DRH2dwd_1} and \ref{fig:DRH2dwd_2}).

\begin{figure}
$$\resizebox{.3\textwidth}{!}{\input{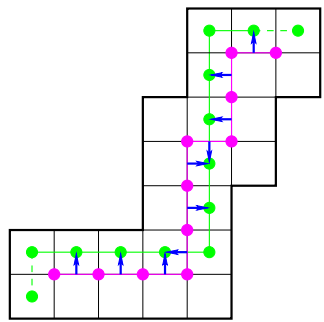}}$$ 
\caption{A DRH with its extended worm, skeleton, and a choice of correspondence.}
\label{fig:DRH2dwd_1}
\end{figure}

\begin{figure}
\centering
\resizebox{\textwidth}{!}{\input{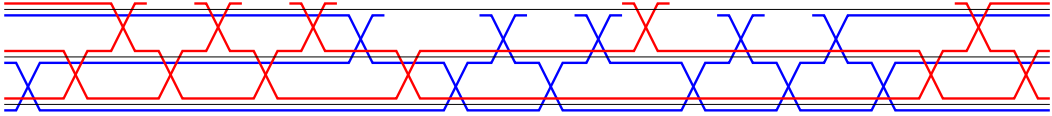}}
\caption{The bottom two rows of the wiring diagram corresponding to the DRH and correspondence choice in the previous figure.}
\label{fig:DRH2dwd_2}
\end{figure}

The second row from the bottom is constructed from the skeleton. First, for each step (edge) of the skeleton, choose a vertex of the worm that is in a cell adjacent to it (an example of such a choice is given by the blue arrows in Figure \ref{fig:DRH2dwd_1}). We will refer to this choice as a choice of correspondence; it is clear that such a choice always exists, but it may not be unique. In the same way as for the bottom row, looking at the skeleton from southwest to northeast, place a blue crossing for every N step of the skeleton and a red crossing for every E step of the skeleton. The placement of each crossing is determined by the associated vertex of the worm, namely the second row cross goes between the bottom crosses corresponding to the steps of the worm adjacent to the vertex (see Figures \ref{fig:DRH2dwd_1} and \ref{fig:DRH2dwd_2}). 

\begin{lemma}
\label{lem:wiring_ok}
In the above construction, between two crosses of the same color in the bottom row, there is a cross of the same color in the top row.
\end{lemma}

Before proving the lemma, let us finish the description. Continue filling the upper rows in some way to achieve a valid double wiring diagram (since we will freeze the corresponding variables as well anything connected to them, we do not care about the way this is achieved).

\begin{ex}
The wiring diagram in Figure \ref{fig:wd} corresponds to the ENNE DRH (with some choice of correspondence).
\end{ex}

\begin{proof}[Proof of Lemma \ref{lem:wiring_ok}]
First we show that every vertex of the worm which is not a corner has an arrow pointing at it. Note that the only cells of the DRH which do not have the skeleton on one edge are the corners. Consider a a vertex $v$ of the worm which is not a corner of it (hence it is not in a corner of the DRH). We already noted that the skeleton covers one edge of the cell containing this vertex. Moreover, one of the two edges parallel to the worm must be covered by the skeleton since otherwise we will violate the condition that the DRH cannot have three cells on the same diagonal. The arrow from this edge must necessarily point to $v$. 

The above paragraph, proves the lemma in case when the two crosses of the same color in the bottom row come consecutively. Suppose not, i.e. we have two crosses of the same color separated by some crosses of the opposite color. For the worm, we are looking at two consecutive bends. Now we notice that, by construction, for every corner of the extended worm, there exists a vertex of the skeleton adjacent to both edges of the worm. Using this observation, we see that between the two corner vertices of the worm, the skeleton must cross the worm. The edge of the skeleton at the crossing corresponds to the sought crossing of the same color in the second row.
\end{proof}

\begin{prop}
Consider a DRH $V$, and its initial seed $(Q, \mathbf{x})$. Form a double wiring diagram from $V$ as described above, take its initial seed, freeze the all the vertices corresponding to the chambers which are not in the bottom row, and get rid of all the disconnected frozen vertices. Call the resulting seed $(Q', \mathbf{x'})$. Then the seeds $(Q, \mathbf{x})$ and $(Q', \mathbf{x'})$ are equivalent.
\end{prop}

\begin{proof}
There is a natural bijection between the vertices of the two quivers which takes frozen vertices precisely to frozen vertices. We only need to check that the connections from every mutable vertex in $Q$, i.e. an internal vertex of the extended worm, are the same in both scenarios. Choose an internal vertex $v$ of the extended worm. There are two cases based on whether or not $v$ is a corner of the extended worm.

Suppose $v$ is not a corner of the extended worm; without loss of generality assume the worm is horizontal at $v$. Let $v_p$ and $v_n$ be the previous, and next vertices of the worm, respectively. As seen in the proof of Lemma \ref{lem:wiring_ok}, a horizontal edge of the skeleton covers an edge of the cell containing $v$. Let $w$, $w'$ be the two vertices of that edge, with $w$ earlier than $w'$ in the skeleton. Every vertex of the skeleton corresponds to a frozen variable; let $\Delta$ and $\Delta'$ be the frozen variables corresponding to $w$ and $w'$, respectively. Then in $Q$, $v$ has incoming arrows from $v_p$ and $\Delta'$ and outgoing arrows to $v_n$ and $\Delta$. 

Now let us analyze the wiring diagram picture. The extended worm edges from $v_p$ to $v$ and from $v$ to $v_n$ correspond to two consecutive red crossings in the bottom row. The skeleton edge from $w$ to $w'$ corresponds to a red crossing in the second row between the two. There may, or may not, be at most one additional blue crossing in the second row between the two. Regardless of that there will be an arrow into the chamber corresponding to $v$ from the chamber corresponding to $w'$ and an arrow out of the chamber corresponding to $v$ to the chamber corresponding to $w$. Thus indeed the two local pictures of the quivers coincide.

The case when $v$ is a corner of the extended worm, the analysis is similar; this is the case corresponding to the $2\times 2$ matrix determinantal identity.
\end{proof}

\bigskip

\section{DRH staircase}

\subsection{Construction of the DRH staircase} \label{ss:constboa}
A convenient way to study the cluster variables is through the DRH staircase, a ``board'' which serves as a bookkeeping device to record the cluster variables and the combinatorial relationships thereof.

To construct the DRH staircase for each given North-East lattice path $\lambda$, we consider the infinite chessboard $\mathbb{Z} \times \mathbb{Z}$. We will label the cells using the matrix convention: we write the integer row indices in the increasing order {\em from the top down to the bottom}, and write the integer column indices in the increasing order {\em from the left to the right}. We embed the $\lambda$-array into this chessboard so that the bottom-leftmost cell of the DRH coincides with the $(1,0)$-cell of the chessboard. Here, the $(i,j)$-cell of the chessboard refers to the cell in column $i$ and row $j$. The row and column indices of this board will then be taken modulo $l+4$. For example, Figure \ref{fig:NNE_SS} shows the NNE-DRH staircase. There, the embedded NNE-DRH is drawn in cyan. We will explain the thick zigzag lines and the green cells later.

\begin{center}
\begin{figure}
\begin{tikzpicture}
\ytableausetup{notabloids}
\ytableausetup{mathmode, boxsize=2.0em}
\node (n) {\ytableausetup{nosmalltableaux}
\ytableausetup{notabloids}
\ydiagram[*(cyan)]{3,3,2,2}
*[*(green)]{0,0,0,0,3+4,3+4,3+2}
*[*(white)]{3,3,4,5,2+4,3+4,3+5}};
\draw[line width= 2.0 pt,black] (-3.24,0.42) -- (-2.40,0.42) -- (-2.40,-0.42) -- (-1.56,-0.42) -- (-1.56,-1.26) -- (-0.72,-1.26) -- (-0.72,-2.10) -- (0.12,-2.10) -- (0.12,-2.94) -- (0.96,-2.94);
\draw[line width= 2.0 pt,black] (-2.40,2.94) -- (-1.56,2.94) -- (-3.24+1.68,0.42+1.68) -- (-2.40+1.68,0.42+1.68) -- (-2.40+1.68,-0.42+1.68) -- (-1.56+1.68,-0.42+1.68) -- (-1.56+1.68,-1.26+1.68) -- (-0.72+1.68,-1.26+1.68) -- (-0.72+1.68,-2.10+1.68) -- (0.12+1.68,-2.10+1.68) -- (0.12+1.68,-2.94+1.68) -- (0.96+1.68,-2.94+1.68) -- (0.96+1.68,-2.10) -- (3.48,-2.10) -- (3.48,-2.94);
\node at (-4.26,2.52) {$-3 = 4$};
\node at (-4.26,1.68) {$-2 = 5$};
\node at (-4.26,0.84) {$-1 = 6$};
\node at (-4.26,0) {$0 = 7$};
\node at (-3.84,-0.84) {$1$};
\node at (-3.84,-1.68) {$2$};
\node at (-3.84,-2.52) {$3$};
\node at (-2.82,3.36) {$1$};
\node at (-1.98,3.36) {$2$};
\node at (-1.14,3.36) {$3$};
\node at (-0.30,3.36) {$4$};
\node at (0.54,3.36) {$5$};
\node at (1.38,3.36) {$6$};
\node at (2.34,3.36) {$7$};
\draw (-2.82,0.84) node[circle,minimum size = 7 pt, inner sep = 0 pt, fill = red]{};
\draw (-1.98,0.84) node[circle,minimum size = 7 pt, inner sep = 0 pt, fill = red]{};
\draw (-1.98,1.68) node[circle,minimum size = 7 pt, inner sep = 0 pt, fill = red]{};
\draw (-1.98,2.52) node[circle,minimum size = 7 pt, inner sep = 0 pt, fill = red]{};
\draw[line width= 2.0 pt,red] (-2.82,0.84) -- (-1.98,0.84) -- (-1.98,1.68) -- (-1.98,2.52);
\end{tikzpicture}

\medskip

\caption{The NNE-DRH staircase} \label{fig:NNE_SS}
\end{figure}
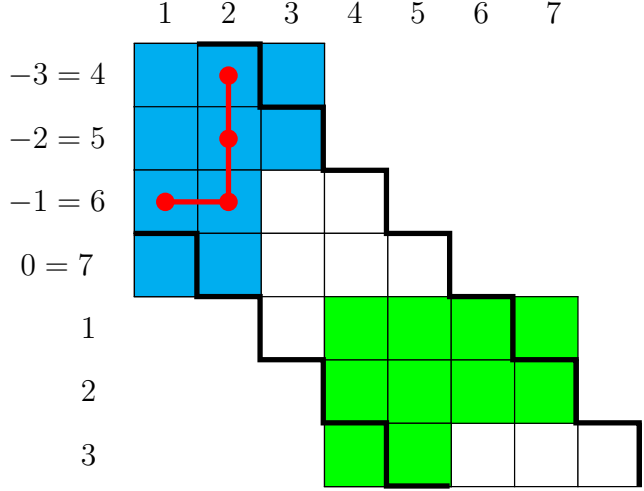
\end{center}

We will not concern all the cells in $\mathbb{Z}/(l+4) \times \mathbb{Z}/(l+4)$. Our focus will be on a subset of the chessboard, which we call the {\em DRH staircase}. In Figure \ref{fig:NNE_SS}, the NNE-DRH staircase is in fact the collection of cells bounded between two {\em staircase paths} drawn in thick black lines. These are zigzag paths which go right one unit and down one unit infinitely alternatively.

The two staircase paths are drawn according to the following rule. The lower one is the zigzag path which cuts out exactly one bottom-leftmost cell $c_{11}$ of the DRH, in a way that $c_{11}$ lies below the staircase path and the rest of the DRH lies above. Similarly, the upper one is the one which cuts out exactly one top-rightmost cell $c_{pq}$ of the DRH. 

Bounded between the two staircase paths is the DRH staircase. The $\lambda$-DRH staircase is $(l+1)$-cell wide, both horizontally and vertically. Currently, we still have not described how to assign a rational function to each cell in the $\lambda$-DRH staircase. We will see that the cells in the DRH staircase are in one-to-one correspondence with the mutable cluster variables of the DRH cluster algebra.

\begin{obs} \label{o:scdiag}
The cell $(i,j) \in \mathbb{Z}/(l+4) \times \mathbb{Z}/(l+4)$ is in the DRH staircase if and only if $i-j \notin \{-1, 0, 1\}$.
\end{obs}

Consider an $(l+4)$-gon $\cP := P_1 P_2 \dots P_{l+4}$. Observation \ref{o:scdiag} implies that there is a bijective correspondence between the cells in the DRH staircase and the directed diagonals of $\cP$: the cell $(i,j)$ corresponds to the directed diagonal $\overrightarrow{P_iP_j}$. If we ignore the direction, each diagonal then corresponds to exactly two different cells in the DRH staircase. This is called the {\em staircase-polygon} correspondence.

\medskip

\subsection{Connection to polygon triangulations and cluster variables} \label{ss:ptandcv}

It is well-known that there is a one-to-one correspondence between the triangulations of $\cP$ and the seeds of the type-A cluster algebra $\cA_{\lambda}$. (See, for example, \cite{Wil14}.) We can therefore assign a cluster variable to every diagonal of the $(l+4)$-gon in a way that for every triangulation, the $l+1$ diagonals in the triangulation are the $l+1$ variables in the corresponding seed. Furthermore, there are many such assignments possible. Given an assignment of variables to a seed (i.e. the diagonals in a triangulation), we can fill in the variables in all other diagonals uniquely in such a way that obeys variable mutations (i.e. diagonal flips).

The mutable part $\cM(\lambda)$ of $Q_{\lambda}$ can be considered inside the DRH staircase as a path consisting of $l+1$ cells, starting from a cell adjacent to the lower boundary of the DRH staircase, going north or east one step at a time, until it reaches a cell adjacent to the upper boundary of the DRH staircase (cf. the definition of $\cM(\lambda)$ in section \ref{ss:iniquiv}). Such path for $\lambda = NNE$ is shown in Figure \ref{fig:NNE_SS} in red. This path is an example of an object which we call a {\em worm} inside the DRH staircase.

\begin{defn}
Let $C$ be a cell inside the $\lambda$-DRH staircase on the lower boundary of the staircase. From $C$, we obtain the other $l(\lambda)$ cells by starting from $C$ and go north or east one step at a time until we reach the upper boundary. A {\em worm} inside the $\lambda$-DRH staircase is defined to be a set of $l+1$ cells obtained from such construction.
\end{defn}

For example, $\cM(\lambda)$ is a worm. We call it the {\em initial worm} of the $\lambda$-DRH staircase.

\begin{obs} \label{o:wormtri}
Every worm corresponds to a triangulation via the staircase-polygon correspondence.
\end{obs}

The $l+1$ cells in the initial worm $\cM(\lambda)$ correspond to $l+1$ diagonals which form a triangulation in $\cP$. As an example, the triangulation which corresponds to the initial worm $\cM(\lambda)$ when $\lambda = NNE$ is shown in Figure \ref{fig:NNEtri}. Recall from the construction of the initial DRH quiver $Q_{\lambda}$ that we associate to every cell in $\cM(\lambda)$ a mutable variable $a_{ij}$. Therefore, we can assign these $l+1$ variables to the $l+1$ diagonals of the seed (triangulation) corresponding to $\cM(\lambda)$. 

\begin{center}
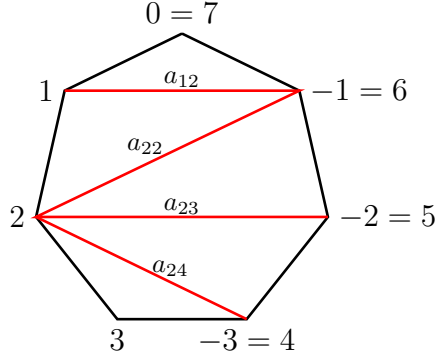
\begin{figure}
\begin{tikzpicture}[scale = 2.0]
\draw[line width = 1.0 pt, black] (0,1) -- (-0.78,0.62) -- (-0.97,-0.22) -- (-0.43,-0.90) -- (0.43,-0.90) -- (0.97,-0.22) -- (0.78,0.62) -- (0,1);
\node[above] at (0,1) {$0=7$};
\node[left] at (-0.78,0.62) {$1$};
\node[left] at (-0.97,-0.22) {$2$};
\node[below] at (-0.43,-0.90) {$3$};
\node[below] at (0.43,-0.90) {$-3 = 4$};
\node[right] at (0.97,-0.22) {$-2 = 5$};
\node[right] at (0.78,0.62) {$-1 = 6$};
\draw[line width= 1.0 pt,red] (-0.78,0.62) -- (0.78,0.62) -- (-0.97,-0.22) -- (0.97,-0.22);
\draw[line width= 1.0 pt,red] (0.43,-0.90) -- (-0.97,-0.22);

\node[above] at (0,0.62-0.05) {\footnotesize $a_{12}$};
\node[left] at (-0.10+0.05,0.20+0.05) {\footnotesize $a_{22}$};
\node[above] at (0,-0.22-0.05) {\footnotesize $a_{23}$};
\node[right] at (-0.27,-0.56) {\footnotesize $a_{24}$};
\end{tikzpicture}
\caption{The heptagon triangulation that corresponds to the initial worm of the NNE-DRH.} \label{fig:NNEtri}
\end{figure}
\end{center}

A sequence of diagonal flips can transform this triangulation to any triangulation while keeping track of how cluster variables are mutated. In particular, the cluster variables associated with any diagonal of $\cP$ can be specified. Since the cells in the staircase correspond to the diagonals of $\cP$, we can associate to every cell of the staircase a mutable cluster variable. Consequently, we obtain all the mutable cluster variables of $\cA_{\lambda}$ in the staircase. Once we fill all the cells in the staircase, we have a complete description of all the cluster variables. The main result of this paper is an explicit description of the variable associated to a given cell in the staircase.

\medskip

\subsection{Climbing worms} \label{ss:climb}
Consider any worm $w = c_1 c_2 \dots c_{l+1}$ in the $\lambda$-DRH staircase. We shall say that the cell $c_i$ is a {\em bend} if $1<i<l+1$ and the centers of the cells $c_{i-1}$, $c_i$, and $c_{i+1}$ are not collinear. If a cell in $w$ is not a bend, then we shall say that it is a {\em non-bend}. Note that the starting vertex $c_1$ and the finishing vertex $c_{l+1}$ are always non-bends.

There is a local operation applied to $w$ that changes $w$ at exactly one cell. This operation is applied to bends, the starting vertex, or the finishing vertex. Suppose that $c_i$ is a bend. It can be either an NE-bend or an EN-bend. An NE-bend is one in which $c_{i-1} c_i$ is vertical while $c_i c_{i+1}$ is horizontal. An EN-bend is one in which $c_{i-1} c_i$ is horizontal while $c_i c_{i+1}$ is vertical. When $c_i$ is an NE-bend, if we put coordinates $c_k = (q_k,r_k)$ for $k = i-1, i, i+1$, we have $q_{i+1} = q_i + 1 = q_{i-1} +1$ and $r_{i+1} = r_i = r_{i-1} - 1$. (Recall how we give coordinates to the cells in \S \ref{ss:constboa}.) In this case, the worm operation at $c_i$ applied to $w$ gives the worm $w'$ with $c_i(q_i,r_i)$ replaced by $\wt{c}_i(q_i+1, r_i+1)$. On the other hand, if $c_i$ is an EN-bend, then the worm operation at $c_i$ applied to $w$ gives the worm $w'$ with $c_i(q_i,r_i)$ replaced by $\wt{c}_i(q_i-1,r_i-1)$.

The worm operation at the starting vertex $c_1(q_1,r_1)$ changes $c_1$ into $\wt{c}_1(q_1+1,r_1+1)$ if $c_1 c_2$ is horizontal, and into $\wt{c}_1(q_1-1,r_1-1)$ if $c_1 c_2$ is vertical. Analogously, the worm operation at the finishing vertex $c_{l+1}(q_{l+1}, r_{l+1})$ changes $c_{l+1}$ into $\wt{c}_{l+1}(q_{l+1}+1,r_{l+1}+1)$ if $c_l c_{l+1}$ is vertical, and into $\wt{c}_{l+1}(q_{l+1}-1,r_{l+1}-1)$ if $c_l c_{l+1}$ is horizontal.

It is not hard to see that given any two worms $w_1, w_2$, there is a sequence of worm operations as described above to transform $w_1$ to $w_2$. In particular, starting with any worm, one can transform it into a worm that contains any cluster variable in $\cA_{\lambda}$.

\begin{lemma} \label{l:woisfd}
Suppose that $w$ is a worm inside the $\lambda$-DRH staircase. Let $w'$ be a worm obtained by applying the worm operation on $w$ at $c_i$. Then, the triangulation of $\cP$ corresponding to $w'$ is obtained by flipping the diagonal corresponding to $c_i$ in the triangulation corresponding to $w$.
\end{lemma}

Before proving Lemma \ref{l:woisfd}, we note that the lemma says that worm operations are diagonal flips. Therefore, we can assign quivers to the worms in the unique manner that respects quiver mutations. We start with the initial worm $\cM(\lambda)$ with the quiver $Q_\lambda$ associated to this particular seed. Then, we do a worm operation on $\cM(\lambda)$ to obtain another worm. We did not know hitherto how we should transform $Q_\lambda$ compatibly with this worm operation. Now, by the lemma, we discover that the right thing to do is to mutate at the vertex we do the worm operation. As a result, the worm operation changes $(w,Q)$ to $(\wt{w},\wt{Q})$, where the changes from $w$ to $\wt{w}$ and from $Q$ to $\wt{Q}$ occur at the same vertex. Therefore, every worm $w$ has an associated quiver $Q_w$. With this new notation, $Q_{\cM(\lambda)} = Q_\lambda$. 

We now prove Lemma \ref{l:woisfd}.

\begin{proof}
For convenience, if $v$ is a worm, let $\Delta(v)$ denote the triangulation of $\cP$ corresponding to $v$ under the staircase-polygon correspondence. If the worm operation occurs on $w$ at an NE-bend $c_i(q_i,r_i)$, then the diagonal corresponding to $c_i$ in $\Delta(w)$ is $P_{q_i} P_{r_i}$. It is a diagonal of the quadrilateral $P_{q_i} P_{q_i + 1} P_{r_i} P_{r_i+1}$ in $\Delta(w)$. Therefore, the diagonal $P_{q_i} P_{r_i}$ flips to $P_{q_i+1} P_{r_i+1}$. This agrees with the worm operation which changes $(q_i,r_i)$ to $(q_i+1,r_i+1)$. The case when $c_i$ is an EN-bend is done similarly.

If the worm operation occurs on $w$ at the starting vertex $c_1(r_1+2,r_1)$ when $c_1 c_2$ is horizontal. We have that $c_2$ has coordinate $(r_1+3, r_1)$. Thus, the diagonal corresponding to $c_1$ is $P_{r_1} P_{r_1+2}$, which is a diagonal of the quadrilateral $P_{r_1} P_{r_1+1} P_{r_1+2} P_{r_1+3}$. Therefore, $P_{r_1} P_{r_1+2}$ flips to $P_{r_1+1} P_{r_1+3}$. This agrees with the worm operation which changes $(r_1+2,r_1)$ to $(r_1+3,r_1+1)$. The case when $c_1 c_2$ is vertical is the reverse flip of this case. The case of worm operation at the finishing vertex is done similarly.
\end{proof}

For instance, starting with the seed of the NNE-cluster algebra shown in Figure \ref{fig:NNEtri}, we may flip the diagonal $P_2P_4$. This produces a new triangulation which has $P_3P_5$ instead of $P_2P_4$, as shown below.

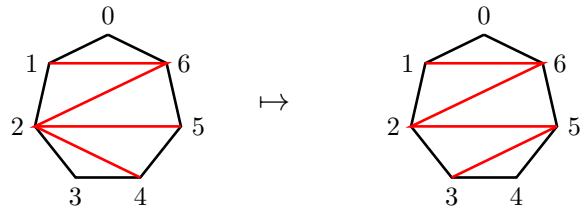
\begin{center}
\begin{tikzpicture}[scale = 1.0]
\begin{scope}[shift={(-2,0)},font = \footnotesize]
\node[above] at (0,1) {$0$};
\node[left] at (-0.78,0.62) {$1$};
\node[left] at (-0.97,-0.22) {$2$};
\node[below] at (-0.43,-0.90) {$3$};
\node[below] at (0.43,-0.90) {$4$};
\node[right] at (0.97,-0.22) {$5$};
\node[right] at (0.78,0.62) {$6$};
\draw[line width = 1.0 pt, black] (0,1) -- (-0.78,0.62) -- (-0.97,-0.22) -- (-0.43,-0.90) -- (0.43,-0.90) -- (0.97,-0.22) -- (0.78,0.62) -- (0,1);
\draw[line width= 1.0 pt,red] (-0.78,0.62) -- (0.78,0.62) -- (-0.97,-0.22) -- (0.97,-0.22);
\draw[line width= 1.0 pt,red] (0.43,-0.90) -- (-0.97,-0.22);
\end{scope}
\begin{scope}
$\mapsto$
\end{scope}
\begin{scope}[shift={(3,0)},font = \footnotesize]
\node[above] at (0,1) {$0$};
\node[left] at (-0.78,0.62) {$1$};
\node[left] at (-0.97,-0.22) {$2$};
\node[below] at (-0.43,-0.90) {$3$};
\node[below] at (0.43,-0.90) {$4$};
\node[right] at (0.97,-0.22) {$5$};
\node[right] at (0.78,0.62) {$6$};
\draw[line width = 1.0 pt, black] (0,1) -- (-0.78,0.62) -- (-0.97,-0.22) -- (-0.43,-0.90) -- (0.43,-0.90) -- (0.97,-0.22) -- (0.78,0.62) -- (0,1);
\draw[line width= 1.0 pt,red] (-0.78,0.62) -- (0.78,0.62) -- (-0.97,-0.22) -- (0.97,-0.22);
\draw[line width= 1.0 pt,red] (-0.43,-0.90) -- (0.97,-0.22);
\end{scope}
\end{tikzpicture}
\end{center}

In the setting of DRH staircase, the transformation produces a new worm by moving $(2,4)$ to $(3,5)$ As shown in Figure \ref{fig:NNEwormop24to35}.

\begin{center}
\begin{figure}
\begin{tikzpicture}
\ytableausetup{notabloids}
\ytableausetup{mathmode, boxsize=2.0em}
\node (n) {\ytableausetup{nosmalltableaux}
\ytableausetup{notabloids}
\ydiagram[*(cyan)]{3,3,2,2}
*[*(white)]{3,3,4,5}};
\node at (-2.58,1.26) {$4$};
\node at (-2.58,0.42) {$5$};
\node at (-2.58,-0.42) {$6$};
\node at (-2.58,-1.26) {$0$};
\node at (-1.56,2.10) {$1$};
\node at (-0.72,2.10) {$2$};
\node at (0.12,2.10) {$3$};
\draw (-2.82+1.26,0.84-1.26) node[circle,minimum size = 7 pt, inner sep = 0 pt, fill = red]{};
\draw (-1.98+1.26,0.84-1.26) node[circle,minimum size = 7 pt, inner sep = 0 pt, fill = red]{};
\draw (-1.98+1.26,1.68-1.26) node[circle,minimum size = 7 pt, inner sep = 0 pt, fill = red]{};
\draw (-1.98+2.10,1.68-1.26) node[circle,minimum size = 7 pt, inner sep = 0 pt, fill = red]{};
\draw (-1.98+1.26,2.52-1.26) node[circle,minimum size = 7 pt, inner sep = 0 pt, fill = red!30]{};
\draw[line width= 2.0 pt,red] (-2.82+1.26,0.84-1.26) -- (-1.98+1.26,0.84-1.26) -- (-1.98+1.26,1.68-1.26) -- (-1.98+2.10,1.68-1.26);
\draw[line width= 2.0 pt,red!30]  (-1.98+1.26,1.68-1.26+0.12) -- (-1.98+1.26,2.52-1.26);
\end{tikzpicture}
\caption{The worm operation on the initial worm of the NNE-DRH at the final vertex} \label{fig:NNEwormop24to35}
\end{figure}
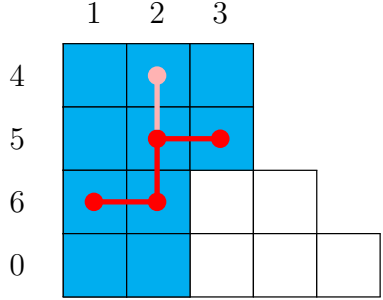
\end{center}

In the construction of the staircase, we saw that the $\lambda$-DRH was initially embedded into the DRH staircase, except for the two cells $c_{11}$ and $c_{pq}$ which lie outside the staircase. We see that in the $l+1$ cells in $\cM(\lambda)$, which are in both the DRH and the DRH staircase, the variables are put into the staircase so that they match the variables inside the DRH by construction.

There are $l+1$ more cells which are in the DRH inside the DRH staircase. In these cells, we associate variables $a_{ij}$ in Section \ref{ss:introboa} when they are considered in the DRH. However, when they are considered inside the DRH staircase, we associate variables to them using quiver mutations. We claim that the variables in the $l+1$ cells from quiver mutations agree with the variables we initially gave in the DRH construction. In particular, this implies that all $a_{ij}$ except $a_{11}$ and $a_{pq}$ are mutable cluster variables.

The result in the following lemma is easy to believe, although its proof may seem tedious. We give the proof anyway to illustrate a preliminary approach to DRH quivers.

\begin{lemma} \label{l:insideboa}
For any cell inside the DRH in the DRH staircase, the variable associated to the cell in the DRH staircase (obtained from quiver mutations) is the same as the variable $a_{ij}$ previously assigned to the DRH in Section \ref{ss:introboa}.
\end{lemma}

\begin{proof}
It suffices to show that for any worm $w = c_1 c_2 \dots c_{l+1}$ inside the $\lambda$-DRH, the mutable variable at the vertex $c_i$ of the quiver $Q_w$ is the corresponding variable associated to $c_i$, as given in Section \ref{ss:introboa}. Suppose that after $k$ worm operations, the worm $w$ is transformed into $\wh{w}$, which contains the desired cell $\wh{c}$. Our strategy of this proof is to write down the quiver explicitly after each $i \le k$ worm operations. We notice that in all the cases the quivers are similar.

Start with the initial worm and its quiver $(\cM(\lambda),Q_\lambda)$, where $\cM(\lambda) = c_1 c_2 \dots c_{l+1}$. Suppose we are given any cell $\wh{c}$ inside the DRH that is not frozen. Then, there is a sequence of worm operations at consecutive vertices to transform the initial worm to a worm containing $\wh{c}$. The sequence of operations is either of the form $i, i+1, i+2, \dots, i+k-1$, or $i, i-1, i-2, \dots, i-k+1$. Here, operation at $j$ means operation at the $j$-th vertex. We consider three cases: when $i= 1$, when $i=l+1$, and when $1< i < l+1$.

\underline{Case 1.} Suppose that $i=1$. This means that the desired cell $\wh{c}$ is on the bottom row. That is, $\wh{c} = c_{k1}$ for some $k>1$. To get from $w$ to a worm $\wh{w}$ that contains $c_{k1}$, we perform the sequence of worm operations: $1,2,\dots, k-1$. The picture below shows the start of the initial worm $\cM(\lambda)$. The worm starts at the cell $a_{12}$, which initially has the variable $a_{12}$. The assumption that the sequence $1,2, \dots, k-1$ takes $w$ to the desired cell $c_{k1}$ shows that at least the first $(k-1)$ steps of the worm are east-steps. The diagonal dotted segment shown denotes the fact that there may be more steps, possibly N or E, after $c_{k2}$.

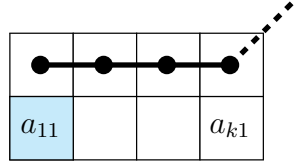
\begin{center}
\begin{tikzpicture}
\ytableausetup{notabloids}
\ytableausetup{mathmode, boxsize=2.0em}
\node (n) {\ytableausetup{nosmalltableaux}
\ytableausetup{notabloids}
\ydiagram[*(cyan!20)]{0,1}
*[*(white)]{4,4}};

\def \a {0.84};

\draw (-1.14,0.42) node[circle,minimum size = 7 pt, inner sep = 0 pt, fill = black]{};
\draw ({-1.14+\a},0.42) node[circle,minimum size = 7 pt, inner sep = 0 pt, fill = black]{};
\draw ({-1.14+2*\a},0.42) node[circle,minimum size = 7 pt, inner sep = 0 pt, fill = black]{};
\draw ({-1.14+3*\a},0.42) node[circle,minimum size = 7 pt, inner sep = 0 pt, fill = black]{};
\draw[line width= 2.0 pt,black] (-1.14,0.42) -- ({-1.14+\a},0.42) -- ({-1.14+2*\a},0.42) -- ({-1.14+3*\a},0.42);
\draw[line width= 2.0 pt, dashed, black] ({-1.14+3*\a},0.42) -- ({-1.14+4*\a},{0.42+\a});

\node at ({-1.14},{0.42-\a}) {$a_{11}$};
\node at ({-1.14+3*\a},{0.42-\a}) {$a_{k1}$};
\end{tikzpicture}
\end{center}

Before the operations, the initial quiver $\cM(\lambda)$ has the following cluster variables in the first $k$ vertices, in order: $a_{12}, a_{22}, \dots, a_{k2}$. For each $i = 1, 2, \dots, k-1$, let
\[
\Delta_i := \begin{vmatrix}
a_{i,2} & a_{i+1,2} \\
a_{i,1} & a_{i+1,1}
\end{vmatrix}
\]
be the $i$-th $2 \times 2$ frozen minor that occurs in the DRH. The local picture of the quiver $Q_\lambda$ near the first $k-1$ vertices is the following.

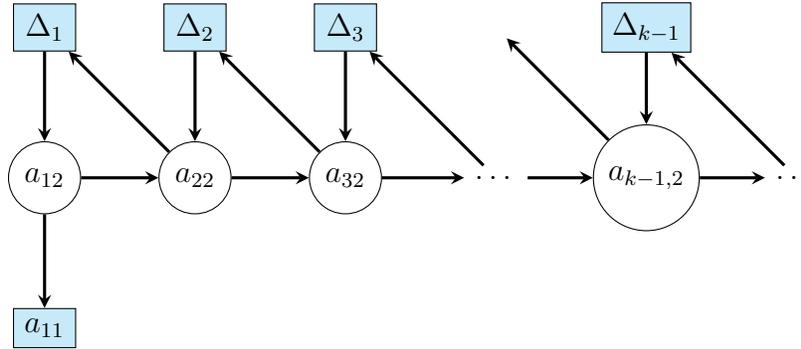
\begin{center}
\begin{tikzpicture}
  [scale=2.0,auto=left,every node/.style={fill=cyan!20,minimum width=2 em, draw = black}]
  \node[rectangle] (D1) at (1,3) {$\Delta_1$};
  \node[rectangle] (D2) at (2,3) {$\Delta_2$};
  \node[rectangle] (D3) at (3,3) {$\Delta_3$};
  \node[draw = none, fill = none] (D4) at (4,3) {};
  \node[rectangle] (D5) at (5,3) {$\Delta_{k-1}$};
  \node[circle, fill = none] (M1) at (1,2)  {$a_{12}$};
  \node[circle, fill = none] (M2) at (2,2)  {$a_{22}$};
  \node[circle, fill = none] (M3) at (3,2)  {$a_{32}$};
  \node[draw = none, fill = none] (M4) at (4,2)  {$\dots$};
  \node[circle, fill = none] (M5) at (5,2)  {$a_{k-1,2}$};
  \node[draw = none, fill = none] (M6) at (6,2)  {$\dots$};
  \node[rectangle] (F1) at (1,1) {$a_{11}$};
  
  \draw[line width = 0.1 em,->,>=stealth] (M1) -- (M2);
  \draw[line width = 0.1 em,->,>=stealth] (M2) -- (M3);
  \draw[line width = 0.1 em,->,>=stealth] (M3) -- (M4);
  \draw[line width = 0.1 em,->,>=stealth] (M4) -- (M5);
  \draw[line width = 0.1 em,->,>=stealth] (M5) -- (M6);
  \draw[line width = 0.1 em,->,>=stealth] (M6) -- (D5);
  \draw[line width = 0.1 em,->,>=stealth] (D5) -- (M5);
  \draw[line width = 0.1 em,->,>=stealth] (M5) -- (D4);
  \draw[line width = 0.1 em,->,>=stealth] (M4) -- (D3);
  \draw[line width = 0.1 em,->,>=stealth] (D3) -- (M3);
  \draw[line width = 0.1 em,->,>=stealth] (M3) -- (D2);
  \draw[line width = 0.1 em,->,>=stealth] (D2) -- (M2);
  \draw[line width = 0.1 em,->,>=stealth] (M2) -- (D1);
  \draw[line width = 0.1 em,->,>=stealth] (D1) -- (M1);
  \draw[line width = 0.1 em,->,>=stealth] (M1) -- (F1);
\end{tikzpicture}
\end{center}

The mutable vertex $a_{12}$ has out-arrows to the mutable vertex $a_{22}$ and the frozen vertex $a_{11}$, and it has an in-arrow from the frozen vertex $\Delta_1$. For $1<i<k$, the mutable vertex $a_{i2}$ has out-arrows to the mutable vertex $a_{i+1,2}$ and to the frozen vertex $\Delta_{i-1}$, and it has in-arrows from the mutable vertex $a_{i-1,2}$ and from the frozen vertex $\Delta_i$.

It is straightforward by induction that after the first $i$ mutations, which mutate $a_{12}, a_{22}, \dots, a_{i2}$, for $i \le k-1$, the quiver becomes as the following picture. The mutable vertex $a_{21}$ has in-arrows from $\Delta_2$ and $a_{11}$, and has out-arrows to $\Delta_1$ and $a_{31}$. For $3 \le j \le i$, the mutable vertex $a_{j1}$ has in-arrows from $a_{j-1,1}$ and $\Delta_j$, and has out-arrows to $a_{j+1,1}$ and $\Delta_{j-1}$. The mutable vertex $a_{i+1,1}$ has in-arrows from $a_{i,1}$ and $a_{i+1,2}$, and has an out-arrow to $\Delta_i$. The mutable vertex $a_{i+1,2}$ has an in-arrow from $\Delta_{i+1}$, and has out-arrows to $a_{i+1,1}$ and $a_{i+2,2}$. The rest of the diagram remains unchanged.

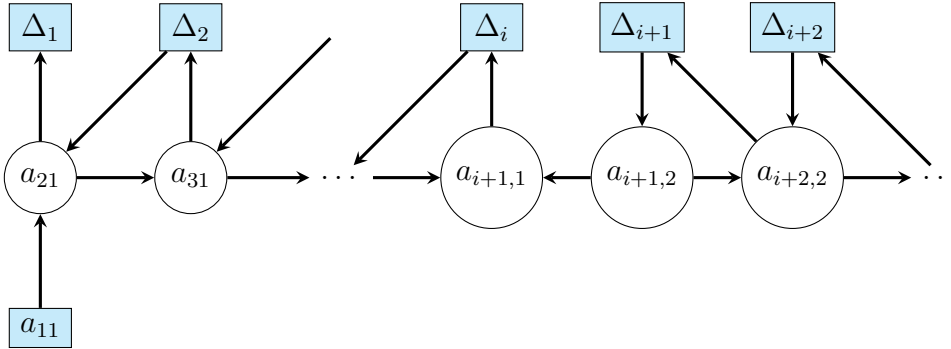
\begin{center}
\begin{tikzpicture}
  [scale=2.0,auto=left,every node/.style={fill=cyan!20,minimum width=2 em, draw = black}]
  \node[rectangle] (D1) at (1,3) {$\Delta_1$};
  \node[rectangle] (D2) at (2,3) {$\Delta_2$};
  \node[draw = none, fill = none] (D3) at (3,3) {};
  \node[rectangle] (D4) at (4,3) {$\Delta_i$};
  \node[rectangle] (D5) at (5,3) {$\Delta_{i+1}$};
  \node[rectangle] (D6) at (6,3) {$\Delta_{i+2}$};
  \node[circle, fill = none] (M1) at (1,2)  {$a_{21}$};
  \node[circle, fill = none] (M2) at (2,2)  {$a_{31}$};
  \node[draw = none, fill = none] (M3) at (3,2)  {$\dots$};
  \node[circle, fill = none] (M4) at (4,2)  {$a_{i+1,1}$};
  \node[circle, fill = none] (M5) at (5,2)  {$a_{i+1,2}$};
  \node[circle, fill = none] (M6) at (6,2)  {$a_{i+2,2}$};
  \node[draw = none, fill = none] (M7) at (7,2)  {$\dots$};
  \node[rectangle] (F1) at (1,1) {$a_{11}$};
  
  \draw[line width = 0.1 em,->,>=stealth] (M1) -- (M2);
  \draw[line width = 0.1 em,->,>=stealth] (M2) -- (M3);
  \draw[line width = 0.1 em,->,>=stealth] (M3) -- (M4);
  \draw[line width = 0.1 em,->,>=stealth] (M5) -- (M4);
  \draw[line width = 0.1 em,->,>=stealth] (M5) -- (M6);
  \draw[line width = 0.1 em,->,>=stealth] (M6) -- (M7);
  \draw[line width = 0.1 em,->,>=stealth] (M7) -- (D6);
  \draw[line width = 0.1 em,->,>=stealth] (D6) -- (M6);
  \draw[line width = 0.1 em,->,>=stealth] (M6) -- (D5);
  \draw[line width = 0.1 em,->,>=stealth] (D5) -- (M5);
  
  \draw[line width = 0.1 em,->,>=stealth] (M4) -- (D4);
  \draw[line width = 0.1 em,->,>=stealth] (D4) -- (M3);
  \draw[line width = 0.1 em,->,>=stealth] (D3) -- (M2);
  \draw[line width = 0.1 em,->,>=stealth] (M2) -- (D2);
  \draw[line width = 0.1 em,->,>=stealth] (D2) -- (M1);
  \draw[line width = 0.1 em,->,>=stealth] (M1) -- (D1);
  \draw[line width = 0.1 em,->,>=stealth] (F1) -- (M1);
\end{tikzpicture}
\end{center}

If the resulting worm has not reached the cell $\wh{c}$ yet, the next vertex to mutate is $a_{i+1,2}$. The variable $a_{i+1,2}$ will change into $a_{i+2,1}$ after mutation, because of the exchange relation
\[
a_{i+1,2} \cdot a_{i+2,1} = \Delta_{i+1} + a_{i+1,1} \cdot a_{i+2,2}.
\]
Therefore, after $k$ steps, the worm will reach the cell $\wh{c} = c_{k1}$, and give the cell the cluster variable $a_{k1}$ as desired.

\underline{Case 2.} Suppose that $i=l+1$. Then, the sequence of worm operations which transforms $w$ into $\wh{w}$ containing $\wh{c}$ is of the form $l+1, l, \dots, l+1-(k-1)$. There are two possibilities, whether the final step of $w$ is horizontal or vertical.

\underline{Case 2.1.} If the final step is horizontal, then the last cell in $\cM(\lambda)$ must be $a_{p,q-1}$. This is because if the last cell were $a_{p-1,q}$ instead, the first worm operation at the last vertex would take the worm outside of the DRH, which violates our assumption. Consequently, the cell $\wh{c}$ must be $c_{p-(k-1),q}$.

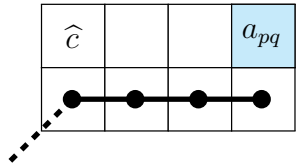
\begin{center}
\begin{tikzpicture}
\ytableausetup{notabloids}
\ytableausetup{mathmode, boxsize=2.0em}
\node (n) {\ytableausetup{nosmalltableaux}
\ytableausetup{notabloids}
\ydiagram[*(cyan!20)]{3+1,0}
*[*(white)]{4,4}};

\def \a {0.84};

\draw (-1.14,{0.42-\a}) node[circle,minimum size = 7 pt, inner sep = 0 pt, fill = black]{};
\draw ({-1.14+\a},{0.42-\a}) node[circle,minimum size = 7 pt, inner sep = 0 pt, fill = black]{};
\draw ({-1.14+2*\a},{0.42-\a}) node[circle,minimum size = 7 pt, inner sep = 0 pt, fill = black]{};
\draw ({-1.14+3*\a},{0.42-\a}) node[circle,minimum size = 7 pt, inner sep = 0 pt, fill = black]{};
\draw[line width= 2.0 pt,black] (-1.14,{0.42-\a}) -- ({-1.14+\a},{0.42-\a}) -- ({-1.14+2*\a},{0.42-\a}) -- ({-1.14+3*\a},{0.42-\a});
\draw[line width= 2.0 pt, dashed, black] (-1.14,{0.42-\a}) -- ({-1.14-\a},{0.42-2*\a});

\node at ({-1.14},{0.42}) {$\wh{c}$};
\node at ({-1.14+3*\a},{0.42}) {$a_{pq}$};
\end{tikzpicture}
\end{center}

Note that the local picture of $Q_\lambda$ at the last $k$ mutable vertices is as the following. Here, $\Delta_i$ denotes the $i$-th frozen $2 \times 2$-minor that occurs in the DRH. To save space, in the diagram below, we may write $a^i_j$ to denote $a_{i,j}$.

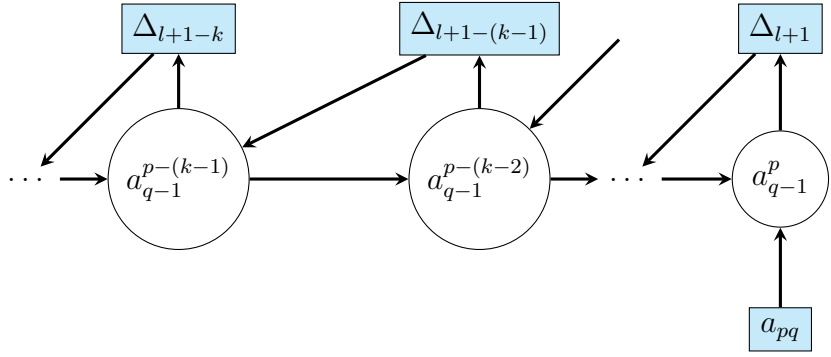
\begin{center}
\begin{tikzpicture}
  [scale=2.0,auto=left,every node/.style={fill=cyan!20,minimum width=2 em, draw = black}]
  \node[draw = none, fill = none] (D1) at (1,3) {};
  \node[rectangle] (D2) at (2,3) {$\Delta_{l+1-k}$};
  \node[rectangle] (D3) at (4,3) {$\Delta_{l+1-(k-1)}$};
  \node[draw = none, fill = none] (D4) at (5,3) {};
  \node[rectangle] (D5) at (6,3) {$\Delta_{l+1}$};
  \node[draw = none, fill = none] (M1) at (1,2)  {$\dots$};
  \node[circle, fill = none] (M2) at (2,2)  {$a^{p-(k-1)}_{q-1}$};
  \node[circle, fill = none] (M3) at (4,2)  {$a^{p-(k-2)}_{q-1}$};
  \node[draw = none, fill = none] (M4) at (5,2)  {$\dots$};
  \node[circle, fill = none] (M5) at (6,2)  {$a^p_{q-1}$};
  \node[rectangle] (F5) at (6,1) {$a_{pq}$};
  
  \draw[line width = 0.1 em,->,>=stealth] (M1) -- (M2);
  \draw[line width = 0.1 em,->,>=stealth] (M2) -- (M3);
  \draw[line width = 0.1 em,->,>=stealth] (M3) -- (M4);
  \draw[line width = 0.1 em,->,>=stealth] (M4) -- (M5);
  \draw[line width = 0.1 em,->,>=stealth] (F5) -- (M5);
  \draw[line width = 0.1 em,->,>=stealth] (M5) -- (D5);
  \draw[line width = 0.1 em,->,>=stealth] (D5) -- (M4);
  
  \draw[line width = 0.1 em,->,>=stealth] (D4) -- (M3);
  \draw[line width = 0.1 em,->,>=stealth] (M3) -- (D3);
  \draw[line width = 0.1 em,->,>=stealth] (D3) -- (M2);
  \draw[line width = 0.1 em,->,>=stealth] (M2) -- (D2);
  \draw[line width = 0.1 em,->,>=stealth] (D2) -- (M1);
\end{tikzpicture}
\end{center}

We observe the similarity of this quiver with the one is Case 1. The rest of the argument is by induction analogous to the previous case.

Similarly, for \underline{Case 2.2}, when the final step is vertical, we use an analogous induction argument to show that the variable to give to $\wh{c}$ is the same as the variable $a_{ij}$ at the cell in the DRH.

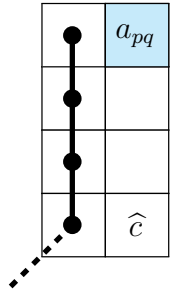
\begin{center}
\begin{tikzpicture}
\ytableausetup{notabloids}
\ytableausetup{mathmode, boxsize=2.0em}
\node (n) {\ytableausetup{nosmalltableaux}
\ytableausetup{notabloids}
\ydiagram[*(cyan!20)]{1+1,0,0,0}
*[*(white)]{2,2,2,2}};

\def \a {0.84};

\draw ({-1.14+\a},{0.42+\a}) node[circle,minimum size = 7 pt, inner sep = 0 pt, fill = black]{};
\draw ({-1.14+\a},{0.42}) node[circle,minimum size = 7 pt, inner sep = 0 pt, fill = black]{};
\draw ({-1.14+\a},{0.42-\a}) node[circle,minimum size = 7 pt, inner sep = 0 pt, fill = black]{};
\draw ({-1.14+\a},{0.42-2*\a}) node[circle,minimum size = 7 pt, inner sep = 0 pt, fill = black]{};
\draw[line width= 2.0 pt,black] ({-1.14+\a},{0.42+\a}) -- ({-1.14+\a},{0.42}) -- ({-1.14+\a},{0.42-\a}) -- ({-1.14+\a},{0.42-2*\a});
\draw[line width= 2.0 pt, dashed, black] ({-1.14+\a},{0.42-2*\a}) -- ({-1.14},{0.42-3*\a});

\node at ({-1.14+2*\a},{0.42-2*\a}) {$\wh{c}$};
\node at ({-1.14+2*\a},{0.42+\a}) {$a_{pq}$};
\end{tikzpicture}
\end{center}

\underline{Case 3.} Suppose that $1<i<l+1$. The sequence of worm operations can either take the form $i, i+1, \dots, i+k-1$ or $i, i-1, \dots, i-k+1$. In each of the two cases, the bend at $c_i$ can either be an NE-bend or an EN-bend. Again, we will see that the four cases are similar. We will explicitly show one case: when the sequence takes the form $i, i+1, \dots, i+k-1$ and the bend at the $i$-th vertex is an NE-bend.
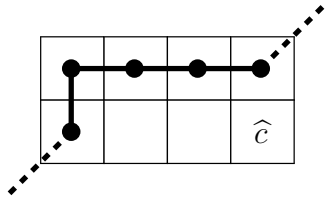
\begin{center}
\begin{tikzpicture}
\ytableausetup{notabloids}
\ytableausetup{mathmode, boxsize=2.0em}
\node (n) {\ytableausetup{nosmalltableaux}
\ytableausetup{notabloids}
\ydiagram[*(cyan!20)]{0,0}
*[*(white)]{4,4}};

\def \a {0.84};

\draw (-1.14,{0.42-\a}) node[circle,minimum size = 7 pt, inner sep = 0 pt, fill = black]{};
\draw (-1.14,0.42) node[circle,minimum size = 7 pt, inner sep = 0 pt, fill = black]{};
\draw ({-1.14+\a},0.42) node[circle,minimum size = 7 pt, inner sep = 0 pt, fill = black]{};
\draw ({-1.14+2*\a},0.42) node[circle,minimum size = 7 pt, inner sep = 0 pt, fill = black]{};
\draw ({-1.14+3*\a},0.42) node[circle,minimum size = 7 pt, inner sep = 0 pt, fill = black]{};
\draw[line width= 2.0 pt,black] (-1.14,{0.42-\a}) -- (-1.14,0.42) -- ({-1.14+\a},0.42) -- ({-1.14+2*\a},0.42) -- ({-1.14+3*\a},0.42);
\draw[line width= 2.0 pt, dashed, black] ({-1.14+3*\a},0.42) -- ({-1.14+4*\a},{0.42+\a});
\draw[line width= 2.0 pt, dashed, black] ({-1.14},{0.42-\a}) -- ({-1.14-\a},{0.42-2*\a});

\node at ({-1.14+3*\a},{0.42-\a}) {$\wh{c}$};
\end{tikzpicture}
\end{center}

Let the cell at the bend be $c_{r,s}$. Thus, $\wh{c}$ is $c_{r+k,s-1}$. The initial quiver can be drawn locally as follows. As before, we may write $a^i_j$ to denote $a_{i,j}$.

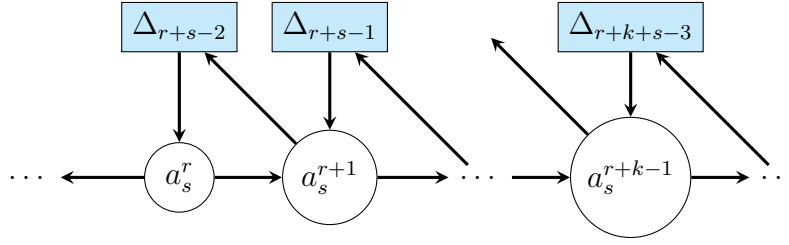
\begin{center}
\begin{tikzpicture}
  [scale=2.0,auto=left,every node/.style={fill=cyan!20,minimum width=2 em, draw = black}]
  \node[draw = none, fill = none] (D1) at (1,3) {};
  \node[rectangle] (D2) at (2,3) {$\Delta_{r+s-2}$};
  \node[rectangle] (D3) at (3,3) {$\Delta_{r+s-1}$};
  \node[draw = none, fill = none] (D4) at (4,3) {};
  \node[rectangle] (D5) at (5,3) {$\Delta_{r+k+s-3}$};
  \node[draw = none, fill = none] (D6) at (6,3) {};
  \node[draw = none, fill = none] (M1) at (1,2)  {$\dots$};
  \node[circle, fill = none] (M2) at (2,2)  {$a^r_s$};
  \node[circle, fill = none] (M3) at (3,2)  {$a^{r+1}_s$};
  \node[draw = none, fill = none] (M4) at (4,2)  {$\dots$};
  \node[circle, fill = none] (M5) at (5,2)  {$a^{r+k-1}_s$};
  \node[draw = none, fill = none] (M6) at (6,2)  {$\dots$};
  
  \draw[line width = 0.1 em,->,>=stealth] (M2) -- (M1);
  \draw[line width = 0.1 em,->,>=stealth] (M2) -- (M3);
  \draw[line width = 0.1 em,->,>=stealth] (M3) -- (M4);
  \draw[line width = 0.1 em,->,>=stealth] (M4) -- (M5);
  \draw[line width = 0.1 em,->,>=stealth] (M5) -- (M6);
  \draw[line width = 0.1 em,->,>=stealth] (M6) -- (D5);
  \draw[line width = 0.1 em,->,>=stealth] (D5) -- (M5);
  \draw[line width = 0.1 em,->,>=stealth] (M5) -- (D4);
  \draw[line width = 0.1 em,->,>=stealth] (M4) -- (D3);
  \draw[line width = 0.1 em,->,>=stealth] (D3) -- (M3);
  \draw[line width = 0.1 em,->,>=stealth] (M3) -- (D2);
  \draw[line width = 0.1 em,->,>=stealth] (D2) -- (M2);
\end{tikzpicture}
\end{center}

We start mutating at the cell with $a_{r,s}$ and continue to the right. For $1 \le i \le k$, after $i$ mutations the quiver becomes as follows. 

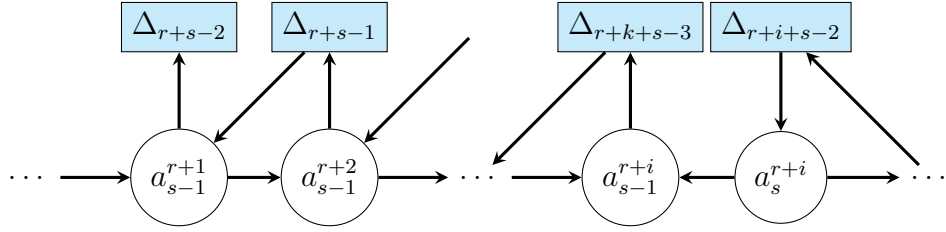
\begin{center}
\begin{tikzpicture}
  [scale=2.0,auto=left,every node/.style={fill=cyan!20,minimum width=2 em, draw = black}]
  \node[draw = none, fill = none] (D1) at (1,3) {};
  \node[rectangle] (D2) at (2,3) {$\Delta_{r+s-2}$};
  \node[rectangle] (D3) at (3,3) {$\Delta_{r+s-1}$};
  \node[draw = none, fill = none] (D4) at (4,3) {};
  \node[rectangle] (D5) at (5,3) {$\Delta_{r+k+s-3}$};
  \node[rectangle] (D6) at (6,3) {$\Delta_{r+i+s-2}$};
  \node[draw = none, fill = none] (D7) at (7,3) {};
  \node[draw = none, fill = none] (M1) at (1,2)  {$\dots$};
  \node[circle, fill = none] (M2) at (2,2)  {$a^{r+1}_{s-1}$};
  \node[circle, fill = none] (M3) at (3,2)  {$a^{r+2}_{s-1}$};
  \node[draw = none, fill = none] (M4) at (4,2)  {$\dots$};
  \node[circle, fill = none] (M5) at (5,2)  {$a^{r+i}_{s-1}$};
  \node[circle, fill = none] (M6) at (6,2)  {$a^{r+i}_s$};
  \node[draw = none, fill = none] (M7) at (7,2)  {$\dots$};
  
  \draw[line width = 0.1 em,->,>=stealth] (M1) -- (M2);
  \draw[line width = 0.1 em,->,>=stealth] (M2) -- (M3);
  \draw[line width = 0.1 em,->,>=stealth] (M3) -- (M4);
  \draw[line width = 0.1 em,->,>=stealth] (M4) -- (M5);
  \draw[line width = 0.1 em,->,>=stealth] (M6) -- (M5);
  \draw[line width = 0.1 em,->,>=stealth] (M6) -- (M7);
  
  \draw[line width = 0.1 em,->,>=stealth] (M7) -- (D6);
  \draw[line width = 0.1 em,->,>=stealth] (D6) -- (M6);
  \draw[line width = 0.1 em,->,>=stealth] (M5) -- (D5);
  \draw[line width = 0.1 em,->,>=stealth] (D5) -- (M4);
  \draw[line width = 0.1 em,->,>=stealth] (D4) -- (M3);
  \draw[line width = 0.1 em,->,>=stealth] (M3) -- (D3);
  \draw[line width = 0.1 em,->,>=stealth] (D3) -- (M2);
  \draw[line width = 0.1 em,->,>=stealth] (M2) -- (D2);
\end{tikzpicture}
\end{center}

Note that if $i = k$, the arrows at $a_{r+i,s}$ might be different from what is shown, but the quiver gives $a_{r+k,s}$ to the cell $\wh{c} = c_{r+k,s}$ as claimed. If $i <k$, then the next vertex to mutate is $a_{r+i,s}$. After the $i+1$-st mutation, the cluster variable $a_{r+i,s}$ will be changed into $a_{r+i+1,s-1}$ because of the exchange relation
\[
a_{r+i,s} \cdot a_{r+i+1,s-1} = a_{r+i,s-1} \cdot a_{r+i+1,s} + \Delta_{r+i+s-2}.
\]
This finishes the proof.
\end{proof}

In the proof of the previous lemma, we also observe that the quiver for worms inside the DRH behaves nicely. In particular, a mutable vertex only has frozen arrows from frozen vertices that are not very far from it: if there is an arrow between the $i$-th mutable vertex and $\Delta_j$, then we expect $i$ and $j$ to be close. Later, when we allow the worm to leave the DRH, this will no longer be the case. The frozen arrows can come from frozen vertices very far from the mutable vertex we consider.

Given the $\lambda$-DRH staircase, we have now filled variables into the cells which are also in the DRH. In Figure \ref{fig:NNE_SS}, these cells are shown in cyan. Recall from the staircase-polygon correspondence that a certain cluster variable (diagonal) corresponds to two cells in the staircase. Namely, if we put variable $\alpha$ into the cell with staircase coordinate $(i,j)$, then we also put the same variable into the cell $(j,i)$. This gives us the {\em transpose DRH}: the cells $(i,j)$ in the DRH staircase for which $(j,i)$ is in the DRH. An example of the transpose NNE-DRH is shown in green in Figure \ref{fig:NNE_SS}.

We can think of worm operations as a worm climbing down the DRH staircase. From the red worm in Figure \ref{fig:NNEwormop24to35}, if we continue to mutate at the vertex at $(2,5)$, the red point at $(2,5)$ will move to $(3,6)$. 

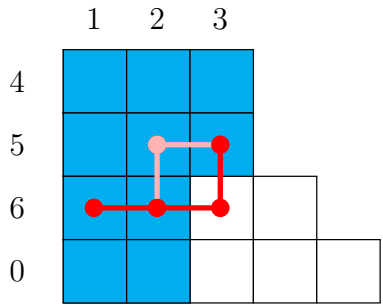
\begin{center}
\begin{tikzpicture}
\ytableausetup{notabloids}
\ytableausetup{mathmode, boxsize=2.0em}
\node (n) {\ytableausetup{nosmalltableaux}
\ytableausetup{notabloids}
\ydiagram[*(cyan)]{3,3,2,2}
*[*(white)]{3,3,4,5}};
\node at (-2.58,1.26) {$4$};
\node at (-2.58,0.42) {$5$};
\node at (-2.58,-0.42) {$6$};
\node at (-2.58,-1.26) {$0$};
\node at (-1.56,2.10) {$1$};
\node at (-0.72,2.10) {$2$};
\node at (0.12,2.10) {$3$};
\draw (-2.82+1.26,0.84-1.26) node[circle,minimum size = 7 pt, inner sep = 0 pt, fill = red]{};
\draw (-1.98+1.26,0.84-1.26) node[circle,minimum size = 7 pt, inner sep = 0 pt, fill = red]{};
\draw (-1.98+1.26,1.68-1.26) node[circle,minimum size = 7 pt, inner sep = 0 pt, fill = red]{};
\draw (-1.98+2.10,1.68-1.26) node[circle,minimum size = 7 pt, inner sep = 0 pt, fill = red]{};
\draw (-1.98+1.26+0.84,1.68-1.26-0.84) node[circle,minimum size = 7 pt, inner sep = 0 pt, fill = red]{};
\draw (-1.98+1.26,2.52-1.26-0.84) node[circle,minimum size = 7 pt, inner sep = 0 pt, fill = red!30]{};
\draw[line width= 2.0 pt,red] (-2.82+1.26,0.84-1.26) -- (-1.98+1.26,0.84-1.26) -- (-1.98+1.26+0.84,1.68-1.26-0.84) -- (-1.98+2.10,1.68-1.26);
\draw[line width= 2.0 pt,red!30]  (-1.98+1.26,1.68-1.26+0.12-0.84) -- (-1.98+1.26,2.52-1.26-0.84);
\draw[line width= 2.0 pt,red!30]  (-1.98+1.26+0.84-0.12,2.52-1.26-0.84)-- (-1.98+1.26,2.52-1.26-0.84);
\end{tikzpicture}
\end{center}

By only mutating at the bends, the starting vertex, and the finishing vertex of worms, the resulting $l+1$ cells in the new seed will always form a worm. As we mutate in this way, we will not -- except when the DRH is small -- obtain all the seeds of the cluster algebra, but as the worm climbs down the staircase, it passes all the different cells. Thus, by considering all the cluster variables of climbing worms, we obtain all the cluster variables for the algebra.

\bigskip

\section{DRH subskeleta}

Let $\lambda$ be a North-East lattice path. By a {\em subskeleton} of $\lambda$, we mean a non-empty connected subword of $\lambda$ together with the data of its location in $\lambda$. Formally, it can be thought of as a triple $(\lambda,r,s)$, where $1 \le r \le s \le l(\lambda)$. If $\lambda = w_1 \dots w_n \in \{N,E\}^n$, then a subskeleton of the $\lambda$-DRH has the form $w_r w_{r+1} \dots w_s$ for some $1 \le r \le s \le n$.

It is convenient to attach with each letter in $\lambda$ its location in the word. The letter $W \in \{N,E\}$ which is the $i$-th letter from the left may be written as $W^i$. For example, instead of writing $\lambda = NENE$, we may write $N^1E^2N^3E^4$. Indeed, there is no addition or loss of information when writing in this way. Labeling so simply makes it easier to keep track of which N (or E) is which.

This convention makes discussions of subskeleta easy. For instance, we have that the subskeleta of $N^1E^2N^3E^4$ are $N^1$, $E^2$, $N^3$, $E^4$, $N^1E^2$, $E^2N^3$, $N^3E^4$, $N^1E^2N^3$, $E^2N^3E^4$, $N^1E^2N^3E^4$. The subskeleta $N^1E^2$ and $N^3E^4$ are considered different.

The number of subskeleta of $\lambda$ is $\frac{l(l+1)}{2}$, which is the same as the number of ``not-yet-filled'' cells in the $\lambda$-DRH staircase between the original DRH and the transpose DRH (the white cells between the cyan and the green in Figure \ref{fig:NNE_SS}). The equality can be understood combinatorially. We will establish a combinatorial bijective correspondence between the subskeleta and the not-yet-filled cells.

For convenience, let $\cW_{\lambda}$ denote the finite board containing the $\frac{l(l+1)}{2}$ cells in the DRH staircase which are in between the original DRH and the transpose DRH. For example, in Figure \ref{fig:NENEEN_SS}, when $\lambda = NENEEN$, $\cW_{\lambda}$ is the finite board of twenty-one white cells in the middle (while the original DRH is in cyan and the transpose DRH is in green). Recall that the row and column indices are taken modulo $l+4 = 10$ in this case.

When there is no confusion, we may use $\cW_{\lambda}$ to also mean the set of $\frac{l(l+1)}{2}$ cells, instead of the board itself. As we mentioned in the previous section, we are interested in the subskeleta of the DRH. We shall denote by $S_{\lambda}$ the set of $\frac{l(l+1)}{2}$ subskeleta of $\lambda$. The goal of this section is to establish the combinatorial bijective map
\[
\Subsk: \cW_{\lambda} \rightarrow S_{\lambda}
\]
assigning a subskeleton to each cell of $\cW_{\lambda}$.

Note that the ``northwest'' boundary of $\cW_{\lambda}$ between the original DRH and $\cW_\lambda$ is a $\lambda$-lattice path, while the ``southeast'' boundary of $\cW_{\lambda}$ between the transpose DRH and $\cW_\lambda$ is the transpose $\lambda^t$. Here, the transpose of the word $w_1 w_2 \dots w_n$ is the word $w_n^t w_{n-1}^t \dots w_1^t$, where $E^t = N$ and $N^t = E$. Thus, we can label each of the $l$ unit segments on the northwest boundary by $E^i$ or $N^i$ following its corresponding letter in $\lambda$. Similarly, we can label each of the $l$ unit segments on the southeast boundary by $E^i$ or $N^i$ with its corresponding letter in $\lambda$. For convenience in this section, let's denote by $B_{NW}$ the collection of $l$ unit segments on the northwest boundary, and by $B_{SE}$ the collection of $l$ unit segments on the southeast boundary.

Observe that each row (or column) of the staircase that intersects (i.e. share at least one cell) with $\cW_\lambda$ contains exactly one element of $B_{NW} \cup B_{SE}$. Therefore, we may label the row (or column) by the label $E^i$ or $N^i$ of the unit segment in $B_{NW} \cup B_{SE}$ that the row (or column) contains. In total, there will be $l$ labeled rows and $l$ labeled columns. By considering the $l \times l$ chessboard obtained from these labeled rows and columns, we embed $\cW_\lambda$ into an $l \times l$ chessboard. Each cell in the $l \times l$ chessboard can be identified with an ordered-pair $(W^r,W^s)$ where $r,s \in 1,2,\dots, l$ with $W^r$ being the row label and $W^s$ being the column label. The letter $W$ here denotes a symbol that is either $N$ or $E$. The label $W^i$ is either $N^i$ or $E^i$, depending on whether the $i$-th letter in $\lambda$ is $N$ or $E$.

For example, the row and column labels for $\cW_{NNE}$ is shown in Figure \ref{fig:rclabelsNNE}. The northwest and the southeast boundaries are shown in think broken segments. The ordered-pair identification of the top-left cell is $(N^2, E^3)$. Given a cell $(W^r, W^s)$ in the $l \times l$ chessboard, how do we know whether it is in $\cW_{\lambda}$? The following lemma answers this question.

\begin{center}
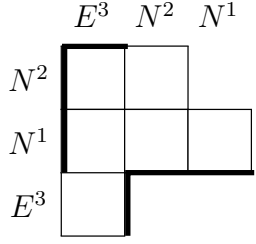
\begin{figure}
\begin{tikzpicture}
\ytableausetup{notabloids}
\ytableausetup{mathmode, boxsize=2.0em}
\node (n) {\ytableausetup{nosmalltableaux}
\ytableausetup{notabloids}
\ydiagram[*(white)]{2,3,1}};

\def \a {0.84};

\node at (-1.56,-0.84) {$E^3$};
\node at (-1.56,{-0.84+\a}) {$N^1$};
\node at (-1.56,{-0.84+2*\a}) {$N^2$};

\node at ({-1.56+\a},{-0.84+3*\a}) {$E^3$};
\node at ({-1.56+2*\a},{-0.84+3*\a}) {$N^2$};
\node at ({-1.56+3*\a},{-0.84+3*\a}) {$N^1$};

\draw[line width= 2.0 pt,black] (-1.08,-0.42) -- (-1.08,{-0.42+2*\a}) -- ({-1.08+\a},{-0.42+2*\a});

\draw[line width= 2.0 pt,black] ({-1.08+\a},{-0.42-\a}) -- ({-1.08+\a},{-0.42}) -- ({-1.08+3*\a},{-0.42});
\end{tikzpicture}

\caption{The row and column labels for $\cW_{NNE}$.} \label{fig:rclabelsNNE}
\end{figure}
\end{center}

\begin{lemma} \label{l:Wliffrles}
In the $l \times l$ chessboard with row and column labels as described, the cell $(W^r,W^s)$ belongs to $\cW_\lambda$ if and only if $r \le s$.
\end{lemma}

\begin{proof}
There exists a unique vertical line in the DRH staircase which separates the DRH and the transpose DRH. There also exists a horizontal line in the DRH staircase which separates the DRH and the transpose DRH. These two lines form an XY-axis, which we call the {\em DRH axis}. In Figure \ref{fig:NENEEN_SS}, the axis is drawn in red. Suppose that each cell has one unit sidelength. Let $\nu$ and $\ve$, respectively, be the numbers of N's and E's in $\lambda$. Considered in this XY-coordinate system, the northwest boundary then starts at $(-\ve,0)$ and ends at $(0,\nu)$. The southeast boundary starts at $(\nu,0)$ and ends at $(0,-\ve)$. The part of the upper staircase path that bounds $\cW_\lambda$ starts at $(0,\nu)$, goes right and then down alternatively, and ends at $(\nu,0)$. On the other hand, the part of the lower staircase path that bounds $\cW_\lambda$ starts at $(-\ve,0)$, goes down first and then right alternatingly, and ends at $(0,-\ve)$.

We make a quick note about the notations that $(u,v)$ refers to a point in the XY-plane if $u$ and $v$ are real numbers, while $(W^i,W^j)$ refers to a cell in $\cW_\lambda$ if $W^i$ is a row label and $W^j$ is a column label.

($\Rightarrow$) We will show that a cell $(W^r, W^s)$ in $\cW_\lambda$ satisfies $r \le s$. When considered with the DRH axis, $(W^r,W^s)$ can be in one of the four quadrants. If $(W^r,W^s)$ is in the first quadrant, then $(W^r,W^s) = (N^r,N^s)$. If it is in the second quadrant, then $(W^r,W^s) = (N^r,E^s)$. If it is in the third quadrant, then $(W^r, W^s) = (E^r,E^s)$. Lastly, if it is in the fourth quadrant, then $(W^r,W^s) = (E^r,N^s)$.

\underline{Case 1.} Suppose that $(W^r, W^s) = (N^r, N^s)$ is in the first quadrant. In XY-coordinate, the cell can be written as $[u,u+1] \times [v,v+1]$ for some $u,v \in \mathbb{Z}_{\ge 0}$ such that $u+v \le \nu -1$. By construction, the cell is in the row $N^r$, where $N^r$ is the $(v+1)$-st $N$ in $\lambda$. Analogously, the column label $N^s$ is the $(\nu - u)$-th $N$ in $\lambda$. Because $v+1 \le \nu -u$, we see that $N^r$ occurs before $N^s$ in $\lambda$, which shows that $r \le s$.

\underline{Case 2.} Suppose that $(W^r, W^s) = (N^r, E^s)$ is in the second quadrant. If we draw a horizontal line from the cell to the left, the line hits the northwest boundary of $\cW_\lambda$ at the segment $N^r$ in $B_{NW}$. If we draw a vertical line from the cell upward, the line hits the northwest boundary at the segment $E^s$ in $B_{NW}$. It is evident that $N^r$ comes before $E^s$ is $\lambda$. Therefore, $r \le s$.

We note that \underline{Case 3} when the cell is in the third quadrant is analogous to Case 1, and \underline{Case 4} when the cell is in the fourth quadrant is analogous to Case 2. Therefore, and cell $(W^r,W^s) \in \cW_\lambda$ satisfies $r \le s$.

($\Leftarrow$) The first part of this proof shows that if $(W^r, W^s) \in \cW_\lambda$, then $r \le s$. Conversely, we know that there are exactly $\frac{l(l+1)}{2}$ cells in $\cW_\lambda$, while there are also exactly $\frac{l(l+1)}{2}$ ordered pairs $(r,s)$ such that $1 \le r \le s \le l$. Therefore, for any such pair $(r,s)$, we must have that $(W^r,W^s) \in \cW_\lambda$.
\end{proof}

\begin{center}
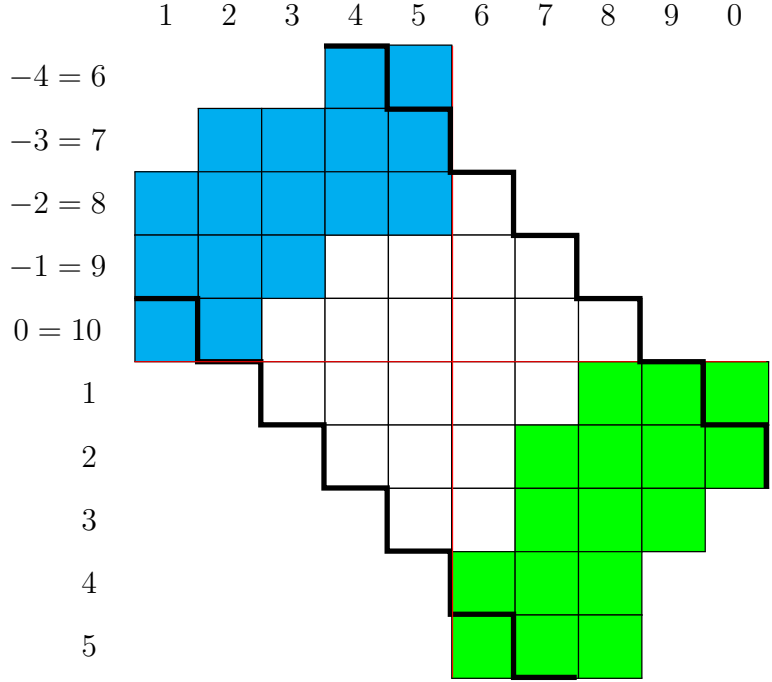
\begin{figure}
\begin{tikzpicture}
\ytableausetup{notabloids}
\ytableausetup{mathmode, boxsize=2.0em}
\node (n) {\ytableausetup{nosmalltableaux}
\ytableausetup{notabloids}
\ydiagram[*(cyan)]{3+2,1+4,0+5,0+3,0+2}
*[*(green)]{0,0,0,0,0,7+3,6+4,6+3,5+3,5+3}
*[*(white)]{0,0,5+1,3+4,2+6,2+5,3+3,4+2}};

\def \a {0.84};

\begin{scope}[shift={(-0.84,+0.42)}]
\draw[line width= 2.0 pt,black] (-3.24,0.42) -- (-2.40,0.42) -- (-2.40,-0.42) -- (-1.56,-0.42) -- (-1.56,-1.26) -- (-0.72,-1.26) -- (-0.72,-2.10) -- (0.12,-2.10) -- (0.12,-2.94) -- (0.96,-2.94) -- (0.96,-3.78) -- (1.80,-3.78) -- (1.80,-4.62) -- (2.64,-4.62);

\draw[red] ({-1.56-2*\a},-0.42) -- ({-1.56+8*\a},-0.42);
\draw[red] ({-1.53+3*\a},{-0.42+5*\a}) -- ({-1.53+3*\a},{-0.42-5*\a});
\end{scope}

\begin{scope}[shift = {(+0.84,+1.26)}]
\draw[line width= 2.0 pt,black] (-2.40,2.94) -- (-1.56,2.94) -- (-3.24+1.68,0.42+1.68) -- (-2.40+1.68,0.42+1.68) -- (-2.40+1.68,-0.42+1.68) -- (-1.56+1.68,-0.42+1.68) -- (-1.56+1.68,-1.26+1.68) -- (-0.72+1.68,-1.26+1.68) -- (-0.72+1.68,-2.10+1.68) -- (0.12+1.68,-2.10+1.68) -- (0.12+1.68,-2.94+1.68) -- (0.96+1.68,-2.94+1.68) -- (0.96+1.68,-2.10) -- (3.48,-2.10) -- (3.48,-2.94);
\end{scope}

\begin{scope}[shift = {(-0.84,+0.42)}]
\node at (-4.26,3.36) {$-4 = 6$};
\node at (-4.26,2.52) {$-3 = 7$};
\node at (-4.26,1.68) {$-2 = 8$};
\node at (-4.26,0.84) {$-1 = 9$};
\node at (-4.26,0) {$0 = 10$};
\node at (-3.84,-0.84) {$1$};
\node at (-3.84,-1.68) {$2$};
\node at (-3.84,-2.52) {$3$};
\node at (-3.84,-3.36) {$4$};
\node at (-3.84,-4.20) {$5$};
\end{scope}

\begin{scope}[shift = {(-0.84,+1.26)}]
\node at (-2.82,3.36) {$1$};
\node at (-1.98,3.36) {$2$};
\node at (-1.14,3.36) {$3$};
\node at (-0.30,3.36) {$4$};
\node at (0.54,3.36) {$5$};
\node at (1.38,3.36) {$6$};
\node at (2.22,3.36) {$7$};
\node at (3.06,3.36) {$8$};
\node at (3.90,3.36) {$9$};
\node at (4.74,3.36) {$0$};
\end{scope}
\end{tikzpicture}

\caption{The NENEEN-DRH staircase} \label{fig:NENEEN_SS}
\end{figure}
\end{center}

From Lemma \ref{l:Wliffrles}, we now define the desired map $\Subsk: \cW_\lambda \rightarrow S_\lambda$. We have established a bijective correspondence between the cells in $\cW_\lambda$ and $(r,s) \in \mathbb{Z} \times \mathbb{Z}$ such that $1 \le r \le s \le l$. Given a cell $\gamma \in \cW_\lambda$, we let $(r_\gamma, s_\gamma)$ denote its corresponding ordered pair. Then, we define
\[
\Subsk(\gamma) := W^{r_\gamma} W^{r_\gamma+1} \cdots W^{s_\gamma}
\]
where $\lambda = W^1 W^2 \cdots W^l$. For example, the twenty-one subskeleta of $\lambda$ are filled in the twenty-one cells of $\cW_\lambda$ in Figure \ref{fig:W_NENEEN}.

\begin{center}
\begin{figure}
\begin{tikzpicture}
\ytableausetup{notabloids}
\ytableausetup{mathmode, boxsize=5.0em}
\node (n) {\ytableausetup{nosmalltableaux}
\ytableausetup{notabloids}
\ydiagram[*(white)]{3+1,1+4,0+6,0+5,1+3,2+2}};

\begin{scope}[scale = 2.1, shift = {(-2.45,-2.5)}]
\tiny
\node at (0,2) {$E^2$};
\node at (0,3) {$N^1 E^2$};
\node at (1,1) {$E^4$};
\node at (1,2) {$E^2N^3E^4$};
\node at (1,3) {$N^1E^2N^3E^4$};
\node at (1,4) {$N^3E^4$};
\node at (2,0) {$E^5$};
\node at (2,1) {$E^4E^5$};
\node at (2,2) {$E^2N^3E^4E^5$};
\node at (2,3) {$NENEE$};
\node at (2,4) {$N^3E^4E^5$};
\node at (3,4) {$N^3E^4E^5N^6$};
\node at (3,5) {$N^6$};
\node at (3,3) {$NENEEN$};
\node at (3,2) {$ENEEN$};
\node at (3,1) {$E^4E^5N^6$};
\node at (3,0) {$E^5N^6$};
\node at (4,2) {$E^2N^3$};
\node at (4,3) {$N^1E^2N^3$};
\node at (4,4) {$N^3$};
\node at (5,3) {$N^1$};
\normalsize
\end{scope}

\begin{scope}[scale = 2.1, shift = {(-2.92,0.02)}]
\draw[line width= 3.0 pt,black] (0,0) -- (0,-1) -- (1,-1) -- (1,-2) -- (2,-2) -- (2,-3) -- (3,-3);
\end{scope}

\begin{scope}[scale = 2.1, shift = {(0.06,3.98)}]
\draw[line width= 3.0 pt,black] (0,-1) -- (1,-1) -- (1,-2) -- (2,-2) -- (2,-3) -- (3,-3) -- (3,-4);
\end{scope}

\end{tikzpicture}
\caption{$\Subsk \left(\cW_{NENEEN}\right)$} \label{fig:W_NENEEN}
\end{figure}
\end{center}

Before moving on, let us prove a nice corollary of Lemma \ref{l:Wliffrles}. The result is not difficult to prove by carefully checking all possible cases, but here we will use bookkeeping techniques we have developed to establish the result swiftly.

\begin{cor} \label{c:notbothD&Dt}
Let $\lambda$ be a North-East lattice path. For positive integers $m,n$, an $m \times n$-rectangle inside the $\lambda$-DRH staircase cannot simultaneously contain both a cell from the DRH and a cell from the DRH transpose.
\end{cor}
\begin{proof}
Suppose, for sake of contradiction, that there exists a rectangle containing both a cell $c$ from the DRH and a cell $c'$ from the DRH transpose. Both DRH axes must therefore pass through this rectangle, with $c$ being in the second quadrant and $c'$ being in the fourth quadrant. Since there are no cells from the DRH or the DRH transpose in the first or the third quadrants, every row and every column of this rectangle contain a cell in $\cW_{\lambda}$, and must therefore have a label. Let $r$ and $s$, respectively, be the row and the column labels of the cell $c$. Let $r'$ and $s'$, respectively, be the row and the column labels of the cell $c'$. Thus, $(W^r,W^s), (W^{r'}, W^{s'}) \notin \cW_{\lambda}$, while $(W^r, W^{s'}), (W^{r'}, W^s) \in \cW_{\lambda}$. Lemma \ref{l:Wliffrles} yields
\[
s \ge r' > s' \ge r > s,
\]
a contradiction.
\end{proof}

Recall that we start with the initial $\lambda$-DRH, create the DRH staircase, and aim to describe the variables to associate to each cell in the staircase. The cells in the DRH or in the transpose DRH are promptly filled. Our goal is to describe the variable in the cell $\gamma \in \cW_\lambda$. The mutable cluster variable we associate to $\gamma$ will be denoted $c(\gamma)$. We will see that $c(\gamma)$ has a close relationship with $\Subsk(\gamma)$.

As a $\lambda$-DRH comes with the data of $2\cdot l(\lambda) +4$ variables filled in the $\lambda$-array, a subskeleton $\mu = (\lambda,r,s)$ of $\lambda$ is also endowed with $2 \cdot l(\mu) + 4$ variables filled in the $\mu$-array. Namely, we keep all the variables from the $\lambda$-DRH near lattice points of $\mu$. In this way, for $\gamma \in \cW_\lambda$, the subskeleton $\Subsk(\gamma)$ is actually a DRH itself. Hence, the map $\Subsk$ gives a DRH to every cell in $\cW_\lambda$. We call such a subskeleton $\mu$ with the data of $2 \cdot l(\mu)+4$ rational functions endowed from the larger DRH a {\em DRH subskeleton}.

Observe also that if we look from the bottom to the top, we see that the row labels for cells in $\cW_{\lambda}$ start from the biggest indexed $E$, go in the decreasing order to the smallest indexed $E$, continue with the smallest indexed $N$, and go in the increasing order to the biggest indexed $N$, with the horizontal DRH axis lying in between the $E$'s and the $N$'s. The column labels are similar. If we look from the left to the right, we see that column labels go from the smallest indexed $E$ to the highest indexed $E$, and then from the highest indexed $N$ to the lowest indexed $N$, with the vertical DRH axis lying in between the $E$'s and the $N$'s.

In Figure \ref{fig:Wlambdalabels}, the axes are the DRH axes. The vertical arrows denote the ordering of the row labels as described above. Likewise, the horizontal arrows denote the ordering of the column labels. The words ``start'' and ``end'' denote the fact that in each cell in $\cW_{\lambda}$, we use the row label to indicate the starting point of the corresponding subskeleton and the column label to indicate the ending point.

To obtain the cluster variable $c(\gamma)$ from $\Subsk(\gamma)$, we introduce a process called {\em DRH standardization} in the next section. This process takes in a DRH and gives out a DRH in {\em standard} form. In particular, it produces a square matrix whose determinant is a desired cluster variable up to a sign.

\begin{center}
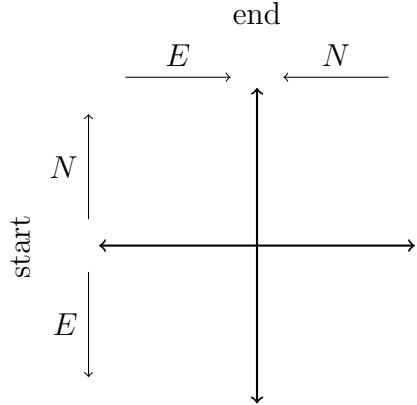
\begin{figure}
\begin{tikzpicture}[scale = 0.7]
\draw[<->, thick] (-3,0) -- (3,0);
\draw[<->, thick] (0,-3) -- (0,3);
\draw[->] (-3.2,0.5) -- (-3.2,2.5);
\draw[->] (-3.2,-0.5) -- (-3.2,-2.5);
\draw[->] (-2.5,3.2) -- (-0.5,3.2);
\draw[->] (2.5,3.2) -- (0.5,3.2);

\node[left] at (-3.2,1.5) {$N$};
\node[left] at (-3.2,-1.5) {$E$};
\node[above] at (-1.5,3.2) {$E$};
\node[above] at (1.5,3.2) {$N$};

\node[left,rotate=90] at (-4.5,0.8) {start};
\node[above] at (0,4) {end};

\end{tikzpicture}
\caption{Ordering of the row and column labels of cells in $\cW_{\lambda}$} \label{fig:Wlambdalabels}
\end{figure}
\end{center}

\section{DRH standardization}
Given a DRH subskeleton $\mu$ of a DRH $\lambda$, the cluster variable $c(\mu)$ is the determinant of a matrix obtained as a result of a process called {\em DRH standardization}. In order to describe the process, we shall discuss certain families of matrices which will appear frequently later. We shall also introduce the notion of matrix concatenation. This notation greatly helps us avoid writing clumsy matrices.

\subsection{Matrix concatenation} We will look at double rim hooks from a more algebraic point of view. We can think of each DRH as a concatenation of small matrices, of dimension $1 \times 2$, $2 \times 1$, or $2 \times 2$. This makes it easier for us to do algebraic manipulations with double rim hooks.

\begin{defn}
A {\em $2 \times 2$ block} is a $2 \times 2$ array with the data of four variables associated to the cells.
\end{defn}

In other words, a $2 \times 2$ block is an $\varnothing$-DRH. It can also be thought of a $2 \times 2$ matrix. The collection of all $2 \times 2$ blocks is denoted $\cB$.

When there are two $2 \times 2$ blocks $b_1, b_2 \in \cB$, we write $b_1 b_2$ to denote the {\em concatenation} of the two blocks, by putting $b_2$ to the right of $b_1$, forming a $2 \times 4$ matrix, which we think of as an $EE$-DRH.

If there is another block $b_3$, we write $b_1 b_2 b_3$ to denote the DRH formed by attaching $b_3$ right above $b_2$ in $b_1 b_2$, thus forming an $EENN$-DRH. If there is also a block $b_4$, the DRH $b_1 b_2 b_3 b_4$ is obtained by attaching $b_4$ to the right of $b_3$. The process continues in this manner: when there is another block $b_k$, attach it to the right or to the top of $b_{k-1}$, so that the attachments are in the zigzagged right-top-right-top order.

\begin{defn}
A {\em horizontal domino} is a $1 \times 2$-matrix. A {\em vertical domino} is a $2 \times 1$-matrix. The collection of all horizontal dominoes is denoted $\cD^H$. The collection of all vertical dominoes is denoted $\cD^V$.
\end{defn}

Suppose that $s \in \cD^H$ is a horizontal domino. We shall write $s b_1 b_2 \cdots b_k$ to denote the DRH formed by attaching $s$ to the bottom of $b_1$ in $b_1 b_2 \cdots b_k$. If, instead $s \in \cD^V$ is a vertical domino, we write $s b_1 b_2 \cdots b_k$ to refer to the DRH formed by first attaching $b_1$ to the right of $s$, and then attaching $b_2$ on top of $b_1$, then $b_3$ to the right of $b_2$ and so on. The attachments are always alternatively right and top (or top and right). It is important to note that when $s \in \cD^V$, $sb_1 \cdots b_k$ is different from attaching $s$ to $b_1 \cdots b_k$.

We also write $b_1 \cdots b_k f$ when it makes sense as follows. If $k$ is even and $f \in \cD^H$, then $b_1 \cdots b_k f$ denotes the DRH obtained by attaching $f$ on top of $b_k$ in $b_1 \cdots b_k$. If $k$ is odd and $f \in \cD^V$, then $b_1 \cdots b_k f$ denotes the DRH obtained by attaching $f$ to the right of $b_k$.

Finally, we write $s b_1 \cdots b_k f$ when it makes sense. Starting with $s$, attaching $b_1$ to $s$, attaching $b_2$ to $b_1$, and so on, we can define $s b_1 \cdots b_k f$ as long as the attachments are always alternatively right and top (or top and right). See Figure \ref{fig:sb1b2bkf} for some cases of $s b_1 \cdots b_k f$.

\begin{center}
\begin{figure}
\begin{tikzpicture}[scale = 0.5]
\tiny
\begin{scope}[shift={(-8,0)}]
\draw (0,0) -- (2,0);
\draw (0,1) -- (4,1);
\draw (0,3) -- (6,3);
\draw (2,5) -- (6,5);
\draw (4,6) -- (6,6);
\draw (0,0) -- (0,3);
\draw (2,0) -- (2,5);
\draw (4,1) -- (4,6);
\draw (6,3) -- (6,6);

\node at (1.0,0.5) {$s$};
\node at (1.0,2.0) {$b_1$};
\node at (3.0,2.0) {$b_2$};
\node at (3.0,4.0) {$\iddots$};
\node at (5.0,4.0) {$b_k$};
\node at (5.0,5.5) {$f$};
\end{scope}

\begin{scope}[shift={(0,0)}]
\draw (0,0) -- (3,0);
\draw (0,2) -- (5,2);
\draw (1,4) -- (6,4);
\draw (3,6) -- (6,6);
\draw (0,0) -- (0,2);
\draw (1,0) -- (1,4);
\draw (3,0) -- (3,6);
\draw (5,2) -- (5,6);
\draw (6,4) -- (6,6);

\node at (0.5,1.0) {$s$};
\node at (2.0,1.0) {$b_1$};
\node at (2.0,3.0) {$b_2$};
\node at (4.0,3.0) {$\iddots$};
\node at (4.0,5.0) {$b_k$};
\node at (5.5,5.0) {$f$};
\end{scope}

\begin{scope}[shift={(+8,0)}]
\draw (0,0) -- (3,0);
\draw (0,2) -- (5,2);
\draw (1,4) -- (7,4);
\draw (3,6) -- (7,6);
\draw (5,7) -- (7,7);
\draw (0,0) -- (0,2);
\draw (1,0) -- (1,4);
\draw (3,0) -- (3,6);
\draw (5,2) -- (5,7);
\draw (7,4) -- (7,7);

\node at (0.5,1.0) {$s$};
\node at (2.0,1.0) {$b_1$};
\node at (2.0,3.0) {$b_2$};
\node at (4.0,3.0) {$b_3$};
\node at (4.0,5.0) {$\iddots$};
\node at (6.0,5.0) {$b_k$};
\node at (6.0,6.5) {$f$};
\end{scope}
\normalsize
\end{tikzpicture}
\caption{Some cases of the DRH $sb_1 b_2 \cdots b_k f$. (Left) $k \in 2 \mathbb{Z}$, $s, f \in \cD^H$. (Center) $k \in 2 \mathbb{Z}$, $s, f \in \cD^V$. (Right) $k \in 2 \mathbb{Z}+1$, $s \in \cD^V$, $f \in \cD^H$.} \label{fig:sb1b2bkf}
\end{figure}
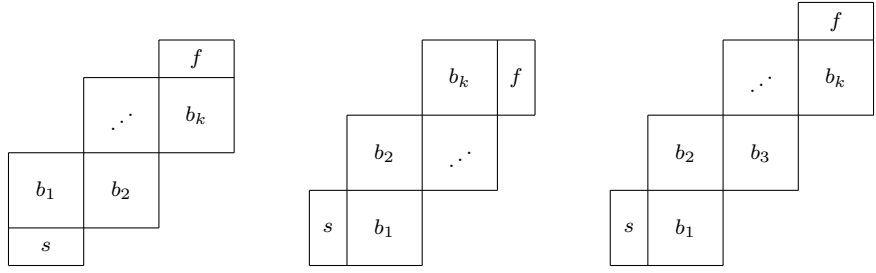
\end{center}

If we have $d_1, d_2 \in \cD^H$, we can form a block $d_1 d_2 \in \cB$ by attaching $d_2$ on the top of $d_1$. Similarly, if $d_1, d_2 \in \cD^V$, attaching $d_2$ to the right of $d_1$ forms a block $d_1 d_2 \in \cB$. 

We have the maps $v_+, v_-: \cB \rightarrow \cD^V$, defined so that for any block $b \in \cB$, $b = v_-(b) v_+(b)$. In other words, $v_-(b)$ is the left $2 \times 1$ column of $b$, while $v_+$ is the right. Similarly, we have $h_+, h_-: \cB \rightarrow \cD^H$ with $b = h_-(b) h_+(b)$ for any block $b$.

The matrices formed by concatenating elements of $\cB$, $\cD^H$, and $\cD^V$, and then adding zeroes in all irrelevant entries will be called {\em DRH matrices}.

The DRH $sb_1b_2\dots b_k f$ is always a square DRH matrix, as long as $s$ and $f$ are appropriate dominoes. We can define the {\em determinants} for these DRH matrices from the usual determinants. A crucial step in proving our main result is establishing determinantal identities of these DRH matrices. The identities we study in this paper are generalizations of Dodgson's identity on minors. Essential examples of these identities will be proved in Section \ref{ss:dodgson}.

\medskip

\subsection{End transformations} There are a certain number of ways of transforming a DRH into another that will be of interest to us. These transformations change a given DRH at one of the two ends of the DRH. First, we observe that given any DRH $\lambda$ of positive length, we can write $\lambda = s \lambda'$, where $s$ is the starting domino of $\lambda$. Here, $s \in \cD^H$ if the first letter in $\lambda$ is $N$, and $s \in \cD^V$ if the letter is $E$. Suppose that $\wt{s}$ is another domino of the same shape as $s$. Then, we consider the following {\em end} transformation of Type I.
\[
\tag{ET1s} \lambda = s \lambda' \mapsto \wt{s} \lambda'.
\]

Suppose $\lambda = s \lambda'$ as before. Let $u$ be another domino of the same shape as $s$, and $\wt{s}$ be another domino of the opposite shape to that of $s$. We have the end transformation of type II.
\[
\tag{ET2s} \lambda = s \lambda' \mapsto \wt{s} (s u) \lambda'.
\]
Finally, suppose again that $\lambda = s \lambda'$. Let $u, \wt{s}$ be dominoes of the same shape as $s$, and let $b \in \cB$ be any block. Then, there is the end transformation of type III.
\[
\tag{ET3s} \lambda = s \lambda' \mapsto \wt{s} b (s u) \lambda'.
\]

Analogous transformations happen at the finishing end of the DRH as well. Suppose that $\lambda = \lambda' f$ where $f$ is the finishing domino. We have
\begin{itemize}
\item (ET1f) $\lambda = \lambda' f \mapsto \lambda' \wt{f}$, where $\wt{f}$ is a domino of the same shape as $f$,
\item (ET2f) $\lambda = \lambda' f \mapsto \lambda' (u f) \wt{f}$, where $u$ is a domino of the same shape as $f$, and $\wt{f}$ is a domino of the opposite shape to that of $f$, and
\item (ET3f) $\lambda = \lambda' f \mapsto \lambda' (u f) b \wt{f}$, where $u$ and $\wt{f}$ are dominoes of the same shape as $f$ and $b$ is a block.
\end{itemize}

An inverse of an end transformation will also be considered an end transformation of the same type. Note that (ET1), whether at the starting or the finishing domino, always preserves the length of the DRH, (ET2) alters the length by $2$, and (ET3) alters the length by $4$.

\medskip

\subsection{Standardization} Given a DRH subskeleton $\mu$ in the DRH staircase for the initial $\lambda$-DRH, we describe the cluster variable $c(\mu)$ associated to the subskeleton via standardization.

\begin{defn}
A non-empty North-East lattice path is said to be {\em standard} (or {\em standardized}) if it is an alternating concatenation of $NE$ and $EN$. A standard path can start with either $NE$ or $EN$.
\end{defn}

In other words, the standard paths include $EN$, $NE$, $ENNE$, $NEEN$, $ENNEEN$, $NEENNE$, and so on. In a manner similar to what we defined in Section \ref{ss:climb}, a lattice point on a North-East lattice path is called a {\em bend} if the two segments which have the point as their endpoints are perpendicular. The starting and the finishing vertices are not considered to be bends. Lattice points on the path that are not bends are called {\em non-bends}. We have the following observations.

\begin{obs}
A non-empty North-East lattice path $\lambda$ is standard if and only if every unit segment in $\lambda$ has one bend endpoint and one non-bend endpoint.
\end{obs}

Let $\lambda$ be a non-empty North-East lattice path. A DRH whose underlying skeleton is $\lambda$ is said to be {\em standard} if $\lambda$ is.

\begin{obs}
A DRH with nontrivial skeleton is standard if and only if it can be decomposed in the form $s b_1 b_2 \dots b_k f$, where $s,f$ are dominoes of the same shape if $2|k$ and of the opposite shape if $2 \nmid k$, and $b_1, b_2, \dots, b_k \in \cB$.
\end{obs}

\begin{defn}
Let $\mu$ be a DRH with nontrivial skeleton. Let $s$ and $f$ be the starting and the finishing dominoes, respectively, of $\mu$. Let $b_1, b_2, \dots, b_k$ be the $2 \times 2$ squares at the bends of $\mu$ in their usual order (from left to right in the underlying North-East word of $\mu$). Then, the {\em standardization} of the DRH $\mu$ is the DRH
\[
\Std(\mu) := s b_1 b_2 \dots b_k f.
\]
\end{defn}

\begin{defn}
Let $\mu$ be a DRH with nontrivial skeleton. Suppose that $\Std(\mu)$ is an $\eta \times \eta$ matrix. Then, the {\em cluster variable} associated to $\mu$ is defined to be the polynomial
\[
c(\mu) := (-1)^{\frac{(\eta-1)(\eta-2)}{2}} \cdot \det \left( \Std(\mu) \right).
\]
\end{defn}

The sign $(-1)^{\frac{(\eta-1)(\eta-2)}{2}}$ is positive when $\eta$ is $1$ or $2$ modulo $4$, and negative when $\eta$ is $3$ or $0$ modulo $4$.

Our main theorem states that these cluster variables $c(\mu)$, together with those already present as entries in the DRH, are the mutable cluster variables of the DRH cluster algebra $\cA_\lambda$ when $\mu \le \lambda$ are subskeleta.

Moreover, since $\Subsk$ is a combinatorial bijection between the cells in $\cW_\lambda$ and the nonempty subskeleta of $\lambda$. We may now put each mutable cluster variable $c(\mu)$ inside $\Subsk^{-1}(\mu) \in \cW_{\lambda}$. In this way, the DRH staircase is now completely filled with all the desired mutable cluster variables. For a cell $C$ inside the staircase (whether it is inside $\cW_{\lambda}$, or the original DRH, or the DRH transpose), we write $e(C)$ to denote the rational function that is the mutable cluster variable associated to the cell $C$.

\subsection{Generalization of Dodgson's identity} \label{ss:dodgson}

The main point of this section is to show that given an identity of a certain form on minors, one can adjoin the same rows and columns to all the minors involved while preserving the identity.

Let $n_1, n_2 \in\mathbb{Z}_{>0}$. Let $X = [x_{i,j}]$ be an $n_1 \times n_2$ matrix. For $I \subseteq [n_1]$ and $J \subseteq [n_2]$ of the same size let $\Delta_{I,J}$ denote the minor of $X$ with rows $I$ and columns $J$. We will be using the common convention that $\Delta_{\varnothing,\varnothing} = 1$.

\begin{lemma}
\label{lem:chunk}
Let $k, l, m \in\mathbb{Z}_{>0}$. Suppose we have a collection $\{I_i^j\}_{\substack{i\in[l]\\j\in[k]}}$ of subsets of $[n_1]$ and a collection $\{J_i^j\}_{\substack{i\in[l]\\j\in[k]}}$ of subsets of $[n_2]$, with $\forall i, j, \abs{I_i^j} = \abs{J_i^j}$ and $\forall i, \sum_j\abs{I_i^j} = m$. Moreover, suppose for some $c_1,\dots, c_l\in\mathbb{Z}$, we have
\begin{equation}
\label{eq:nochunk}
\sum_{i=1}^l c_i \left(\prod_{j=1}^k\Delta_{I_i^j,J_i^j}\right) = 0.
\end{equation} 
Then for any $I'\subseteq [n_1]\setminus\left(\bigcup_{i,j} I_i^j\right)$ and $J'\subseteq [n_2]\setminus\left(\bigcup_{i,j} J_i^j\right)$ of the same size we have 
\begin{equation}
\label{eq:yeschunk}
\sum_{i=1}^l c_i \left(\prod_{j=1}^k\Delta_{I_i^j\cup I',J_i^j\cup J'}\right) = 0.
\end{equation}
\end{lemma}

We will now handle a special case of Lemma \ref{lem:chunk}.
\begin{lemma}
Suppose $n' \le \min(n_1, n_2)$, and $I = \{i_1,\dots,i_{n'}\} \subseteq [n_1]$, $J = \{j_1,\dots,j_{n'}\} \subseteq [n_2]$. Then $\Delta_{I,J}$ is a polynomial $P_{I,J}(x_{i,j})$. Fix a pair of equinumerous sets: $I'\subseteq [n_1]\setminus I$ and $J'\subseteq [n_2]\setminus J$. For $i\in I$ and $j\in J$ define $\tilde{x}_{i,j} = \Delta_{\{i\}\cup I', \{j\}\cup J'}$. Then 
\[\Delta_{I\cup I', J\cup J'}\left(\Delta_{I',J'}\right)^{n'-1} = P_{I,J}(\tilde{x}_{i,j}).\]
\end{lemma}
\begin{proof}
We will need to use a certain Pl\"ucker relation; these are described nicely in \cite[Section 2.1]{Lec93}. In the notation thereof, we will take the special case of the matrix schematically pictured in Figure \ref{fig:matrix}. Note that in the actual matrix $X$, the indices $I$, $J$, $I'$ and $J'$ will interlace in some way. To be more precise we take the sub-matrix of $X$ on rows $I\cup I'$ and columns $J\cup J'$, and augment it on the right by a similar matrix except that (i) we put the identity matrix in place of the matrix with rows $I$ and columns $J$ and (ii) we put the zero matrix in place of the matrix with rows $I'$ and columns $J$. In the notation of Leclerc's, we take $a,b,\dots, d$ to be the first $n' + \abs{I'}$ columns and $e,f,\dots, h$ to be the last $n' + \abs{I'}$ columns, and apply Equation (2) in \cite[Section 2.1]{Lec93}.

\begin{center}
\begin{figure}
\begin{tikzpicture}[scale = 0.7]
\draw (0,0) -- (10,0);
\draw (0,2) -- (10,2);
\draw (0,5) -- (10,5);
\draw (0,0) -- (0,5);
\draw (3,0) -- (3,5);
\draw (5,0) -- (5,5);
\draw (8,0) -- (8,5);
\draw (10,0) -- (10,5);

\node at (5.5,4.5) {$1$};
\node at (6.0,4.0) {$1$};
\node at (6.5,3.5) {$1$};
\node at (7.0,3.0) {$1$};
\node at (7.5,2.5) {$1$};

\node at (-0.5,3.5) {$I$};
\node at (-0.5,1) {$I'$};

\node at (1.5,-0.5) {$J$};
\node at (4,-0.5) {$J'$};
\node at (6.5,-0.5) {$J$};
\node at (9,-0.5) {$J'$};
\node at (6.5,1) {\large $0$};

\fill[orange!20] (0.05,0.05) rectangle (3-0.05,2-0.05);
\fill[orange!30] (0.05,2.05) rectangle (3-0.05,5-0.05);
\fill[orange!40] (3.05,0.05) rectangle (5-0.05,2-0.05);
\fill[orange!50] (3.05,2.05) rectangle (5-0.05,5-0.05);

\fill[orange!40] (8.05,0.05) rectangle (10-0.05,2-0.05);
\fill[orange!50] (8.05,2.05) rectangle (10-0.05,5-0.05);

\end{tikzpicture}
\caption{The matrix concerning the Pl\"ucker relation of interest.} \label{fig:matrix}
\end{figure}
\end{center}

The resulting identity is 
\[\Delta_{I\cup I', J\cup J'}\Delta_{I',J'}
   =\sum_{l=1}^{n'} (-1)^{l + 1}\Delta_{\{i_1\}\cup I', \{j_{l}\}\cup J'}\Delta_{\left(I\setminus\{i_1\}\right)\cup I',\left(J\setminus\{j_{l}\}\right)\cup J'}.\]
In case $I'$ and $J'$ are empty, this is just row expansion of the determinant. 

Now we can finish the calculation
\[\begin{array}{rcl}\medskip
\displaystyle\Delta_{I\cup I', J\cup J'}\left(\Delta_{I',J'}\right)^{n'-1} 
    &=&\displaystyle\sum_{l=1}^{n'} (-1)^{l + 1}\Delta_{\{i_1\}\cup I', \{j_{l}\}\cup J'}\Delta_{\left(I\setminus\{i_1\}\right)\cup I',\left(J\setminus\{j_{l}\}\right)\cup J'} \left(\Delta_{I',J'}\right)^{n'-2}\\\medskip
		&=&\displaystyle\sum_{l=1}^{n'} (-1)^{l + 1}\tilde{x}_{i_1,j_{l}}\Delta_{\left(I\setminus\{i_1\}\right)\cup I',\left(J\setminus\{j_{l}\}\right)\cup J'} \left(\Delta_{I',J'}\right)^{n'-2}\\\medskip
		&=&\displaystyle\sum_{l=1}^{n'} (-1)^{l + 1}\tilde{x}_{i_1,j_{l}}P_{I\setminus\{i_1\},J\setminus\{j_{l}\}}(\tilde{x}_{i,j})\\
		&=&\displaystyle P_{I,J}(\tilde{x}_{i,j}),
\end{array}\]
Where the third equality holds by induction on $n'$ and the fourth is just the expansion of the determinant in $j_{n'}$-th row.
\end{proof}

\begin{proof}[Proof of Lemma \ref{lem:chunk}]
The proof now becomes easy. Equation \eqref{eq:nochunk}, in the notation of the previous lemma becomes
\[\sum_{i=1}^l c_i \left(\prod_{j=1}^k P_{I_i^j,J_i^j}(x_{i,j})\right) = 0,\]
while the left hand side of equation \eqref{eq:yeschunk} becomes 
\[\frac{\displaystyle\sum_{i=1}^l c_i \left(\prod_{j=1}^k P_{I_i^j,J_i^j}(\tilde{x}_{i,j})\right)}{\displaystyle\left(\Delta_{I',J'}\right)^{n'-1}}.\]
This is obtained from the left hand side of the previous equation via a substitution, thus it must be $0$, finishing the proof.
\end{proof}

We will briefly discuss here why Lemma \ref{lem:chunk} is referred to as a generalization of Dodgson's identity. Suppose we would like to calculate the $3 \times 3$ determinant:
\[
\begin{vmatrix}
x_{11} & x_{12} & x_{13} \\
x_{21} & x_{22} & x_{23} \\
x_{31} & x_{32} & x_{33} \\
\end{vmatrix}.
\]
Dodgson \cite{Dod1866} proposed the following method. First, construct the matrix of minors
\[
M = \begin{bmatrix}
\Delta_{12,12} & \Delta_{12,23} \\
\Delta_{23,12} & \Delta_{23,23}
\end{bmatrix}.
\]
Then, the desired $3 \times 3$ determinant is $\det(M)/x_{22}$. This procedure is called Dodgson condensation. The underlying algebraic identity is simply
\[
\Delta_{123,123} \Delta_{2,2} = \Delta_{12,12} \Delta_{23,23} - \Delta_{12,23} \Delta_{23,12}.
\]
More generally, the following identity, called Dodgson's identity (also known as Desnanot-Jacobi identity), holds true for $n \times n$ matrices:
\[
\Delta_{[n],[n]} \Delta_{[n]-\{1,n\}, [n] - \{1,n\}} = \Delta_{[n-1],[n-1]} \Delta_{[n]-\{1\},[n]-\{1\}} - \Delta_{[n-1],[n]-\{1\}} \Delta_{[n]-\{1\},[n-1]}.
\]
What is this identity in light of Lemma \ref{lem:chunk}? Note that we may take $I' = J' = [n] - \{1,n\}$. Removing $I'$ and $J'$ from every minor reveals the underlying identity:
\[
\Delta_{\{1,n\}, \{1,n\}} \Delta_{\varnothing, \varnothing} = \Delta_{1,1} \Delta_{n,n} - \Delta_{1,n} \Delta_{n,1}
\]
which is the most basic non-trivial determinantal identity. Our lemma says that we can always add new indices to every minor in a homogeneous determinantal identity -- like the one above -- to obtain a new higher-dimensional identity on minors. Thus, Dodgson's identity immediately follows.

From the lemma, we can deduce a number of corollaries. These algebraic results will be useful in the proof of Theorem \ref{thm:ultiDodgson}, which is one of our main results.

\begin{cor} \label{c:detid00++}
Let $k$ be an even positive integer. Let $s, \wt{s}, f, \wt{f} \in \cD^H$ and $b_1, \dots, b_k \in \cB$. Then, we have the following determinantal identity
\begin{align*}
&\det(s b_1 \cdots b_k \wt{f}) \det(\wt{s} b_1 \cdots b_k f) - \det(s b_1 \cdots b_k f) \det(\wt{s} b_1 \cdots b_k \wt{f}) \\
&= |\wt{s} s| \cdot |b_1| \cdot \, \cdots \, \cdot |b_k| \cdot |f \wt{f}|.
\end{align*}
\end{cor}
\begin{proof}
Consider the $(k+4) \times (k+2)$ matrix $X$ formed by attaching the block $\wt{s} s$ right below $b_1$ in $b_1 b_2 \cdots b_k (f \wt{f})$ with rows and columns labeled as in the figure below.
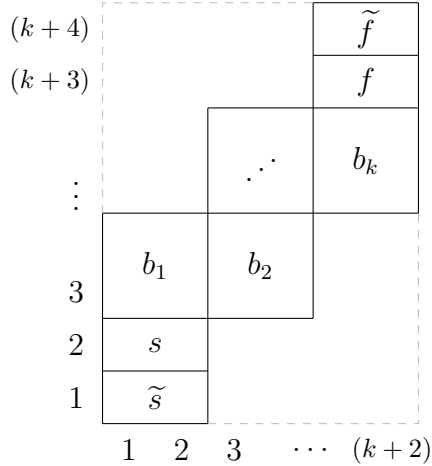
\begin{center}
\begin{tikzpicture}[scale = 0.7]
\draw (0,0) -- (2,0);
\draw (0,1) -- (2,1);
\draw (0,2) -- (4,2);
\draw (0,4) -- (6,4);
\draw (2,6) -- (6,6);
\draw (4,7) -- (6,7);
\draw (4,8) -- (6,8);
\draw (0,0) -- (0,4);
\draw (2,0) -- (2,6);
\draw (4,2) -- (4,8);
\draw (6,4) -- (6,8);

\draw[gray!50, dashed] (0,4) -- (0,8) -- (4,8);
\draw[gray!50, dashed] (2,0) -- (6,0) -- (6,4);

\node at (1.0,0.5) {$\wt{s}$};
\node at (1.0,1.5) {$s$};
\node at (1.0,3.0) {$b_1$};
\node at (3.0,3.0) {$b_2$};
\node at (3.0,5.0) {$\iddots$};
\node at (5.0,5.0) {$b_k$};
\node at (5.0,6.5) {$f$};
\node at (5.0,7.5) {$\wt{f}$};

\node at (0.5,-0.5) {$1$};
\node at (1.5,-0.5) {$2$};
\node at (2.5,-0.5) {$3$};
\node at (4.0,-0.5) {$\dots$};
\footnotesize
\node at (5.5,-0.5) {$(k+2)$};
\normalsize

\node at (-0.5,0.5) {$1$};
\node at (-0.5,1.5) {$2$};
\node at (-0.5,2.5) {$3$};
\node at (-0.5,4.5) {$\vdots$};
\footnotesize
\node at (-1.0,6.5) {$(k+3)$};
\node at (-1.0,7.5) {$(k+4)$};
\normalsize
\end{tikzpicture}
\end{center}
The left-hand-side of the identity we want to prove can be written in terms of minors of $X$ as follows.
\begin{align*}
\Delta_{\{2, \dots, k+2, k+4\}, [k+2]} \Delta_{\{1,3,\dots, k+3\}, [k+2]} - \Delta_{\{2,\dots, k+3\},[k+2]} \Delta_{\{1,3,\dots, k+2, k+4\}, [k+2]}.
\end{align*}
Consider $I' = \{3, \dots, k+2\}$ and $J' = \{3, \dots, k+2\}$. In all the four minors, the rows and the columns include $I'$ and $J'$. Taking $I'$ and $J'$ out from the minors yields
\[
\Delta_{\{2,k+4\},\{1,2\}} \cdot \Delta_{\{1,k+3\}, \{1,2\}} - \Delta_{\{2,k+3\},\{1,2\}} \cdot \Delta_{\{1,k+4\},\{1,2\}}
\]
Since we have the identity
\begin{align*}
& \Delta_{\{2,k+4\},\{1,2\}} \cdot \Delta_{\{1,k+3\}, \{1,2\}} - \Delta_{\{2,k+3\},\{1,2\}} \cdot \Delta_{\{1,k+4\},\{1,2\}} \\
& = \Delta_{\{1,2\},\{1,2\}} \cdot \Delta_{\{k+3,k+4\}, \{1,2\}},
\end{align*}
Lemma \ref{lem:chunk} allows us to add back $I'$ and $J'$ to obtain the higher-dimensional identity
\begin{align*}
LHS &= \Delta_{[k+2],[k+2]} \cdot \Delta_{\{3,\dots, k+4\},[k+2]} \\
&= \big( |\wt{s} s| \cdot |b_2| \cdot \, \cdots \, \cdot |b_k| \big) \big( |b_1| \cdot |b_3| \cdot \, \cdots \, \cdot |b_{k-1}| \cdot |f\wt{f}| \big)
\end{align*}
as desired. It is helpful to look at the determinantal identity via the following diagram.
\begin{center}
\begin{tikzpicture}[scale = 0.25]
\begin{scope}[shift={(-26,0)}]
\fill[orange!50] (0,1) rectangle (2,4);
\fill[orange!50] (2,2) rectangle (4,6);
\fill[orange!50] (4,4) rectangle (6,6);
\fill[orange!50] (4,7) rectangle (6,8);

\draw (0,0) -- (2,0);
\draw (0,1) -- (2,1);
\draw (0,2) -- (4,2);
\draw (0,4) -- (6,4);
\draw (2,6) -- (6,6);
\draw (4,7) -- (6,7);
\draw (4,8) -- (6,8);
\draw (0,0) -- (0,4);
\draw (2,0) -- (2,6);
\draw (4,2) -- (4,8);
\draw (6,4) -- (6,8);

\node at (7,4) {$\cdot$};
\end{scope}

\begin{scope}[shift={(-18,0)}]
\fill[orange!50] (0,0) rectangle (2,1);
\fill[orange!50] (0,2) rectangle (4,4);
\fill[orange!50] (2,4) rectangle (6,6);
\fill[orange!50] (4,6) rectangle (6,7);

\draw (0,0) -- (2,0);
\draw (0,1) -- (2,1);
\draw (0,2) -- (4,2);
\draw (0,4) -- (6,4);
\draw (2,6) -- (6,6);
\draw (4,7) -- (6,7);
\draw (4,8) -- (6,8);
\draw (0,0) -- (0,4);
\draw (2,0) -- (2,6);
\draw (4,2) -- (4,8);
\draw (6,4) -- (6,8);
\node at (8,4) {$-$};
\end{scope}

\begin{scope}[shift={(-8,0)}]
\fill[orange!50] (0,1) rectangle (2,4);
\fill[orange!50] (2,2) rectangle (4,6);
\fill[orange!50] (4,4) rectangle (6,7);

\draw (0,0) -- (2,0);
\draw (0,1) -- (2,1);
\draw (0,2) -- (4,2);
\draw (0,4) -- (6,4);
\draw (2,6) -- (6,6);
\draw (4,7) -- (6,7);
\draw (4,8) -- (6,8);
\draw (0,0) -- (0,4);
\draw (2,0) -- (2,6);
\draw (4,2) -- (4,8);
\draw (6,4) -- (6,8);

\node at (7,4) {$\cdot$};
\end{scope}

\begin{scope}[shift={(0,0)}]
\fill[orange!50] (0,0) rectangle (2,1);
\fill[orange!50] (0,2) rectangle (4,4);
\fill[orange!50] (2,4) rectangle (6,6);
\fill[orange!50] (4,7) rectangle (6,8);

\draw (0,0) -- (2,0);
\draw (0,1) -- (2,1);
\draw (0,2) -- (4,2);
\draw (0,4) -- (6,4);
\draw (2,6) -- (6,6);
\draw (4,7) -- (6,7);
\draw (4,8) -- (6,8);
\draw (0,0) -- (0,4);
\draw (2,0) -- (2,6);
\draw (4,2) -- (4,8);
\draw (6,4) -- (6,8);

\node at (8,4) {$=$};
\end{scope}

\begin{scope}[shift={(10,0)}]
\fill[orange!50] (0,0) rectangle (2,4);
\fill[orange!50] (2,2) rectangle (4,6);
\fill[orange!50] (4,4) rectangle (6,6);

\draw (0,0) -- (2,0);
\draw (0,1) -- (2,1);
\draw (0,2) -- (4,2);
\draw (0,4) -- (6,4);
\draw (2,6) -- (6,6);
\draw (4,7) -- (6,7);
\draw (4,8) -- (6,8);
\draw (0,0) -- (0,4);
\draw (2,0) -- (2,6);
\draw (4,2) -- (4,8);
\draw (6,4) -- (6,8);

\node at (7,4) {$\cdot$};
\node at (-2,4) {$=$};
\end{scope}

\begin{scope}[shift={(18,0)}]
\fill[orange!50] (0,2) rectangle (2,4);
\fill[orange!50] (2,2) rectangle (4,6);
\fill[orange!50] (4,4) rectangle (6,8);

\draw (0,0) -- (2,0);
\draw (0,1) -- (2,1);
\draw (0,2) -- (4,2);
\draw (0,4) -- (6,4);
\draw (2,6) -- (6,6);
\draw (4,7) -- (6,7);
\draw (4,8) -- (6,8);
\draw (0,0) -- (0,4);
\draw (2,0) -- (2,6);
\draw (4,2) -- (4,8);
\draw (6,4) -- (6,8);
\end{scope}

\begin{scope}[shift={(10,-10)}]
\fill[orange!50] (0,0) rectangle (2,2);
\fill[orange!50] (2,2) rectangle (4,4);
\fill[orange!50] (4,4) rectangle (6,6);

\draw (0,0) -- (2,0);
\draw (0,1) -- (2,1);
\draw (0,2) -- (4,2);
\draw (0,4) -- (6,4);
\draw (2,6) -- (6,6);
\draw (4,7) -- (6,7);
\draw (4,8) -- (6,8);
\draw (0,0) -- (0,4);
\draw (2,0) -- (2,6);
\draw (4,2) -- (4,8);
\draw (6,4) -- (6,8);

\node at (7,4) {$\cdot$};
\node at (-2,4) {$=$};
\end{scope}

\begin{scope}[shift={(18,-10)}]
\fill[orange!50] (0,2) rectangle (2,4);
\fill[orange!50] (2,4) rectangle (4,6);
\fill[orange!50] (4,6) rectangle (6,8);

\draw (0,0) -- (2,0);
\draw (0,1) -- (2,1);
\draw (0,2) -- (4,2);
\draw (0,4) -- (6,4);
\draw (2,6) -- (6,6);
\draw (4,7) -- (6,7);
\draw (4,8) -- (6,8);
\draw (0,0) -- (0,4);
\draw (2,0) -- (2,6);
\draw (4,2) -- (4,8);
\draw (6,4) -- (6,8);
\end{scope}
\end{tikzpicture}
\end{center}
Each colored part in the equation above shows the minor the corresponding term represents.
\end{proof}

We have now seen the power of Lemma \ref{lem:chunk} illustrated in the proof of Corollary \ref{c:detid00++}. When we want to simplify a certain expression involving minors of many rows and columns, we can remove ``common chunks'' to reduce the minor identity to a low-dimensional identity, and then we add back the removed chunks.

This technique can be applied very generally. In fact, given a standard DRH $\mu$, a compatible starting end transformation $l_s$, and a compatible finishing end transformation $l_f$, we can always decompose
\[
\det(\mu) \cdot \det(l_s l_f \mu) - \det(l_s \mu) \cdot \det(l_f \mu)
\]
using Lemma \ref{lem:chunk}. Corollary \ref{c:detid00++} is an example in which both $l_s$ and $l_f$ are Type I end transformations. We illustrate further examples in the following corollaries.

\begin{cor} \label{c:detid-20++}
Let $k$ be an even positive integer. Let $s, \wt{s}, f, \wt{f} \in \cD^H$ and $b_1, \dots, b_k, b'_w, b'_0 \in \cB$ such that $f = h_+(b'_w)$. Then, we have the following determinantal identity
\begin{align*}
& \det(s b_1 \cdots b_k f) \det( \wt{s} b_1 \cdots b_k b'_w b'_0 \wt{f}) - \det(\wt{s} b_1 \cdots b_k f) \det( s b_1 \cdots b_k b'_w b'_0 \wt{f}) \\
&= |\wt{s} s| \cdot |b_1| \cdot \, \cdots \, \cdot |b_k| \cdot |b'_w| \cdot |h_+(b'_0) \wt{f}|.
\end{align*}
\end{cor}
\begin{proof}
Similarly to the proof of the previous corollary, we consider the following $(k+5) \times (k+4)$ matrix $X$ formed by attaching the block $\wt{s} s$ right below $b_1$ in $b_1 b_2 \cdots b_k b'_w b'_0 \wt{f}$ with rows and columns labeled as in the figure below.
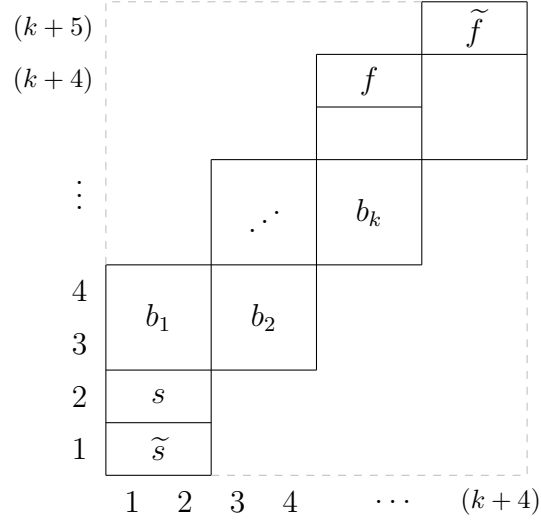
\begin{center}
\begin{tikzpicture}[scale = 0.7]
\draw (0,0) -- (2,0);
\draw (0,1) -- (2,1);
\draw (0,2) -- (4,2);
\draw (0,4) -- (6,4);
\draw (2,6) -- (8,6);
\draw (4,7) -- (6,7);
\draw (4,8) -- (8,8);
\draw (6,9) -- (8,9);
\draw (0,0) -- (0,4);
\draw (2,0) -- (2,6);
\draw (4,2) -- (4,8);
\draw (6,4) -- (6,9);
\draw (8,6) -- (8,9);

\draw[gray!50, dashed] (0,4) -- (0,9) -- (6,9);
\draw[gray!50, dashed] (2,0) -- (8,0) -- (8,6);

\node at (1.0,0.5) {$\wt{s}$};
\node at (1.0,1.5) {$s$};
\node at (1.0,3.0) {$b_1$};
\node at (3.0,3.0) {$b_2$};
\node at (3.0,5.0) {$\iddots$};
\node at (5.0,5.0) {$b_k$};
\node at (5.0,7.5) {$f$};
\node at (7.0,8.5) {$\wt{f}$};

\node at (0.5,-0.5) {$1$};
\node at (1.5,-0.5) {$2$};
\node at (2.5,-0.5) {$3$};
\node at (3.5,-0.5) {$4$};
\node at (5.5,-0.5) {$\dots$};
\footnotesize
\node at (7.5,-0.5) {$(k+4)$};
\normalsize

\node at (-0.5,0.5) {$1$};
\node at (-0.5,1.5) {$2$};
\node at (-0.5,2.5) {$3$};
\node at (-0.5,3.5) {$4$};
\node at (-0.5,5.5) {$\vdots$};
\footnotesize
\node at (-1.0,7.5) {$(k+4)$};
\node at (-1.0,8.5) {$(k+5)$};
\normalsize
\end{tikzpicture}
\end{center}
The left-hand-side of the identity to prove is
\[
\Delta_{\{2,\dots, k+2,k+4\},[k+2]} \Delta_{\{1,3,\dots, k+5\},[k+4]} - \Delta_{\{1,3,\dots,k+2,k+4\},[k+2]} \Delta_{\{2, \dots, k+5\},[k+4]}.
\]
Removing $I' = \{3, \dots, k+2\} \cup \{k+4\}$ and $J' = \{2, \dots, k+2\}$ yields
\[
\Delta_{\{2\},\{1\}} \cdot \Delta_{\{1,k+3,k+5\},\{1,k+3,k+4\}} - \Delta_{\{1\},\{1\}} \cdot \Delta_{\{2,k+3,k+5\}, \{1,k+3,k+4\}}.
\]
Note that we have the following low-dimensional identity:
\begin{align*}
&\Delta_{\{2\},\{1\}} \cdot \Delta_{\{1,k+3,k+5\},\{1,k+3,k+4\}} - \Delta_{\{1\},\{1\}} \cdot \Delta_{\{2,k+3,k+5\}, \{1,k+3,k+4\}} \\
& = \Delta_{\{k+3\},\{1\}} \cdot \Delta_{\{1,2,k+5\},\{1,k+3,k+4\}} - \Delta_{\{k+5\},\{1\}} \cdot \Delta_{\{1,2,k+3\}, \{1,k+3,k+4\}}.
\end{align*}
Adding back chunks yields
\begin{align*}
LHS &= \Delta_{\{3, \dots, k+4\},[k+2]} \cdot \Delta_{\{1,\dots, k+2, k+4, k+5\}, [k+4]} \\
& - \Delta_{\{3, \dots, k+2, k+4, k+5\},[k+2]} \cdot \Delta_{[k+4],[k+4]}.
\end{align*}
Observe that $\Delta_{\{3, \dots, k+2, k+4, k+5\},[k+2]} = 0$. Therefore,
\begin{align*}
LHS &= \Delta_{\{3, \dots, k+4\},[k+2]} \cdot \Delta_{\{1,\dots, k+2, k+4, k+5\}, [k+4]} \\
&= \big( |b_1| \cdot |b_3| \cdot \, \cdots \, \cdot |b_{k-1}| \cdot |b'_w| \big) \cdot \big( |\wt{s} s| \cdot |b_2| \cdot |b_4| \cdot \, \cdots \, \cdot |b_k| \cdot |h_+(b'_0) \wt{f}| \big)
\end{align*}
as desired.
\end{proof}

\begin{cor} \label{c:detid-1+20+}
Let $k$ be an even positive integer. Let $s, \wt{s}, f \in \cD^H$, $\varphi \in \cD^V$, $b_0, b_1, \dots, b_k, b_w, b'_w \in \cB$. Suppose that $v_-(b_w) = s$ and $v_+(b'_w) = f$. Then,
\begin{align*}
& \det(s b_1 \cdots b_k b'_w \varphi) \det(\wt{s} b_0 b_w b_1 \cdots b_k f) - \det(s b_1 \cdots b_k f) \det( \wt{s} b_0 b_w b_1 \cdots b_k b'_w \varphi) \\
&= |\wt{s} v_-(b_0)| \cdot |b_w| \cdot |b'_w| \cdot \delta \cdot |b_1| \cdot \, \cdots \, \cdot |b_k|,
\end{align*}
where $\delta$ is the upper entry in $\varphi$.
\end{cor}
\begin{proof}
Consider the following $(k+5) \times (k+5)$ matrix $X = \wt{s} b_0 b_w b_1 \cdots b_k b'_w \varphi$, with rows and columns labeled as in the figure below.
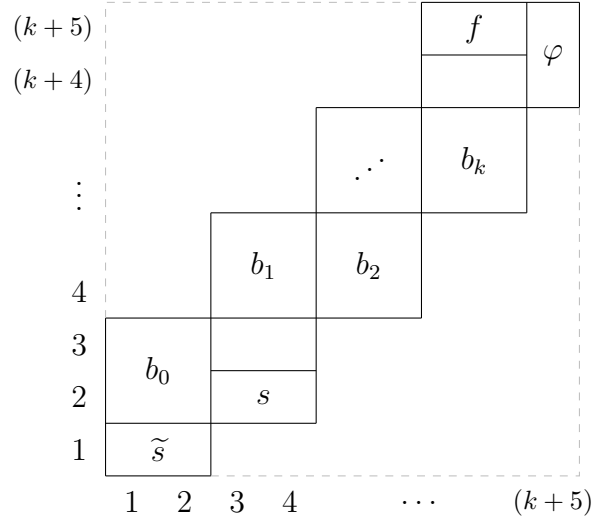
\begin{center}
\begin{tikzpicture}[scale = 0.7]
\draw (0,0) -- (2,0);
\draw (0,1) -- (4,1);
\draw (2,2) -- (4,2);
\draw (0,3) -- (6,3);
\draw (2,5) -- (8,5);
\draw (4,7) -- (9,7);
\draw (6,8) -- (8,8);
\draw (6,9) -- (9,9);
\draw (0,0) -- (0,3);
\draw (2,0) -- (2,5);
\draw (4,1) -- (4,7);
\draw (6,3) -- (6,9);
\draw (8,5) -- (8,9);
\draw (9,7) -- (9,9);

\draw[gray!50, dashed] (0,3) -- (0,9) -- (6,9);
\draw[gray!50, dashed] (2,0) -- (9,0) -- (9,7);

\node at (1.0,0.5) {$\wt{s}$};
\node at (1.0,2.0) {$b_0$};
\node at (3.0,1.5) {$s$};
\node at (3.0,4.0) {$b_1$};
\node at (5.0,4.0) {$b_2$};
\node at (5.0,6.0) {$\iddots$};
\node at (7.0,6.0) {$b_k$};
\node at (7.0,8.5) {$f$};
\node at (8.5,8) {$\varphi$};

\node at (0.5,-0.5) {$1$};
\node at (1.5,-0.5) {$2$};
\node at (2.5,-0.5) {$3$};
\node at (3.5,-0.5) {$4$};
\node at (6.0,-0.5) {$\dots$};
\footnotesize
\node at (8.5,-0.5) {$(k+5)$};
\normalsize

\node at (-0.5,0.5) {$1$};
\node at (-0.5,1.5) {$2$};
\node at (-0.5,2.5) {$3$};
\node at (-0.5,3.5) {$4$};
\node at (-0.5,5.5) {$\vdots$};
\footnotesize
\node at (-1.0,7.5) {$(k+4)$};
\node at (-1.0,8.5) {$(k+5)$};
\normalsize
\end{tikzpicture}
\end{center}
The left-hand-side of the identity to prove is
\[
\Delta_{\{2,4,\dots,k+5\}, \{3, \dots, k+5\}} \Delta_{\{1, \dots, k+3, k+5\}, [k+4]} - \Delta_{\{2,4, \dots, k+3, k+5\}, \{3, \dots, k+4\}} \Delta_{[k+5],[k+5]}.
\]
Removing $I' = \{2,4, \dots, k+3, k+5\}$ and $J' = \{3, \dots, k+4\}$ yields
\[
\Delta_{\{k+4\},\{k+5\}} \cdot \Delta_{\{1,3\},\{1,2\}} - \Delta_{\varnothing, \varnothing} \Delta_{\{1,3,k+4\}, \{1,2,k+5\}}.
\]
We have the following low-dimensional identity
\begin{align*}
& \Delta_{\{k+4\},\{k+5\}} \cdot \Delta_{\{1,3\},\{1,2\}} - \Delta_{\varnothing, \varnothing} \Delta_{\{1,3,k+4\}, \{1,2,k+5\}} \\
& = \Delta_{\{3\},\{k+5\}} \cdot \Delta_{\{1,k+4\},\{1,2\}} - \Delta_{\{1\}, \{k+5\}} \cdot \Delta_{\{3,k+4\},\{1,2\}}.
\end{align*}
Adding back chunks yields
\begin{align*}
LHS &= \Delta_{\{2,\dots, k+3,k+5\}, \{3, \dots, k+5\}} \cdot \Delta_{\{1,2,4,\dots, k+5\},[k+4]} \\
& \hphantom{=} - \underbrace{\Delta_{\{1,2,4, \dots, k+3, k+5\}, \{3, \dots, k+5\}}}_{=0} \cdot \Delta_{\{2, \dots, k+5\}, [k+4]} \\
&= \big( |b_w| \cdot |b_2| \cdot |b_4| \cdot \, \cdots \, \cdot |b_k| \cdot \delta \big) \cdot \big( |\wt{s} v_-(b_0)| \cdot |b_1| \cdot |b_3| \cdot \, \cdots \, \cdot |b_{k-1}| \cdot |b'_w| \big)
\end{align*}
as desired.
\end{proof}

\section{The main decomposition theorem}

Recall that we have a proposed mutable cluster variable which is the determinant (up to a sign) of a standard DRH in each cell of $\cW_{\lambda}$. We say that $(m_{11}, m_{12}; m_{21}, m_{22})$ is a {\em quadruple inside $\cW_{\lambda}$} if they are the upper-left, upper-right, lower-left, and lower-right cells, respectively, of a $2 \times 2$-square inside $\cW_{\lambda}$.

Quadruples inside $\cW_{\lambda}$ are of great interest to us, because they are where worm operations happen. Consider, for example, the quadruple
\[
\fq = (m_{11}, m_{12}; m_{21}, m_{22})
\]
inside $\cW_{\lambda}$ and a worm $w$ that contains $m_{21} \rightarrow m_{11} \rightarrow m_{12}$. The worm operation on $w$ at $m_{11}$ sends the worm to another worm $w'$ with $m_{11}$ replaced with $m_{22}$. In Section \ref{ss:fcdesc}, we will describe the quiver $Q_w$ for every worm $w$. Understanding the determinant
\[
|\fq| := \begin{vmatrix}
e(m_{11}) & e(m_{12}) \\
e(m_{21}) & e(m_{22})
\end{vmatrix}
\]
is important in describing the frozen arrows at $m_{11}$ in $w$ (and therefore $m_{22}$ in $w'$). Our main result of this section is Theorem \ref{thm:ultiDodgson}, in which we factor the determinant $|\fq|$ into a product of frozen variables, and thus allowing us to describe the frozen arrows. In Corollaries \ref{c:detid00++}, \ref{c:detid-20++}, and \ref{c:detid-1+20+}, we have seen some examples of how such a $2 \times 2$ determinant of standard DRH determinants factors into small matrices. We will see in the following that such factorizations occur in all quadruples in $\cW_{\lambda}$.

\begin{obs} \label{o:lslf}
Let $(m_{11}, m_{12}; m_{21}, m_{22})$ be a quadruple inside $\cW_{\lambda}$. Let $\mu_{ij} := \Subsk(m_{ij})$. Then, there is a finishing-domino end transformation $l_f$ which sends $\Std(\mu_{11})$ to $\Std(\mu_{12})$ and there is a starting-domino end transformation $l_s$ which sends $\Std(\mu_{11})$ to $\Std(\mu_{21})$. These end transformations are necessarily unique. Furthermore, $l_f$ sends $\Std(\mu_{21})$ to $\Std(\mu_{22})$ and $l_s$ sends $\Std(\mu_{12})$ to $\Std(\mu_{22})$.
\end{obs}

If $(m_{11}, m_{12}; m_{21}, m_{22})$ is a quadruple inside $\cW_{\lambda}$. Then, we may consider the row labels $r_1, r_2$ of the rows containing $m_{11}, m_{21}$, respectively. We also consider the column labels $s_1, s_2$ of the columns containing $m_{11}$ and $m_{12}$, respectively. Then, for $i,j \in \{1,2\}$, $\Subsk(m_{ij})$ starts at segment $w_{r_i}$ and ends at segment $w_{s_j}$, inclusive. Therefore, among the four subskeleta, there is a unique one that is the shortest, and there is a unique one that is the longest.

\begin{defn} \label{d:descF}
Suppose that $\fq := (m_{11}, m_{12}; m_{21}, m_{22})$ is a quadruple inside $\cW_{\lambda}$. Let $\mu_{ij} := \Subsk(m_{ij})$, and let $\mu'$ and $\mu''$, respectively, be the shortest and the longest of the four subskeleta among $\mu_{ij}$. Then, we define $\cF(\fq)$ to be the subset of frozen variables of $Q_{\lambda}$ given by the following rule.

\begin{enumerate}
\item If $\fq$ is not split by any DRH axis, then $\cF(\fq)$ includes
\begin{enumerate}[label=(\roman*)]
\item the frozen coefficients $\Delta_i$ ($1 \le i \le l+1$) corresponding to the bends of $\mu'$,
\item the frozen coefficients $\Delta_i$ ($1 \le i \le l+1$) corresponding to the endpoints of $\mu''$, and
\item the frozen coefficients $\Delta_i$ ($1 \le i \le l+1$) corresponding to the endpoints of $\mu'$ that are also bends of $\mu''$.
\end{enumerate}

\item If $\fq$ is split by a DRH axis, then follow (1), but also include the following modifications.
\begin{enumerate}
\item If $\fq$ is split by the vertical DRH axis, then necessarily, (ii) includes $\Delta_{l+1}$. In this case, remove $\Delta_{l+1}$ and insert $a_{pq}$ in $\cF(\fq)$.
\item If $\fq$ is split by the horizontal DRH axis, then necessarily, (ii) includes $\Delta_1$. In this case, remove $\Delta_1$ and insert $a_{11}$ in $\cF(\fq)$.
\end{enumerate}
\end{enumerate}
\end{defn}

\begin{thm}[The Main Decomposition Theorem] \label{thm:ultiDodgson}
Let $\fq$ be a quadruple inside $\cW_{\lambda}$. Then,
\[
|\fq| = \prod_{\Delta \in \cF (\fq)} \Delta.
\]
\end{thm}

\begin{proof}
The identities to prove in the various cases of $\fq$ will be different. In this proof, we will show the identities for certain important cases. The rest can be argued analogously.

The collection $\cF(\fq)$ depends crucially on where $\fq$ is in $\cW_{\lambda}$. In particular, whether $\fq$ is cut by either or both of the DRH axes will affect the elements of $\cF(\fq)$.

The DRH axes split $\cW_{\lambda}$ into four quadrants. Let the center of $\fq$ have coordinate $(q_x, q_y)$ with respect to the DRH axes. We will introduce the following notation to bookkeep the casework: we say that $\fq$ is in the case $(i,j)^{\alpha \beta}$ where $i,j \in \mathbb{Z}$ and $\alpha, \beta \in \{-,0,+\}$, if $\deg(e(m_{12})) - \deg(e(m_{11})) = i$, $\deg(e(m_{21})) - \deg(e(m_{11})) = j$, and the signs of $q_x$ and $q_y$ are $\alpha$ and $\beta$, respectively.

for example, the case $(0,2)^{-+}$ refers to when the whole $\fq$ is to the left of the vertical DRH axis and is above the horizontal DRH axis, $\deg(e(m_{12})) = \deg(e(m_{11}))$, and $\deg(e(m_{21})) = \deg(e(m_{11})) + 2$.

In light of Observation \ref{o:lslf}, we use the notation $\mu_{ij}$ to denote $\Subsk(m_{ij})$. We will also use $l_s$ and $l_f$ to denote the end transformations which transform $\Std(\mu_{11}) \mapsto \Std(\mu_{21})$ and $\Std(\mu_{11}) \mapsto \Std(\mu_{12})$, respectively.

The number of possible $(i,j)^{\alpha \beta}$ is finite. It is not hard to notice the following:
\begin{itemize}
\item if $\alpha = -$, then $i \in \{0, +2\}$,
\item if $\alpha = 0$, then $i \in \{-1, +1\}$,
\item if $\alpha = +$, then $i \in \{-2, 0\}$,
\item if $\beta = -$, then $j \in \{-2, 0\}$,
\item if $\beta = 0$, then $j \in \{-1, +1\}$, and
\item if $\beta = +$, then $j \in \{0, +2\}$.
\end{itemize}

There are hence $36$ possibilities for $(i,j)^{\alpha \beta}$, each of which gives a certain algebraic identities we will prove. Most cases will be analogous and we will show only particular cases.

As we saw earlier, there is a unique shortest subskeleton and a unique longest one among the four $\mu_{ij}$. We let $\wt{\mu}$ denote the shortest subskeleton $\mu_{ij}$. Consider the lattice path $\lambda$ inside the initial $\lambda$-DRH again, as we did in Section \ref{ss:introboa}. We can now consider $\wt{\mu}$ inside $\lambda$ inside the DRH as well. Let $v_1, \dots, v_k$ denote all the bends in $\wt{\mu}$ (in that order). We will denote by $b_i \in \cB$ the $2 \times 2$-square in the DRH centered at $v_i$. 

Let $w$ and $w'$ denote the starting vertex and the final vertex of $\wt{\mu}$. Let $u$ and $u'$ denote the starting vertex and the final vertex of the longest subskeleton. Note that from $u$ to $w$, there may or may not be a bend. Nevertheless, since transformations at endpoints are end transformations, there can be at most one bend. If it exists, call it $v_0$. Analogously, if a bend between $w'$ and $u'$ exists, we call it $v'_0$. The $2 \times 2$-squares at $u, u', w, w', v_0, v'_0$ will be called $b_u, b'_u, b_w, b'_w, b_0, b'_0$, respectively. These notations will come in handy when we standardize many DRHs from subskeleta.

Recall that if $b$ is a block, we use $h_+(b), h_-(b) \in \cD^{H}$ to denote the upper and the lower $2 \times 1$ submatrices of $b$. Similarly, we use $v_+(b), v_-(b) \in \cD^{V}$ to denote the right and the left $1 \times 2$ submatrices of $b$.

\underline{Case $(0,0)^{++}$.} In this case, $\wt{\mu} := \mu_{12}$. Necessarily, $\wt{\mu}$ starts with $N$ and ends with $N$ (and thus $2|k$). We have $\mu_{11} = \wt{\mu} N$, $\mu_{22} = N \wt{\mu}$, and $\mu_{21} = N \wt{\mu} N$. Note also that both $l_s$ and $l_f$ are of type I. Write $s := h_-(b_w) \in \cD^{H}$, $f := h_+(b'_w) \in \cD^{H}$, $\wt{s} := h_-(b_u) \in \cD^{H}$, and $\wt{f} := h_+(b'_u) \in \cD^{H}$. We have
\begin{align*}
& \Std(\wt{\mu}) = s b_1 \cdots b_k f \\
& \Std(N \wt{\mu}) = \wt{s} b_1 \cdots b_k f \\
& \Std(\wt{\mu} N) = s b_1 \cdots b_k \wt{f} \\
& \Std(N \wt{\mu} N) = \wt{s} b_1 \cdots b_k \wt{f}.
\end{align*}
Note that all these four matrices have the same dimension $(k+2) \times (k+2)$. We then have
\[
e(m_{11}) = c(\mu_{11}) = (-1)^{\frac{(k+1)k}{2}} \cdot \det (s b_1 \cdots b_k \wt{f})
\]
and similar expressions for $e(m_{12})$, $e(m_{21})$, and $e(m_{22})$. We have obtained:
\[
|\fq| = \det(s b_1 \cdots b_k \wt{f}) \det(\wt{s} b_1 \cdots b_k f) - \det(s b_1 \cdots b_k f) \det(\wt{s} b_1 \cdots b_k \wt{f}).
\]
We want to show that $|\fq|$ equals the product of frozen variables in $\cF(\fq)$. Going back to Definition \ref{d:descF}, we find that $\cF(\fq)$ contains the frozen cluster variables that correspond to: (i) all the bends inside $\wt{\mu}$ and (ii) the endpoints of $N \wt{\mu} N$. Note that the endpoints of $\wt{\mu}$ are not bends in $N \wt{\mu} N$ and therefore we do not have contributions from (iii) from Definition \ref{d:descF}.

The frozen variables that correspond to the bends inside $\wt{\mu}$ are precisely $|b_1|, |b_2|, \dots, |b_k|$. The frozen variables that correspond to the endpoints of $N \wt{\mu} N$ are $|b_u| = |\wt{s} s|$ and $|b'_u| = |f \wt{f}|$. Therefore, it remains to show that
\begin{align*}
&\det(s b_1 \cdots b_k \wt{f}) \det(\wt{s} b_1 \cdots b_k f) - \det(s b_1 \cdots b_k f) \det(\wt{s} b_1 \cdots b_k \wt{f}) \\
&= |\wt{s} s| \cdot |b_1| \cdot \, \cdots \, \cdot |b_k| \cdot |f \wt{f}|.
\end{align*}
This is precisely Corollary \ref{c:detid00++}. Now we start to see how our hard work with algebraic identities pays off. The identities we proved will be directly useful in dealing with the cases in this proof.

\underline{Case $(-2,0)^{++}$.} We still have $\wt{\mu} = \mu_{12}$, $\wt{\mu}$ starts with $N$ and ends with $N$ forcing $2|k$. Now that the degree of $e(m_{11})$ is two greater than that of $e(m_{12})$. There is a positive integer $\sigma$ such that $\mu_{11} = \wt{\mu} E^{\sigma} N$. We also have $\mu_{22} = N \wt{\mu}$ and $\mu_{21} = N \wt{\mu} E^{\sigma} N$. The end transformation $l_s$ in this case is of Type I, while $l_f$ is of Type III. As in the previous case, let $s := h_-(b_w)$, $f := h^+(b'_w)$, $\wt{s} = h_-(b_u)$, and $\wt{f} := h^+(b'_u)$. We have
\begin{align*}
& \Std(\wt{\mu}) = s b_1 \cdots b_k f \\
& \Std(N \wt{\mu}) = \wt{s} b_1 \cdots b_k f \\
& \Std(\wt{\mu} E^{\sigma} N) = s b_1 \cdots b_k b'_w b'_0 \wt{f} \\
& \Std(N \wt{\mu} E^{\sigma} N) = \wt{s} b_1 \cdots b_k b'_w b'_0 \wt{f}.
\end{align*}
Therefore,
\[
|\fq| = \det(s b_1 \cdots b_k f) \det( \wt{s} b_1 \cdots b_k b'_w b'_0 \wt{f}) - \det(\wt{s} b_1 \cdots b_k f) \det( s b_1 \cdots b_k b'_w b'_0 \wt{f}).
\]
On the other hand, we have $\cF(\fq) = \{b_1, \dots, b_k, b_u, b'_u, b'_w\}$. It remains to show that
\begin{align*}
& \det(s b_1 \cdots b_k f) \det( \wt{s} b_1 \cdots b_k b'_w b'_0 \wt{f}) - \det(\wt{s} b_1 \cdots b_k f) \det( s b_1 \cdots b_k b'_w b'_0 \wt{f}) \\
&= |\wt{s} s| \cdot |b_1| \cdot \, \cdots \, \cdot |b_k| \cdot |b'_w| \cdot |h_+(b'_0) \wt{f}|.
\end{align*}
This is precisely Corollary \ref{c:detid-20++}. The work in all other cases is analogous in the sense that we follow the following procedure: first, write down the four matrices $\Std(\mu_{ij})$ (and $e(m_{ij})$); second, notice which types of end transformations $l_s$ and $l_f$ are; third, decompose the determinant $|\fq|$ using the generalization of Dodgson's identity in Lemma \ref{lem:chunk}.

The two examples we showed above concern end transformations of Types I and III. It will be interesting to see how decomposition of $|\fq|$ behaves for end transformations of Type II. Therefore, we will show another case.

\underline{Case $(-1,+2)^{0+}$.} The $2 \times 2$-square $\fq$ is cut in half by the vertical DRH axis, but is above the horizontal DRH axis. As a result, all $\mu_{ij}$ starts with $N$, while $\mu_{11}$ and $\mu_{21}$ end with $E$ but $\mu_{12}$ and $\mu_{22}$ end with $N$. The shortest subskeleton is $\wt{\mu} = \mu_{12}$. Let $s = h_-(b_w)$, $f = h^+(b'_w)$, $\wt{s} = h_-(b_u)$, and $\varphi = v_+(b'_u)$. Note that we use $\varphi$ instead of $f$ to indicate that the shape of the finishing domino of the longest subskeleton is different from that of $\wt{\mu}$. We have
\begin{align*}
& e(m_{11}) = (-1)^{\frac{(k+2)(k+1)}{2}} \cdot \det(s b_1 \cdots b_k b'_w \varphi) \\
& e(m_{12}) = (-1)^{\frac{(k+1)k}{2}} \cdot \det( s b_1 \cdots b_k f) \\
& e(m_{21}) = (-1)^{\frac{(k+4)(k+3)}{2}} \cdot \det( \wt{s} b_0 b_w b_1 \cdots b_k b'_w \varphi) \\
& e(m_{22}) = (-1)^{\frac{(k+3)(k+2)}{2}} \cdot \det(\wt{s} b_0 b_w b_1 \cdots b_k f).
\end{align*}
Note that $k$ is still even in this case. We have
\[
|\fq| = \det(s b_1 \cdots b_k b'_w \varphi) \det(\wt{s} b_0 b_w b_1 \cdots b_k f) - \det(sb_1 \cdots b_k f) \det(\wt{s} b_0 b_w b_1 \cdots b_k \varphi).
\]
Note that the end transformation $l_s$ is of Type III, while $l_f$ is of Type II. From Definition \ref{d:descF}, we have $\cF(\fq) = \{b_1, \dots, b_k, b_u, a_{pq}, b_w, b'_w\}$. Again, the equality $|\fq| = \prod_{\Delta \in \cF(\fq)} \Delta$ follows directly from Lemma \ref{lem:chunk}. We proved this identity specifically in Corollary \ref{c:detid-1+20+}.
\end{proof}

\begin{cor} \label{c:detiszero}
Suppose that $M = [m_{ij}]_{1 \le i,j \le 3}$ is a $3 \times 3$ square inside $\cW_{\lambda}$. Then,
\[
\det \begin{bmatrix}
e(m_{11}) & e(m_{12}) & e(m_{13}) \\
e(m_{21}) & e(m_{22}) & e(m_{23}) \\
e(m_{31}) & e(m_{32}) & e(m_{33})
\end{bmatrix} = 0.
\]
\end{cor}
\begin{proof}
By Dodgson condensation, it suffices to show that
\[
|\fq_{11}| \cdot |\fq_{22}| = |\fq_{12}| \cdot |\fq_{21}|
\]
where $\fq_{ij}$ denotes the quadruple $(m_{i,j}, m_{i,j+1};m_{i+1,j}, m_{i+1,j+1})$. By the main decomposition theorem, we need to show that
\[
\Big( \prod_{\Delta \in \cF(\fq_{11})} \Delta \Big) \cdot \Big( \prod_{\Delta \in \cF(\fq_{22})} \Delta \Big) = \Big( \prod_{\Delta \in \cF(\fq_{12})} \Delta \Big) \cdot \Big( \prod_{\Delta \in \cF(\fq_{21})} \Delta \Big). \tag{$\ast$}
\]
We will introduce some notations for this proof. If $\mu$ is a subskeleton of $\lambda$, let $E(\mu)$ denote the product of the two frozen variables corresponding to the endpoints of $\mu$, let $B(\mu)$ denote the product of the frozen variables corresponding to the bends inside $\mu$, and let $I(\mu)$ denote the product of the frozen variables corresponding to the endpoints of $\mu$ that are also bends in $\lambda$.

Note that if $i,j,k,l \in \{1,2,3\}$, then one of the following four cases is true
\begin{itemize}
\item $\mu_{ik} \cap \mu_{jl} = \mu_{il}$ and $\mu_{ik} \cup \mu_{jl} = \mu_{jk}$,
\item $\mu_{ik} \cap \mu_{jl} = \mu_{jk}$ and $\mu_{ik} \cup \mu_{jl} = \mu_{il}$,
\item $\mu_{il} \cap \mu_{jk} = \mu_{ik}$ and $\mu_{il} \cup \mu_{jk} = \mu_{jl}$, or
\item $\mu_{il} \cap \mu_{jk} = \mu_{jl}$ and $\mu_{il} \cup \mu_{jk} = \mu_{ik}$
\end{itemize}
depending on which of the four subskeleta is the longest. In every case, the following relations always hold:
\begin{align*}
& B(\mu_{ik}) B(\mu_{jl}) = B(\mu_{il}) B(\mu_{jk}) \\
& E(\mu_{ik}) E(\mu_{jl}) = E(\mu_{il}) E(\mu_{jk}) \\
& I(\mu_{ik}) I(\mu_{jl}) = I(\mu_{il}) I(\mu_{jk}).
\end{align*}
For each $i,j \in \{1,2\}$, let $\mu_{r_{ij},s_{ij}}$ and $\mu_{R_{ij},S_{ij}}$ denote the shortest and the longest, respectively, subskeleta among $\mu_{i,j}, \mu_{i,j+1}, \mu_{i+1,j}, \mu_{i+1,j+1}$. Recall that, for each $i=1,2$, there is a starting end transformation $l_{s_i}$ that changes $\mu_{i,j}$ to $\mu_{i+1,j}$, for every $j$. Analogously, for each $j=1,2$, there is a finishing end transformation $l_{f_j}$ that changes $\mu_{i,j}$ to $\mu_{i,j+1}$, for every $i$. This implies that necessarily $r_{11} = r_{12} =: r_1$, $r_{21} = r_{22} =: r_2$, $s_{11} = s_{21} =: s_1$, and $s_{12} = s_{22} =: s_2$. We can define $R_i$ and $S_j$ similarly.

For each quadruple $\fq$, define
\[
\chi^X(\fq) = \begin{cases}
\frac{a_{11}}{\Delta_1} & \text{if } \fq \text{ is split by the horizontal DRH axis,} \\
1 & \text{otherwise.}
\end{cases}
\]
and
\[
\chi^Y(\fq) = \begin{cases}
\frac{a_{pq}}{\Delta_{l+1}} & \text{if } \fq \text{ is split by the vertical DRH axis,} \\
1 & \text{otherwise.}
\end{cases}
\]

The main decomposition theorem says that
\begin{align*}
\prod_{\Delta \in \cF(\fq_{ij})} \Delta &= B(\mu_{r_{ij},s_{ij}}) E(\mu_{R_{ij},S_{ij}}) I(\mu_{r_{ij},s_{ij}}) \chi^X(\fq_{ij}) \chi^Y(\fq_{ij}) \\
&= B(\mu_{r_{i},s_{j}}) E(\mu_{R_{i},S_{j}}) I(\mu_{r_{i},s_{j}}) \chi^X(\fq_{ij}) \chi^Y(\fq_{ij}).
\end{align*}
Note that $\chi^X(\fq_{ij}) = \chi^X(\fq_{ij'})$ and $\chi^Y(\fq_{ij}) = \chi^Y(\fq_{i'j})$. Thus,
\[
\chi^X(\fq_{11})\chi^Y(\fq_{11})\chi^X(\fq_{22})\chi^Y(\fq_{22}) = \chi^X(\fq_{12})\chi^Y(\fq_{12})\chi^X(\fq_{21})\chi^Y(\fq_{21}).
\]

From ($\ast$), it remains to show that
\begin{align*}
& B(\mu_{r_1,s_1}) E(\mu_{R_1,S_1}) I(\mu_{r_1,s_1}) B(\mu_{r_2,s_2}) E(\mu_{R_2,S_2}) I(\mu_{r_2,s_2}) \\
& = B(\mu_{r_1,s_2}) E(\mu_{R_1,S_2}) I(\mu_{r_1,s_2}) B(\mu_{r_2,s_1}) E(\mu_{R_2,S_1}) I(\mu_{r_2,s_1}),
\end{align*}
which is indeed true because of the relations we noted above for the indices $(i,j,k,l) = (r_1, r_2, s_1, s_2)$ and $(R_1, R_2, S_1, S_2)$. We have finished the proof.
\end{proof}

\medskip

\bigskip

\section{The main theorem} 
In this section, we state and prove the main theorem.

\begin{restatable}{thm}{main} 
\label{thm:main}
Let $\lambda$ be a North-East lattice path. The set of mutable cluster variables of the DRH algebra $\cA_{\lambda}$ is
\[
\{e(m): m \text{ is a cell in the }\lambda\text{-DRH staircase.}\}.
\]
\end{restatable}

To prove the main theorem, we will describe the quiver $Q_w$ for every worm $w$ inside the DRH staircase. By construction, the proposed quiver for the initial worm $\cM(\lambda)$ will agree with the initial DRH quiver $Q_\lambda$. We made a remark earlier that any worm can be transformed to any other worm via a sequence of worm operations. Therefore, to prove the main theorem, it suffices to show that our proposed quivers are compatible with worm operations everywhere in the DRH staircase.

In an arbitrary worm $w$ inside the DRH staircase, the frozen arrows to each mutable vertex can get complicated. Instead of always talking about frozen arrows, we may use {\em frozen coefficients} at mutable vertices instead.

\medskip

\subsection{Frozen coefficients for quivers} Frozen coefficients are defined at mutable vertices of a given quiver.

\begin{defn}
Let $Q$ be a quiver. Suppose that $v$ is a mutable vertex in $Q$. Let $f_1, f_2, \dots, f_n$ denote the frozen variables of $Q$. For each $i$, we denote by $m_i$, the multiplicity of frozen arrows from $v$ to $f_i$. Here, an out-arrow from $v$ to $f_i$ is counted as $+1$ arrow, while an in-arrow from $f_i$ to $v$ is counted as $-1$ arrow. Then, the {\em frozen coefficient} at $v$ in $Q$ is defined to be
\[
\fc_Q(v) := \prod_{i=1}^n f_i^{m_i}.
\]
\end{defn}

Frozen coefficients are useful because they record the data of all frozen arrows at a certain mutable vertex in one expression. When the frozen variables are algebraically independent, the data of frozen coefficients at all mutable vertices and the data of arrows between the mutable vertices are sufficient to recover the whole quiver.

\begin{lemma} \label{l:inmuta}
Let $Q$ be a quiver. Let $A$ and $B$ be mutable vertices of $Q$ such that there is only one arrow from $B$ to $A$ between the two vertices. Suppose also that all the frozen arrows at $A$ are out-arrows. Let $\wh{Q}$ be the resulting quiver after mutating $Q$ at $A$. Then, we have the following equations.
\[
\fc_{\wh{Q}}(A) = \frac{1}{\fc_Q(A)}
\]
and
\[
\fc_{\wh{Q}}(B) = \fc_Q(A) \cdot \fc_Q(B).
\]
\end{lemma}
\begin{proof}
The first equation is clear because the mutation at $A$ flips all the frozen arrows at $A$. For the second equation, consider any frozen vertex $F$. If $F$ is not connected to $A$, then mutating at $A$ does not affect the frozen arrows at $F$. If $F$ is connected to $A$, then by assumption there are $m \ge 1$ arrows from $A$ to $F$. Mutating at $A$ adds $m$ arrows from $B$ to $F$. By considering all the frozen vertices in this way, we conclude that $\fc_{\wh{Q}}(B) = \fc_Q(A) \cdot \fc_Q(B)$ as desired.
\end{proof}

Indeed, the following analogous result holds when the arrows go in the opposite direction.

\begin{cor} \label{c:outmuta}
Let $Q$ be a quiver. Let $A$ and $B$ be mutable vertices of $Q$ such that there is only one arrow from $A$ to $B$ between the two vertices. Suppose also that all the frozen arrows at $A$ are in-arrows. Let $\wh{Q}$ be the resulting quiver after mutating $Q$ at $A$. Then, we have the following equations.
\[
\fc_{\wh{Q}}(A) = \frac{1}{\fc_Q(A)}
\]
and
\[
\fc_{\wh{Q}}(B) = \fc_Q(A) \cdot \fc_Q(B).
\]
\end{cor}

\medskip

\subsection{Frozen coefficients for lattice points inside the DRH staircase} \label{ss:FCdef}
In this section, we will describe how to associate a frozen coefficient to each lattice point inside the DRH staircase. These {\em frozen coefficients} are different from ones we associate to mutable vertices of quivers. We will establish their connections later. The reason we introduce them is so that the description of quivers $Q_w$ for worms can be given more easily. Like frozen coefficients for mutable vertices in quivers, these coefficients are Laurent monomials in the frozen variables of $Q_\lambda$.

A lattice point inside the DRH staircase refers to the center of any $2 \times 2$-square inside the staircase. Our description of frozen coefficients for the points will be given in three steps. Through combinatorial descriptions below, we will associate a frozen coefficient $\FC(u)$ to each lattice point $u$ in each of the three cases: (1) that is also inside the DRH or the DRH transpose, (2) that is also inside $\cW_\lambda$, and (3) that is on either the northwest boundary or the southeast boundary of $\cW_\lambda$.

In almost all the cases, the frozen coefficient $\FC(u)$ is defined to be the determinant of the $2 \times 2$ square centered at $u$. Namely, if $u$ is the center of $[m_{ij}]_{1 \le i, j \le 2}$, then we define
\[
\FC(u) := \begin{vmatrix}
e(m_{11}) & e(m_{12}) \\
e(m_{21}) & e(m_{22})
\end{vmatrix}.
\]
As we will see in the steps below, this is the case as long as the determinant is a homogeneous polynomial; in other words,
\[
\deg(e(m_{11})) + \deg(e(m_{22})) = \deg(e(m_{12})) + \deg(e(m_{21})).
\]
The exceptional cases are Cases 3.1 and 3.1' where there are $3$ of the $\deg(e(m_{ij}))$ are $1$ and the other one is $3$. Only in these cases do we give a special rule for the coefficients. We call these lattice points with nonhomogeneous determinants {\em special points}.

\underline{Step 1.} Let $u$ be the center of a $2 \times 2$-square inside the original DRH or the transpose DRH. Note that the determinant of the $2 \times 2$ square at $u$ is a frozen variable $\Delta_i$ by definition. Then, to $u$, associate $\FC(u) := \Delta_i$.

It will also be convenient to associate the following frozen coefficients to these extra points:
\begin{itemize}
\item $\Delta_1$ to the upper right corner of $c_{11}$,
\item $\Delta_0 := a_{11}$ to the lower left corner of $c_{11}$,
\item $\Delta_{l+1}$ to the lower left corner of $c_{pq}$, and
\item $\Delta_{l+2} := a_{pq}$ to the upper right corner of $c_{pq}$.
\end{itemize}
These exceptional lattice points are not inside the staircase, but having frozen coefficients there turns out to be convenient in Step 3. Do analogous extra associations for the transpose DRH so that the frozen coefficients conform with the DRH under transposition. For convenience, we will colloquially call the lattice points to which we associate frozen coefficients in this step {\em Step-1 lattice points}.

\underline{Step 2.} Suppose that $u$ is a lattice point inside $\cW_\lambda$. Then, let the quadruple centered at $u$ be denoted by $\fq := (m_{11}, m_{12}; m_{21}, m_{22})$. As before, let $\mu_{ij} := \Subsk(m_{ij})$ be the corresponding subskeleton of $\lambda$ at the cell $m_{ij}$. By the main decomposition theorem (Theorem \ref{thm:ultiDodgson}), the determinant $c(\mu_{11}) c(\mu_{22}) - c(\mu_{12}) c(\mu_{21})$ is a product of frozen variables of $Q_\lambda$. We define the frozen coefficient $\FC(u)$ at $u$ to be this product:
\[
\FC(u) := \prod_{\Delta \in \cF(\fq)} \Delta.
\]

\underline{Step 3.} Let $u$ be a lattice point on the boundary of $\cW_\lambda$. As above, let $m_{11}$, $m_{12}$, $m_{21}$, and $m_{22}$ be four cells around $u$. In this step, some of $m_{ij}$ will be inside the DRH (or DRH transpose), while some will be in $\cW_\lambda$.

\underline{Case 3.1.} Suppose that $m_{11}$, $m_{12}$, and $m_{21}$ are in the DRH, but $m_{22}$ is in $\cW_{\lambda}$. This case happens when the DRH has a bend at $u$. If we let $u_{11}$, $u_{12}$, $u_{21}$, and $u_{22} = u$, respectively, be the upper-left, upper-right, lower-left, and lower-right corners of $m_{11}$. Then, the frozen coefficients at $u_{11}$, $u_{12}$, and $u_{21}$ are described in Step 1. We associate
\[
\FC(u) := \frac{\FC(u_{12}) \FC(u_{21})}{\FC(u_{11})}
\]
to the lattice point $u$. Necessarily, $\FC(u)$ is in the form $\frac{\Delta_i \Delta_{i+2}}{\Delta_{i+1}}$ in this case.

\underline{Case 3.1'.} Suppose that $m_{12}$, $m_{21}$, and $m_{22}$ are in the transpose DRH, but $m_{11}$ is in $\cW_\lambda$. This case is analogous to Case 1. We define $\FC(u)$ similarly: if $v_{11} = u$, $v_{12}$, $v_{21}$, $v_{22}$ are the corners of $m_{22}$, then define $\FC(u) := \frac{\FC(v_{12}) \FC(v_{21})}{\FC(v_{22})}$.

\underline{Case 3.2.} Suppose that $m_{11}$ and $m_{12}$ are in the original DRH, while $m_{21}$ and $m_{22}$ are in $\cW_{\lambda}$. If we start from $u$ and go up along the vertical line $u$ is in, the first Step-1 lattice point we meet is the point one unit right above $u$. Call this point $v$. On the other hand, if we start from $u$ and go left, the first Step-1 lattice point we meet is at least two units away from $u$. Call this point $w$. The point $w$ may be the exceptional points we associated frozen coefficients to, which may not be in the DRH skeleton. This is why we associate frozen variables to extra points in Step 1. In this case, define
\[
\FC(u) := \FC(v) \cdot \FC(w).
\]
This coefficient is necessarily of the form $\Delta_i \Delta_j$, for some $i,j = 0,1, \dots, l+2$.

\underline{Case 3.2'} Analogous to Case 3.2 is when $m_{21}$ and $m_{22}$ are in the transpose DRH, while $m_{11}$ and $m_{12}$ are in $\cW_\lambda$. Starting from $u$, instead of going up and left to find $v$ and $w$, we go down and right to find $v$ and $w$. Then, define $\FC(u) := \FC(v) \cdot \FC(w)$.

\underline{Case 3.3} Suppose that $m_{11}$ and $m_{21}$ are in the original DRH, while $m_{12}$ and $m_{22}$ are in $\cW_{\lambda}$. This case is similar to Case 3.2. We perform the same procedure of starting at $u$, going up until we find a Step-1 lattice point $v$, going left until we find a Step-1 lattice point $w$. In this case, $w$ is one unit away from $u$, while $v$ is farther away. Define $\FC(u) := \FC(v) \cdot \FC(w)$.

\underline{Case 3.3'} Analogous to Case 3.3 is when $m_{12}$ and $m_{22}$ are in the transpose DRH, while the other two are in $\cW_\lambda$. Starting from $u$, we go down until we find a Step-1 lattice point $v$, and go right until we find a Step-1 lattice point $w$. Define $\FC(u) := \FC(v) \cdot \FC(w)$.

\underline{Case 3.4} Suppose that only $m_{11}$ is in the original DRH, while the others are in $\cW_{\lambda}$. Let $x$ be the upper left corner of $m_{11}$. Note that $x$ is a Step-1 lattice point. Similar to the previous cases, we start from $u$ and go up until we meet a Step-1 lattice point $v$. Starting from $u$ again, we go left until we meet a Step-1 lattice point $w$. Define
\[
\FC(u) := \FC(v) \FC(w) \FC(x).
\]
Necessarily, $\FC(u)$ is in the form $\Delta_i \Delta_j \Delta_k$ for some $i <j <k$.

\underline{Case 3.4'} Suppose that only $m_{22}$ is in the transpose DRH. Starting at $u$, we go down until we meet a Step-1 point $v$, and go right until we meet a Step-1 point $w$. Let $x$ be the lower-right corner of $m_{22}$. Define $\FC(u) := \FC(v) \FC(w) \FC(x)$.

We note that in all steps, except Cases 3.1 and 3.1' (the special points), the frozen coefficients we associate to the lattice points are monomials of $\Delta_0, \Delta_1, \dots, \Delta_{l+2}$ with the degree of $\Delta_i$ being either $0$ or $1$ for each $i$. Only in Cases 3.1 and 3.1' do we associate a frozen coefficient of the form $\frac{\Delta_i \Delta_j}{\Delta_k}$ to the lattice point. See Table \ref{t:FCsummary} for a summary.

\begin{center}
\setlength{\tabcolsep}{10 pt}
\def\arraystretch{1.5}
\begin{table}
\begin{tabular}{|c|c|}
\hline
non-special lattice points $v$ & special lattice points $v$ \\ [0.5 ex]
\hline
\hline
Exactly $1$ cell in $\fq$ is in $\cW_{\lambda}$ & Either $0$, $2$, $3$, or $4$ cells in $\fq$ are in $\cW_{\lambda}$. \\
\hline
$\FC(v)$ is a monomial in $\Delta_i$'s & $\FC(v)$ is in the form $\frac{\Delta_i \Delta_{i+2}}{\Delta_{i+1}}$ \\
\hline
$\FC(v) = \det(\fq)$ & $\FC(v) \neq \det(\fq)$ \\
\hline
$\det(\fq)$ is homogeneous. & $\det(\fq)$ is non-homogeneous. \\
\hline
\end{tabular}
\bigskip
\caption{A summary of frozen coefficients at lattice points $v$ in the DRH staircase. Here, $\fq$ denotes the $2 \times 2$-square centered at $v$.} \label{t:FCsummary}
\end{table}
\end{center}

An example of frozen coefficients for the lattice points inside $NENNEE$-DRH staircase is given in Figure \ref{fig:NENNEE_LFC}. Notice the resemblance between our staircase diagram and Auslander-Reiten quivers.

\begin{center}
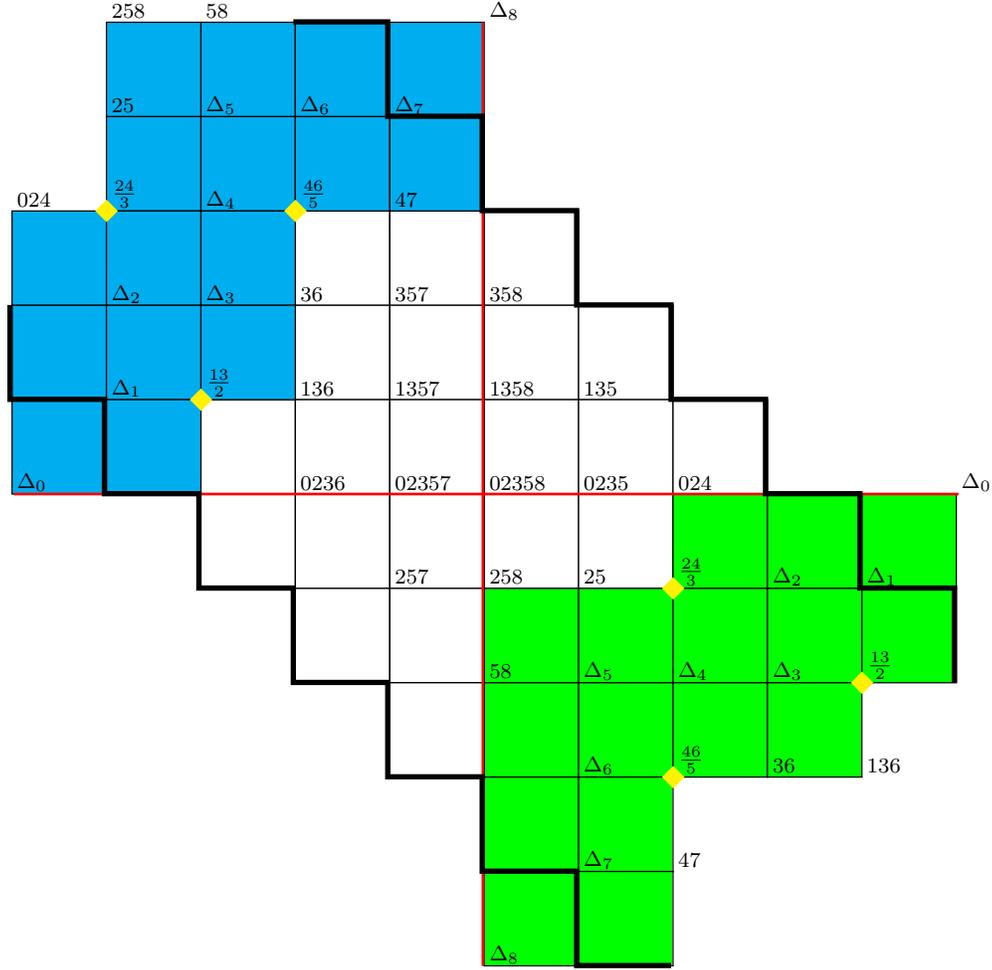
\begin{figure}
\begin{tikzpicture}[scale = 1.5]
\ytableausetup{notabloids}
\ytableausetup{mathmode, boxsize=3.0em}
\node (n) {\ytableausetup{nosmalltableaux}
\ytableausetup{notabloids}
\ydiagram[*(cyan)]{1+4,1+4,0+3,0+3,0+2}
*[*(green)]{0,0,0,0,0,7+3,5+5,5+4,5+2,5+2}
*[*(white)]{0,0,3+3,3+4,2+6,2+5,3+2,4+1}};

\def \a {0.837};

\def \zxs {-3.27};
\def \zys {0.42};

\begin{scope}[shift={(-0.84,+0.42)}]
\draw[red, line width = 1 pt] ({-1.56-2*\a},-0.42) -- ({-1.56+8*\a},-0.42);
\draw[red, line width = 1 pt] ({-1.59+3*\a},{-0.42+5*\a}) -- ({-1.59+3*\a},{-0.42-5*\a});

\draw[line width= 2.0 pt,black] (\zxs,\zys+\a) -- (\zxs, \zys) -- (\zxs+\a, \zys) -- (\zxs+\a, \zys-\a) -- (\zxs+2*\a, \zys-\a) -- (\zxs+2*\a, \zys-2*\a) -- (\zxs+3*\a, \zys-2*\a) -- (\zxs+3*\a, \zys-3*\a) -- (\zxs+4*\a, \zys-3*\a) -- (\zxs+4*\a, \zys-4*\a) -- (\zxs+5*\a, \zys-4*\a) -- (\zxs+5*\a, \zys-5*\a) -- (\zxs+6*\a, \zys-5*\a) -- (\zxs+6*\a, \zys-6*\a) -- (\zxs+7*\a, \zys-6*\a);
\end{scope}

\begin{scope}[shift = {(-0.84+3*\a,+0.42+4*\a)}]
\draw[line width= 2.0 pt,black] (\zxs, \zys) -- (\zxs+\a, \zys) -- (\zxs+\a, \zys-\a) -- (\zxs+2*\a, \zys-\a) -- (\zxs+2*\a, \zys-2*\a) -- (\zxs+3*\a, \zys-2*\a) -- (\zxs+3*\a, \zys-3*\a) -- (\zxs+4*\a, \zys-3*\a) -- (\zxs+4*\a, \zys-4*\a) -- (\zxs+5*\a, \zys-4*\a) -- (\zxs+5*\a, \zys-5*\a) -- (\zxs+6*\a, \zys-5*\a) -- (\zxs+6*\a, \zys-6*\a) -- (\zxs+7*\a, \zys-6*\a) -- (\zxs+7*\a, \zys-7*\a);
\end{scope}

\def \xs {0.042};
\def \ys {0.1};
\def \frl {0.05};

\begin{scope}
\node[anchor = west] at (\xs-4*\a,\ys+5*\a) {\tiny $258$};
\node[anchor = west] at (\xs-3*\a,\ys+5*\a) {\tiny $58$};
\node[anchor = west] at (\xs,\ys+5*\a) {\tiny $\Delta_8$};

\node[anchor = west] at (\xs-4*\a,\ys+4*\a) {\tiny $25$};
\node[anchor = west] at (\xs-3*\a,\ys+4*\a) {\tiny $\Delta_5$};
\node[anchor = west] at (\xs-2*\a,\ys+4*\a) {\tiny $\Delta_6$};
\node[anchor = west] at (\xs-\a,\ys+4*\a) {\tiny $\Delta_7$};

\node[anchor = west] at (\xs-5*\a,\ys+3*\a) {\tiny $024$};
\node[anchor = west] at (\xs-4*\a,\ys+3*\a+\frl) {\tiny $\frac{24}{3}$};
\node[anchor = west] at (\xs-3*\a,\ys+3*\a) {\tiny $\Delta_4$};
\node[anchor = west] at (\xs-2*\a,\ys+3*\a+\frl) {\tiny $\frac{46}{5}$};
\node[anchor = west] at (\xs-\a,\ys+3*\a) {\tiny $47$};

\node[anchor = west] at (\xs-4*\a,\ys+2*\a) {\tiny $\Delta_2$};
\node[anchor = west] at (\xs-3*\a,\ys+2*\a) {\tiny $\Delta_3$};
\node[anchor = west] at (\xs-2*\a,\ys+2*\a) {\tiny $36$};
\node[anchor = west] at (\xs-\a,\ys+2*\a) {\tiny $357$};
\node[anchor = west] at (\xs,\ys+2*\a) {\tiny $358$};

\node[anchor = west] at (\xs-4*\a,\ys+\a) {\tiny $\Delta_1$};
\node[anchor = west] at (\xs-3*\a,\ys+\a+\frl) {\tiny $\frac{13}{2}$};
\node[anchor = west] at (\xs-2*\a,\ys+\a) {\tiny $136$};
\node[anchor = west] at (\xs-\a,\ys+\a) {\tiny $1357$};
\node[anchor = west] at (\xs,\ys+\a) {\tiny $1358$};
\node[anchor = west] at (\xs+\a,\ys+\a) {\tiny $135$};

\node[anchor = west] at (\xs-5*\a,\ys) {\tiny $\Delta_0$};
\node[anchor = west] at (\xs-2*\a,\ys) {\tiny $0236$};
\node[anchor = west] at (\xs-\a,\ys) {\tiny $02357$};
\node[anchor = west] at (\xs,\ys) {\tiny $02358$};
\node[anchor = west] at (\xs+\a,\ys) {\tiny $0235$};
\node[anchor = west] at (\xs+2*\a,\ys) {\tiny $024$};
\node[anchor = west] at (\xs+5*\a,\ys) {\tiny $\Delta_0$};

\node[anchor = west] at (\xs-\a,\ys-\a) {\tiny $257$};
\node[anchor = west] at (\xs,\ys-\a) {\tiny $258$};
\node[anchor = west] at (\xs+\a,\ys-\a) {\tiny $25$};
\node[anchor = west] at (\xs+2*\a,\ys-\a+\frl) {\tiny $\frac{24}{3}$};
\node[anchor = west] at (\xs+3*\a,\ys-\a) {\tiny $\Delta_2$};
\node[anchor = west] at (\xs+4*\a,\ys-\a) {\tiny $\Delta_1$};

\node[anchor = west] at (\xs,\ys-2*\a) {\tiny $58$};
\node[anchor = west] at (\xs+\a,\ys-2*\a) {\tiny $\Delta_5$};
\node[anchor = west] at (\xs+2*\a,\ys-2*\a) {\tiny $\Delta_4$};
\node[anchor = west] at (\xs+3*\a,\ys-2*\a) {\tiny $\Delta_3$};
\node[anchor = west] at (\xs+4*\a,\ys-2*\a+\frl) {\tiny $\frac{13}{2}$};

\node[anchor = west] at (\xs+\a,\ys-3*\a) {\tiny $\Delta_6$};
\node[anchor = west] at (\xs+2*\a,\ys-3*\a+\frl) {\tiny $\frac{46}{5}$};
\node[anchor = west] at (\xs+3*\a,\ys-3*\a) {\tiny $36$};
\node[anchor = west] at (\xs+4*\a,\ys-3*\a) {\tiny $136$};

\node[anchor = west] at (\xs+\a,\ys-4*\a) {\tiny $\Delta_7$};
\node[anchor = west] at (\xs+2*\a,\ys-4*\a) {\tiny $47$};

\node[anchor = west] at (\xs,\ys-5*\a) {\tiny $\Delta_8$};

\node[scale = 0.4, diamond, fill=yellow,minimum width=1 em] at (-2.417,+\a) {\hphantom{o}};
\node[scale = 0.4, diamond, fill=yellow,minimum width=1 em] at (-2.417+\a,+3*\a) {\hphantom{o}};
\node[scale = 0.4, diamond, fill=yellow,minimum width=1 em] at (-2.417+5*\a,-\a) {\hphantom{o}};

\node[scale = 0.4, diamond, fill=yellow,minimum width=1 em] at (-2.417-\a,+3*\a) {\hphantom{o}};
\node[scale = 0.4, diamond, fill=yellow,minimum width=1 em] at (-2.417+5*\a,-3*\a) {\hphantom{o}};
\node[scale = 0.4, diamond, fill=yellow,minimum width=1 em] at (-2.417+7*\a,-2*\a) {\hphantom{o}};
\end{scope}

\end{tikzpicture}

\caption{Frozen coefficients for lattice points inside $NENNEE$-DRH staircase. Here, the digit $i$ denotes the frozen variable $\Delta_i$. For example, $\frac{24}{3}$ denotes $\frac{\Delta_2 \Delta_4}{\Delta_3}$. Recall that $\Delta_0 := a_{11}$ and $\Delta_{l+2} = a_{pq}$. Special lattice points are marked with diamonds.} \label{fig:NENNEE_LFC}
\end{figure}
\end{center}

The special points are exactly the centers of $2 \times 2$-squares inside the DRH staircase whose determinant is non-homogeneous. Inside $\cW_{\lambda}$, every $2 \times 2$-square has homogeneous determinant. This reflects the fact that the special points only occur on the boundary between $\cW_{\lambda}$ and the DRH (or the DRH transpose). Away from the special lattice points, there is a generalization to Corollary \ref{c:detiszero}.

\begin{prop} \label{p:gendet}
Suppose that $[c_{ij}]_{1 \le i,j \le 3}$ is a $3 \times 3$ square inside the DRH staircase. If all the four vertices of $c_{22}$ are non-special, then
\[
\det \begin{bmatrix}
e(c_{11}) & e(c_{12}) & e(c_{13}) \\
e(c_{21}) & e(c_{22}) & e(c_{23}) \\
e(c_{31}) & e(c_{32}) & e(c_{33})
\end{bmatrix} = 0.
\]
\end{prop}
\begin{proof}
The case when the $3\times 3$-square is in $\cW_{\lambda}$ was done in Corollary \ref{c:detiszero}. From now, suppose that the $3 \times 3$-square has some cells from the DRH or the DRH transpose. Recall from Corollary \ref{c:notbothD&Dt} that the $3\times 3$-square cannot simultaneously have both cells from the DRH and the DRH transpose. By symmetry, we may assume that it contains a cell from the DRH. As long as all the four vertices of $c_{22}$, the middle cell, are non-special, there are eight following cases.

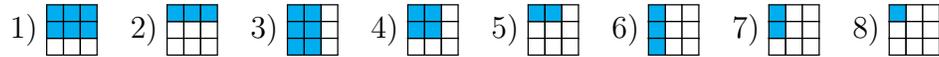
\begin{center}
\begin{tikzpicture}
\def \a {0.0};
\def \b {+1.6};

\begin{scope}[shift={(\a,0)}]
\ytableausetup{notabloids}
\ytableausetup{mathmode, boxsize=0.5 em}
\node (n) {\ytableausetup{nosmalltableaux}
\ytableausetup{notabloids}
\ydiagram[*(cyan)]{3,3,0}
*[*(white)]{0,0,3}};
\node at (-0.5,0) {1)};
\end{scope}

\begin{scope}[shift={(\a+\b,0)}]
\ytableausetup{notabloids}
\ytableausetup{mathmode, boxsize=0.5 em}
\node (n) {\ytableausetup{nosmalltableaux}
\ytableausetup{notabloids}
\ydiagram[*(cyan)]{3,0,0}
*[*(white)]{0,3,3}};
\node at (-0.5,0) {2)};
\end{scope}

\begin{scope}[shift={(\a+2*\b,0)}]
\ytableausetup{notabloids}
\ytableausetup{mathmode, boxsize=0.5 em}
\node (n) {\ytableausetup{nosmalltableaux}
\ytableausetup{notabloids}
\ydiagram[*(cyan)]{2,2,2}
*[*(white)]{2+1,2+1,2+1}};
\node at (-0.5,0) {3)};
\end{scope}

\begin{scope}[shift={(\a+3*\b,0)}]
\ytableausetup{notabloids}
\ytableausetup{mathmode, boxsize=0.5 em}
\node (n) {\ytableausetup{nosmalltableaux}
\ytableausetup{notabloids}
\ydiagram[*(cyan)]{2,2,0}
*[*(white)]{2+1,2+1,3}};
\node at (-0.5,0) {4)};
\end{scope}

\begin{scope}[shift={(\a+4*\b,0)}]
\ytableausetup{notabloids}
\ytableausetup{mathmode, boxsize=0.5 em}
\node (n) {\ytableausetup{nosmalltableaux}
\ytableausetup{notabloids}
\ydiagram[*(cyan)]{2,0,0}
*[*(white)]{2+1,3,3}};
\node at (-0.5,0) {5)};
\end{scope}

\begin{scope}[shift={(\a+5*\b,0)}]
\ytableausetup{notabloids}
\ytableausetup{mathmode, boxsize=0.5 em}
\node (n) {\ytableausetup{nosmalltableaux}
\ytableausetup{notabloids}
\ydiagram[*(cyan)]{1,1,1}
*[*(white)]{1+2,1+2,1+2}};
\node at (-0.5,0) {6)};
\end{scope}

\begin{scope}[shift={(\a+6*\b,0)}]
\ytableausetup{notabloids}
\ytableausetup{mathmode, boxsize=0.5 em}
\node (n) {\ytableausetup{nosmalltableaux}
\ytableausetup{notabloids}
\ydiagram[*(cyan)]{1,1,0}
*[*(white)]{1+2,1+2,3}};
\node at (-0.5,0) {7)};
\end{scope}

\begin{scope}[shift={(\a+7*\b,0)}]
\ytableausetup{notabloids}
\ytableausetup{mathmode, boxsize=0.5 em}
\node (n) {\ytableausetup{nosmalltableaux}
\ytableausetup{notabloids}
\ydiagram[*(cyan)]{1,0,0}
*[*(white)]{1+2,3,3}};
\node at (-0.5,0) {8)};
\end{scope}
\end{tikzpicture}
\end{center}

In each diagram above, the white cells denote cells in $\cW_{\lambda}$ and the cyan cells denote cells in the DRH. (Compare Figure \ref{fig:NENNEE_LFC}.)

Let $u_{11}$, $u_{12}$, $u_{21}$, and $u_{22}$ be the upper-left, upper-right, lower-left, and lower-right vertices, respectively, of $c_{22}$. Also, let $\fq_{ij}$ denote the quadruple centered at $u_{ij}$. Since all $u_{ij}$ are non-special, we have that $|\fq_{ij}| = \FC(u_{ij})$. Therefore, it suffices to show in each case that $\FC(u_{11}) \FC(u_{22}) = \FC(u_{12}) \FC(u_{21})$.

\underline{Case 1.} In this case, $\FC(u_{11}) = \Delta_i$ and $\FC(u_{12}) = \Delta_{i+1}$ for some $i$. The frozen coefficients $\FC(u_{21})$ and $\FC(u_{22})$ were defined in Case 3.2 in Section \ref{ss:FCdef}. Following the description, if we let $w$ be the closest Step-1 lattice point to the left to $u_{21}$, we obtain $\FC(u_{21}) = \Delta_i \cdot \FC(w)$ and $\FC(u_{22}) = \Delta_{i+1} \cdot \FC(w)$. Hence, $\FC(u_{11}) \FC(u_{22}) = \Delta_i \Delta_{i+1} \FC(w) = \FC(u_{12}) \FC(u_{21})$.

\underline{Case 2.} We write the row and column labels for the six squares as follows.

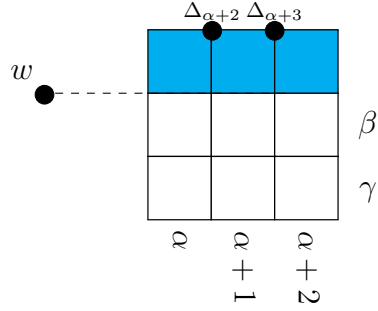
\begin{center}
\begin{tikzpicture}[scale = 1.0]
\begin{scope}
\ytableausetup{notabloids}
\ytableausetup{mathmode, boxsize=2.0 em}
\node (n) {\ytableausetup{nosmalltableaux}
\ytableausetup{notabloids}
\ydiagram[*(cyan)]{3,0,0}
*[*(white)]{0,3,3}};
\node[rotate = -90] at (-0.7,-1.53) {$\alpha$};
\node[rotate = -90] at (0.15,-1.9) {$\alpha+1$};
\node[rotate = -90] at (1.0,-1.9) {$\alpha+2$};

\node at (1.8,0) {$\beta$};
\node at (1.8,-0.9) {$\gamma$};
\end{scope}

\begin{scope}
\draw[dashed] (0.7,0.42) -- (-2.5,0.42);
\draw[fill] (-2.5,0.4) circle [radius = 0.3 em];
\node at (-2.8,0.7) {$w$};

\draw[fill] (0.56,1.25) circle [radius = 0.3 em];
\draw[fill] (-0.27,1.25) circle [radius = 0.3 em];

\node at (-0.27,1.5) {\tiny $\Delta_{\alpha+2}$};
\node at (0.56,1.5) {\tiny $\Delta_{\alpha+3}$};
\end{scope}

\end{tikzpicture}
\end{center}

Let $w$ be the closest Step-1 lattice point to the left of $u_{11}$. Note that $\FC(u_{11}) = \Delta_{\alpha+2} \cdot \FC(w)$ and $\FC(u_{12}) = \Delta_{\alpha+3} \cdot \FC(w)$. Recall the notations $B$, $E$, $I$, $\chi^X$, and $\chi^Y$ we used in the proof of Corollary \ref{c:detiszero}. We will use these functions again. For convenience, we will use $[r,s]$ to denote the subskeleton $(\lambda,r,s) = W^r W^{r+1} \cdots W^s$ of $\lambda$. The main decomposition theorem gives
\[
\FC(u_{21}) = B([\beta \vee \gamma, \alpha]) E([\beta \wedge \gamma, \alpha+1]) I([\beta \vee \gamma, \alpha]) \chi^X(\fq_{21}) \chi^Y(\fq_{21}),
\]
where $\beta \wedge \gamma := \min\{\beta, \gamma\}$ and $\beta \vee \gamma := \max \{\beta, \gamma\}$, and also
\[
\FC(u_{22}) = B([\beta \vee \gamma, \alpha+1]) E([\beta \wedge \gamma, \alpha+2]) I([\beta \vee \gamma, \alpha+1]) \chi^X(\fq_{22}) \chi^Y(\fq_{22}).
\]
Note that $[\beta \vee \gamma, \alpha]$ and $[\beta \vee \gamma, \alpha+1]$ share the same bends. Thus, $B([\beta \vee \gamma, \alpha]) = B([\beta \vee \gamma, \alpha+1])$. Also, since the $\alpha$-th, $(\alpha+1)$-st, and $(\alpha+2)$-nd letter in $\lambda$ are all $E$'s, the end points of both $[\beta \vee \gamma, \alpha]$ and $[\beta \vee \gamma, \alpha+1]$ are not bends in $\lambda$. Thus, $I([\beta \vee \gamma, \alpha]) = I([\beta \vee \gamma, \alpha+1])$. On the other hand, by considering the ending vertices of $[\beta \wedge \gamma, \alpha+2]$ and $[\beta \wedge \gamma, \alpha+3]$, we have
\[
\frac{E([\beta \wedge \gamma, \alpha+2])}{E([\beta \wedge \gamma, \alpha+1])} = \frac{\Delta_{\alpha+3}}{\Delta_{\alpha+2}}.
\]
Note also that $\chi^Y(\fq_{21}) = 1 = \chi^Y(\fq_{22})$, since both $\fq_{21}$ and $\fq_{22}$ are not split by the vertical DRH axis, and that $\chi^X(\fq_{21}) = \chi^X(\fq_{22})$. This gives 
\[
\FC(u_{11}) \FC(u_{22}) = \FC(u_{12}) \FC(u_{21}).
\]

\underline{Case 3.} This case is analogous to Case 1.

\underline{Case 4.} Let $v$ be the closest Step-1 lattice point above $u_{12}$ and $w$ be the closest Step-1 lattice point to the left of $u_{21}$. As $u_{11}$ is a lattice point inside the DRH, write $\FC(u_{11}) = \Delta_i$. We have $\FC(u_{12}) = \Delta_i \FC(v)$, $\FC(u_{21}) = \Delta_i \FC(w)$, and $\FC(u_{22}) = \Delta_i \FC(v) \FC(w)$. It is clear that $\FC(u_{11}) \FC(u_{22}) = \FC(u_{12}) \FC(u_{21})$.

\underline{Case 5.} We write the row and column labels and define $v$ and $w$ as in the following diagram.

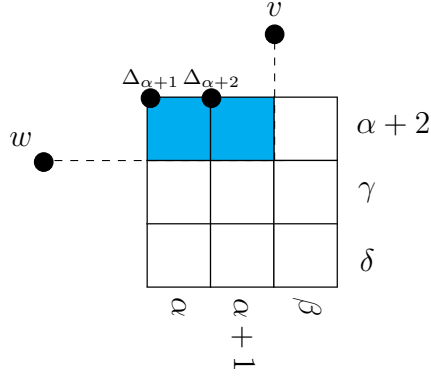
\begin{center}
\begin{tikzpicture}[scale = 1.0]
\begin{scope}
\ytableausetup{notabloids}
\ytableausetup{mathmode, boxsize=2.0 em}
\node (n) {\ytableausetup{nosmalltableaux}
\ytableausetup{notabloids}
\ydiagram[*(cyan)]{2,0,0}
*[*(white)]{2+1,3,3}};
\node[rotate = -90] at (-0.7,-1.53) {$\alpha$};
\node[rotate = -90] at (0.15,-1.9) {$\alpha+1$};
\node[rotate = -90] at (1.0,-1.53) {$\beta$};

\node at (2.15,0.9) {$\alpha+2$};
\node at (1.8,0) {$\gamma$};
\node at (1.8,-0.9) {$\delta$};
\end{scope}

\begin{scope}
\draw[dashed] (0.7,0.42) -- (-2.5,0.42);
\draw[fill] (-2.5,0.4) circle [radius = 0.3 em];
\node at (-2.8,0.7) {$w$};

\draw[dashed] (0.57,0.42) -- (0.57,2.1);
\draw[fill] (0.57,2.1) circle [radius = 0.3 em];
\node[above] at (0.57,2.2) {$v$};

\draw[fill] (-1.08,1.25) circle [radius = 0.3 em];
\draw[fill] (-0.27,1.25) circle [radius = 0.3 em];

\node at (-1.08,1.5) {\tiny $\Delta_{\alpha+1}$};
\node at (-0.27,1.5) {\tiny $\Delta_{\alpha+2}$};
\end{scope}
\end{tikzpicture}
\end{center}
Note that we have $\beta \ge \alpha+2$ and $\alpha \ge \gamma, \delta$. Using a similar analysis to Case 2, we find that $\FC(u_{11}) = \Delta_{\alpha+2} \FC(w)$ and $\FC(u_{12}) = \Delta_{\alpha+2} \FC(v) \FC(w)$. The main decomposition theorem gives
\[
\FC(u_{21}) = B([\gamma \vee \delta,\alpha]) E([\gamma \wedge \delta, \alpha + 1]) I([\gamma \vee \delta, \alpha]) \chi^X(\fq_{21}) \chi^Y(\fq_{21})
\]
and
\[
\FC(u_{22}) = B([\gamma \vee \delta,\alpha+1]) E([\gamma \wedge \delta, \beta]) I([\gamma \vee \delta, \alpha+1]) \chi^X(\fq_{22}) \chi^Y(\fq_{22}).
\]
Therefore,
\begin{align*}
\frac{\FC(u_{22})}{\FC(u_{21})} &= \frac{B([\gamma \vee \delta,\alpha+1])}{B([\gamma \vee \delta,\alpha])} \frac{E([\gamma \wedge \delta, \beta]) \chi^Y(\fq_{22})}{E([\gamma \wedge \delta, \alpha + 1]) \chi^Y(\fq_{21})} \frac{I([\gamma \vee \delta, \alpha+1])}{I([\gamma \vee \delta, \alpha])} \\
&= 1 \cdot \frac{\FC(v)}{\Delta_{\alpha+2}} \cdot \Delta_{\alpha+ 2} = \FC(v) = \frac{\FC(u_{12})}{\FC(u_{11})}.
\end{align*}

\underline{Case 6.} This case is analogous to Case 2.

\underline{Case 7.} This case is analogous to Case 5.

\underline{Case 8.} We give the row and column labels and define $v$ and $w$ in the diagram below.

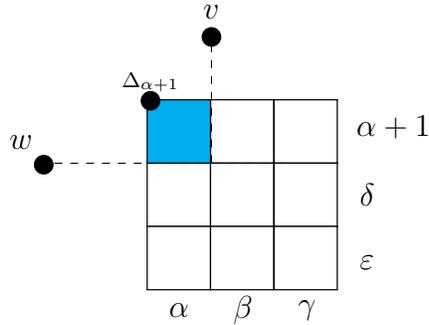
\begin{center}
\begin{tikzpicture}[scale = 1.0]
\def \a {0.84}
\begin{scope}
\ytableausetup{notabloids}
\ytableausetup{mathmode, boxsize=2.0 em}
\node (n) {\ytableausetup{nosmalltableaux}
\ytableausetup{notabloids}
\ydiagram[*(cyan)]{1,0,0}
*[*(white)]{1+2,3,3}};
\node at (-0.7,-1.53) {$\alpha$};
\node at (0.15,-1.53) {$\beta$};
\node at (1.0,-1.53) {$\gamma$};

\node at (2.15,0.9) {$\alpha+1$};
\node at (1.8,0) {$\delta$};
\node at (1.8,-0.9) {$\ve$};
\end{scope}

\begin{scope}
\draw[dashed] (0.7,0.42) -- (-2.5,0.42);
\draw[fill] (-2.5,0.4) circle [radius = 0.3 em];
\node at (-2.8,0.7) {$w$};

\draw[dashed] (0.57-\a,0.42) -- (0.57-\a,2.1);
\draw[fill] (0.57-\a,2.1) circle [radius = 0.3 em];
\node[above] at (0.57-\a,2.2) {$v$};

\draw[fill] (-1.08,1.25) circle [radius = 0.3 em];

\node at (-1.08,1.5) {\tiny $\Delta_{\alpha+1}$};
\end{scope}
\end{tikzpicture}
\end{center}

Note that $\beta, \gamma \ge \alpha + 1$ and $\alpha \ge \delta, \ve$. We have $\FC(u_{11}) = \Delta_{\alpha+1} \FC(v) \FC(w)$. The main decomposition theorem gives
\[
\FC(u_{12}) = B([\alpha+1,\beta \wedge \gamma]) E([\delta, \beta \vee \gamma]) I([\alpha+1, \beta \wedge \gamma]) \chi^X (\fq_{12}) \chi^Y(\fq_{12}),
\]
\[
\FC(u_{21}) = B([\delta \vee \ve, \alpha]) E([\delta \wedge \ve, \beta]) I([\delta \vee \ve,\alpha]) \chi^X (\fq_{21}) \chi^Y(\fq_{21}),
\]
and
\[
\FC(u_{22}) = B([\delta \vee \ve, \beta \wedge \gamma]) E([\delta \wedge \ve, \beta \vee \gamma]) I([\delta \vee \ve, \beta \wedge \gamma]) \chi^X (\fq_{22}) \chi^Y(\fq_{22}).
\]
Hence,
\[
\frac{\FC(u_{21})}{\FC(u_{22})} = \frac{1}{\Delta_{\alpha+1} \cdot B([\alpha+1, \beta \wedge \gamma])} \cdot \frac{\FC(v)}{\FC(\wh{v})} \cdot \frac{\Delta_{\alpha+1}}{\eta},
\]
where $\wh{v}$ denotes the lattice point inside the DRH corresponding to $\beta \vee \gamma$, while $\eta$ denotes the frozen coefficient at the lattice point corresponding to $\beta \wedge \gamma$ if the point is a bend in $\lambda$ and denotes $1$ if it is not.

Note that we have
\[
\FC(u_{12}) = B([\alpha+1, \beta \wedge \gamma]) \cdot \left( \FC(w)\FC(\wh{v}) \right) \cdot \left( \Delta_{\alpha+1} \cdot \eta \right).
\]
Therefore,
\[
\frac{\FC(u_{12}) \FC(u_{21})}{\FC(u_{22})} = \Delta_{\alpha+1} \FC(v) \FC(w) = \FC(u_{11}).
\]

We have finished the proof.
\end{proof}

\medskip

\subsection{Description of the quiver of each worm inside the DRH staircase} \label{ss:fcdesc}
In this section, we will use the frozen coefficients defined at the lattice points inside the DRH staircase to propose the quivers for the worms inside the DRH staircase. Our strategy of the proof of our main theorem (Theorem \ref{thm:main}) is as follows.
\begin{enumerate}
\item Propose, for every worm $w$ inside the DRH staircase, a quiver $Q_w$.
\item Show that the proposed quiver agrees with the description for the initial worm $(w_0, Q_{w_0})$.
\item Prove that if a worm operation sends the worm $w$ to the worm $w'$, then the corresponding quiver mutation sends the proposed quiver $Q_w$ to the quiver $Q_{w'}$.
\end{enumerate}
Once the steps above are finished, we conclude that the description of the proposed quiver at each worm is the correct description of the desired quiver for the cluster algebra, and therefore we obtain all the cluster variable of the DRH algebra.

We now propose the quivers. Although the amount of casework that will follow may look daunting, the idea of how we will propose the quivers $Q_w$ is simple: since we have proposed all the cluster variables $e(c)$ in every cell $c$ inside the staircase, we will propose the quiver $Q_w$ so that it is compatible with worm operations. Namely, suppose $c$ is a bend or an endpoint of $w$ and let $w'$ be the resulting worm once we perform the worm operation on $w$ at $c$. Then, we want the quiver $Q_w$ to be such that the exchange relation changes the variable $e(c)$ to the variable $e(c')$. 

Let $w = c_1 c_2 \cdots c_{l+1}$ be a worm inside the staircase. To describe the quiver $Q_w$, we need to propose (1) the arrows between the mutable vertices $c_1, \dots, c_{l+1}$ and (2) the frozen coefficients $\fc_{Q_w}(c_i)$ for all $i = 1, \dots, l+1$. The arrows between the mutable vertices are easy to describe. Since we would like the quivers $Q_w$ to be the correct quivers we would observe when we mutate from the original worm to $w$, the arrows between the mutable vertices in $Q_w$ must still follow the same rule we gave in the construction of the initial DRH quiver. Namely, there is a unique arrow between $c_i$ and $c_{i+1}$ for $i = 1, \dots, l$ and the arrow goes $c_i \rightarrow c_{i+1}$ if the cell $c_{i+1}$ is to the right of $c_i$ while it goes $c_i \leftarrow c_{i+1}$ if the cell $c_{i+1}$ is above $c_i$. It is direct to verify that this rule respects quiver mutations.

Next, we propose $\fc_{Q_w}(c_i)$ for all $i$. We consider whether $c_i$ is a bend in $w$. \underline{Case 1.1.} If $c_i$ is an EN-bend (that is, there are arrows from the two mutable vertices $c_{i-1}$ and $c_{i+1}$), then we propose
\[
\fc_{Q_w}(c_i) := \FC(u)
\]
where $u$ is the upper-left corner of the cell $c_i$. \underline{Case 1.2.} If $c_i$ is an NE-bend (that is, there are arrows from $c_i$ to the two mutable vertices $c_{i-1}$ and $c_{i+1}$), then we propose
\[
\fc_{Q_w}(c_i) := \frac{1}{\FC(v)}
\]
where $v$ is the lower-right corner of the cell $c_i$.

When $c_i$ is not a bend in $w$, we consider whether $c_i$ is an endpoint of $w$ (that is, $i = 1$ or $i=l+1$). When $c_i$ is an endpoint, say $i=1$, it is straightforward to see what $\fc_{Q_w}(c_1)$ must be. This is because we know that the mutation at $c_1$ must change $c_1$ to another cell on the lower zigzag path where the associated variable has a relatively simple formula. \underline{Case 2.1.} Suppose that $c_i$ is the starting point of the worm $w$ (That is, $i=1$). \underline{Case 2.1.1.} Suppose that $c_2$ is to the right of $c_1$ (That is, there is a mutable arrow $c_1 \rightarrow c_2$). Let $\gamma$ denote the cell immediately below $c_2$. In this case, a mutation at $c_1$ should move $c_1$ to $\gamma$. Note that $\gamma$ is on the lower zigzag line of the DRH staircase. We will consider the degrees of the polynomials $e(c_1)$, $e(c_2)$, and $e(\gamma)$. Evidently, there are only a few possibilities for the triple of degrees:
\[
\Big( \deg(e(c_1)), \deg(e(c_2)), \deg(e(\gamma)) \Big).
\]
Since both $c_1$ and $\gamma$ are on the lower zigzag path, the degrees $\deg(e(c_1))$ and $\deg(e(\gamma))$ can only be $1$ or $2$. The degree of such a cell is $1$ precisely when the cell is inside the DRH or the DRH transpose; otherwise, the degree is $2$ when the cell is in $\cW_{\lambda}$. The triple $\Big( \deg(e(c_1)), \deg(e(c_2)), \deg(e(\gamma)) \Big)$ may be one of the following eight possibilities: $(1,1,1)$, $(1,1,2)$, $(1,3,2)$, $(2,1,1)$, $(2,3,1)$, $(2,2,2)$, $(2,4,2)$, and $(1,2,1)$. 

\underline{Case 2.1.1.1.} Either all of $c_1$, $c_2$, $\gamma$ are in the DRH or they are in the DRH transpose. Let $\zeta$ be the frozen variable in the cell below $c_1$ and let $u$ be the lower-right corner of $c_1$. Note that necessarily
\[
\FC(u) = \begin{vmatrix}
e(c_1) & e(c_2) \\
\zeta & e(\gamma)
\end{vmatrix}.
\]
We propose
\[
\fc_{Q_w}(c_1) := \frac{\zeta}{\FC(u)}.
\]
In other words, we propose that there is an arrow from $c_1$ to $\zeta$ and an arrow from $\FC(u)$ to $c_1$. 

\underline{Case 2.1.1.2.} $\Big( \deg(e(c_1)), \deg(e(c_2)), \deg(e(\gamma)) \Big) = (1,1,2)$. This happens when both $c_1$ and $c_2$ are inside the DRH or the DRH transpose and $\gamma$ is outside. Let $u$ and $v$, respectively, denote the upper-left and the upper-right corners of the cell $c_1$. Let $\zeta$ be the frozen variable in the cell to the left of $c_1$ (which is $a_{11}$ if $c_1$ is in the DRH and which is $a_{pq}$ if $c_1$ is in the DRH transpose). We propose
\[
\fc_{Q_w}(c_1) := \frac{\FC(u)}{\zeta \cdot \FC(v)}.
\]

\underline{Case 2.1.1.3.} $\Big( \deg(e(c_1)), \deg(e(c_2)), \deg(e(\gamma)) \Big) = (1,3,2)$. Let $\zeta$ be the frozen cell to the left of $c_1$ and let $w$ be the upper right corner of $c_1$. From $w$, go upwards until we find the first lattice point which is the center of a $2 \times 2$-square inside the DRH or the transpose DRH, and call that lattice point $u$. We propose
\[
\fc_{Q_w}(c_1) := \frac{1}{\zeta \cdot \FC(u)}.
\]

\underline{Case 2.1.1.4.} $\Big( \deg(e(c_1)), \deg(e(c_2)), \deg(e(\gamma)) \Big) = (2,1,1)$. This case is similar to Case 2.1.1.2. Let $\zeta$ be the frozen variable right below $\gamma$. Let $u$ and $v$, respectively, be the lower-right and upper-right corners of $\gamma$. Then, we propose
\[
\fc_{Q_w}(c_1) := \frac{\FC(u)}{\zeta \cdot \FC(v)}.
\]

\underline{Case 2.1.1.5.} $\Big( \deg(e(c_1)), \deg(e(c_2)), \deg(e(\gamma)) \Big) = (2,3,1)$. This case is similar to Case 2.1.1.3. Let $\zeta$ be the frozen variable in the cell right below $\gamma$. Let $w$ be the upper-right corner of $\gamma$. From $w$, go right until we find the first lattice point which is the center of a $2 \times 2$-square inside the DRH or the transpose DRH, and call that lattice point $u$. We propose
\[
\fc_{Q_w}(c_1) = \frac{1}{\zeta \cdot \FC(u)}.
\]

The two cases 2.1.1.6 and 2.1.1.7 concern the situation in which $\deg(e(c_1))$ and $\deg(e(\gamma))$ are both $2$. This happens when $c_1$ and $\gamma$ are inside $\cW_{\lambda}$. Consider the columns of cells inside $\cW_{\lambda}$ that $c_1$ and $\gamma$ are in. Let $c^{\ast}_1$ and $c^{\ast}_2 \in \cW_{\lambda}$, respectively, be the highest cells inside $\cW_{\lambda}$ that are in the same column as $c_1$ and $\gamma$. It is evident that the cell $c^{\ast}_2$ is always higher or at the same height as the cell $c^{\ast}_1$. If $c^{\ast}_1$ and $c^{\ast}_2$ are at the same height (that is, $c^{\ast}_2$ is the next cell to the right of $c^{\ast}_1$), then $\deg(e(c_2))$ is $2$. Otherwise, $\deg(e(c_2)) = 4$.

\underline{Case 2.1.1.6.} $\Big( \deg(e(c_1)), \deg(e(c_2)), \deg(e(\gamma)) \Big) = (2,2,2)$. Let $w$ be the upper-right corner of the cell $c_1$. From $w$, go upwards until we find the first lattice point which is the center of a $2 \times 2$-square inside the DRH or the transpose DRH, and call that lattice point $u_1$. Let $u_2$ be the lattice point one unit to the left of $u_1$, and let $u_3$ be the lattice point one unit to the left of $u_2$. Necessarily, all of $u_1$, $u_2$, and $u_3$ are the centers of some $2 \times 2$-squares inside the DRH or the DRH transpose. In particular, $\FC(u_1)$, $\FC(u_2)$, and $\FC(u_3)$ are frozen cluster variables. We propose
\[
\fc_{Q_w}(c_1) := \frac{\FC(u_2)}{\FC(u_1) \FC(u_3)}.
\]

\underline{Case 2.1.1.7.} $\Big( \deg(e(c_1)), \deg(e(c_2)), \deg(e(\gamma)) \Big) = (2,4,2)$. Let $w_1$ and $w_2$, respectively, be the upper-left corners of $c_1$ and $c_2$. For each $i = 1,2$, starting from $w_i$, we go upwards until we find the first lattice point that is the center of a $2 \times 2$-square inside the DRH or the transpose DRH, and call that lattice point $u_i$. For each $i = 1,2$, let $v_i$ be the lattice point one unit to the left of $u_i$. Necessarily, $v_i$ is also the center of a $2 \times 2$-square inside the DRH or the DRH transpose. We propose
\[
\fc_{Q_w}(c_1) := \frac{1}{\FC(v_1) \FC(u_2)}.
\]

\underline{Case 2.1.1.8.} $\Big( \deg(e(c_1)), \deg(e(c_2)), \deg(e(\gamma)) \Big) = (1,2,1)$. This is a rather exceptional case, when one of $c_1$ and $\gamma$ is in the DRH and the other is in the DRH transpose. This only happens when $\lambda = NN \cdots N$. We define
\[
\fc_{Q_w}(c_1) := \frac{1}{a_{11} a_{pq}}.
\]

Now that we have proposed the quiver in all the subcases of Case 2.1.1, the rest of Case 2 are straightforward: either each of the remaining cases is analogous to or it is one step away from some case in Case 2.1.1. 

\underline{Case 2.1.2.} Suppose that $c_2$ is above $c_1$ (That is, there is a mutable arrow $c_1 \leftarrow c_2$). Let $\xi$ be the cell to the left of $c_2$. Note that a mutation at $c_1$ in $w$ changes $c_1$ to $\xi$. Let $\wt{w}$ be the resulting worm after the mutation. That is, $\wt{w}$ has $\xi$ as the starting cell instead of $c_1$, while all other $l$ cells remain the same as those in $w$. Case 2.1.1 has proposed the quiver $Q_{\wt{w}}$. Thus, we may propose $Q_w$ to be the resulting quiver when we mutate $Q_{\wt{w}}$ at $\xi$.

\underline{Case 2.2.} Suppose $c_i$ is the final cell. That is, $i = l+1$. This case is analogous to Case 2.1. We may conjugate the whole $\lambda$-DRH staircase to obtain the $\lambda^t$-DRH staircase. The cell $c_{l+1}$ inside the worm $w$ will become the starting cell inside $w^t$. Describe the quiver at the cell analogously.

The remaining case of our description is when $c_i$ is neither an endpoint nor a bend. The mutable arrows near $c_i$ are either $c_{i-1} \rightarrow c_i \rightarrow c_{i+1}$ or $c_{i-1} \leftarrow c_i \leftarrow c_{i+1}$.

\underline{Case 3.1.} Suppose that $c_{i-1}$, $c_i$, and $c_{i+1}$ are on the same row. That is, the arrows go as $c_{i-1} \rightarrow c_i \rightarrow c_{i+1}$. Let $c^{\ast}$ and $c_{\ast}$, respectively, be the cell above and the one below $c_i$. Let $c^{\ast \ast}$ be the cell to the right of $c^{\ast}$ and let $c_{\ast \ast}$ be the cell to the left of $c_{\ast}$. Consider any worm $w^\ast$ which has $c_{i-1} \rightarrow c_i \rightarrow c^{\ast} (\rightarrow c^{\ast \ast})$ inside, and any worm $w_{\ast}$ which has $(c_{\ast \ast} \rightarrow) c_{\ast} \rightarrow c_i \rightarrow c_{i+1}$ inside. (The parentheses about $c^{\ast \ast}$ and $c_{\ast \ast}$ indicate that the parts of the worm may not be there if $c^{\ast}$ or $c_{\ast}$ is already the endpoint of the worm.) In Cases 1 and 2, we have given the description of the quiver $Q_{w_{\ast}}$ locally at $c_{\ast}$ and $c_i$ and the quiver $Q_{w^{\ast}}$ locally at $c_i$ and $c^{\ast}$. We can mutate $Q_{w_{\ast}}$ at $c_{\ast}$ to obtain a quiver $Q'_{w_{\ast}}$. We can also mutate $Q_{w^{\ast}}$ at $c^{\ast}$ to obtain a quiver $Q'_{w^{\ast}}$. Note that both quivers $Q'_{w_{\ast}}$ and $Q'_{w^{\ast}}$ has $c_{i-1} \rightarrow c_i \rightarrow c_{i+1}$ inside their underlying worms. 

To describe the quiver $Q_w$ at $c_i$, we claim that both quivers have the same local data of arrows at $c_i$, and therefore, we can use these local data to propose the arrows at $c_i$ in $Q_w$. If $c_i$ is not next to an endpoint of the worm, this follows as a result of Lemma \ref{l:abcompat} below. On the other hand, $c_{i-1}$ is the starting cell or if $c_{i+1}$ is the finishing cell. If it is the former case but not the latter, that is $i=2$, then this follows as a result of Lemma \ref{l:startabcompat}. If it is the latter but not the former, then we can argue analogously. Finally, if both $c_{i-1}$ is the starting cell and $c_{i+1}$ is the finishing cell, then $l=2$, which is a small case where we can write down all the quivers and cluster variables explicitly.

\underline{Case 3.2.} Suppose that $c_{i-1}$, $c_i$, and $c_{i+1}$ are on the same column. That is, the arrows go as $c_{i-1} \leftarrow c_i \leftarrow c_{i+1}$. We can define the quiver $Q_w$ locally at $c_i$ analogously to Case 3.1.

We have now given a complete description of the quiver for every worm.

\begin{lemma} \label{l:abcompat}
Suppose that $[c_{ij}]_{1 \le i, j \le 3}$ is a $3 \times 3$ square of cells inside the DRH staircase. Let $w_1$ be a worm that contains $c_{21} \rightarrow c_{22} \leftarrow c_{12} \rightarrow c_{13}$. Let $w_2$ be a worm that contains $c_{31} \rightarrow c_{32} \leftarrow c_{22} \rightarrow c_{23}$. For a worm $w$, let $Q_w$ denote its proposed quiver. Let $(w'_1,Q_{w'_1})$ be the resulting quiver when mutating $(w_1,Q_{w_1})$ at $c_{12}$. Let $(w'_2,Q_{w'_2})$ be the resulting quiver when mutating $(w_2,Q_{w_2})$ at $c_{32}$. Then,
\[
\fc_{Q_{w'_1}}(c_{22}) = \fc_{Q_{w'_2}}(c_{22}).
\]
\end{lemma}

Before proving the lemma, let us discuss the motivation behind this result. Let $v_{11}$, $v_{12}$, $v_{21}$, and $v_{22}$, respectively, be the upper-left, upper-right, lower-left, and lower-right corners of $c_{22}$. The result is complicated when some of the $v_{ij}$ are special points. On the other hand, when all the four points are non-special, the proof is rather straightforward. If all four are non-special, then the frozen coefficients at the lattice points are the corresponding $2 \times 2$ determinants at the points. Namely,
\[
\FC(v_{ij}) := \begin{vmatrix}
e(c_{i,j}) & e(c_{i,j+1}) \\
e(c_{i+1,j}) & e(c_{i+1,j+1})
\end{vmatrix}.
\]
Inside the worm $w_1$, the cell $c_{22}$ is an EN-bend, while the cell $c_{12}$ is an NE-bend. From the description in the beginning of this section, we have
\[
\fc_{Q_{w_1}}(c_{22}) = \FC(v_{11})
\]
and
\[
\fc_{Q_{w_1}}(c_{12}) = \frac{1}{\FC(v_{12})}.
\]
When all $v_{ij}$ are non-special, $\FC(v_{ij})$ are products of frozen variables of non-negative degrees. Hence, all the arrows at $c_{22}$ in $w_1$ are out-arrows, while all the arrows at $c_{12}$ in $w_1$ are in-arrows. In this situation, Corollary \ref{c:outmuta} applies. We have
\[
\fc_{Q_{w'_1}}(c_{22}) = \fc_{Q_{w_1}}(c_{22}) \cdot \fc_{Q_{w_1}}(c_{12}) = \frac{\FC(v_{11})}{\FC(v_{12})}.
\]
Analogously, in $w_2$, Lemma \ref{l:inmuta} gives
\[
\fc_{Q_{w'_2}}(c_{22}) = \frac{\FC(v_{21})}{\FC(v_{22})}.
\]
Lemma \ref{l:abcompat} says that $\frac{\FC(v_{11})}{\FC(v_{12})} = \frac{\FC(v_{21})}{\FC(v_{22})}$ and therefore we can define the quiver locally at $c_{22}$ either from ``above'' or from ``below''. To see why the lemma is true in this case, note that the $\FC(v_{ij})$ are $2 \times 2$-minors inside the matrix
\[
\begin{bmatrix}
e(c_{11}) & e(c_{12}) & e(c_{13}) \\
e(c_{21}) & e(c_{22}) & e(c_{23}) \\
e(c_{31}) & e(c_{32}) & e(c_{33}) \\
\end{bmatrix}.
\]
By Dodgson condensation, we have
\[
\FC(v_{11}) \FC(v_{22}) - \FC(v_{12}) \FC(v_{21}) = e(c_{22}) \cdot \begin{vmatrix}
e(c_{11}) & e(c_{12}) & e(c_{13}) \\
e(c_{21}) & e(c_{22}) & e(c_{23}) \\
e(c_{31}) & e(c_{32}) & e(c_{33}) \\
\end{vmatrix}
\]
which is zero by Proposition \ref{p:gendet}. Therefore, $\frac{\FC(v_{11})}{\FC(v_{12})} = \frac{\FC(v_{21})}{\FC(v_{22})}$ as desired.

As mentioned earlier, the result becomes more complicated when some $v_{ij}$ are special points. We now finish the proof of Lemma \ref{l:abcompat} below.

\begin{proof}
Recall that there are two types of special points: the ones on the boundary of the DRH and the ones on the boundary of the DRH transpose. We observe that both types of special points cannot occur in $\{v_{11}, v_{12}, v_{21}, v_{22}\}$ simultaneously. By symmetry of quiver mutation, we may in fact assume that the only special points that occur are on the boundary of the original DRH. In the following, when we refer to a special point, we shall mean the one which is on the boundary of the original DRH.

If none of the four $v_{ij}$ are special, then the analysis above yields the result. The following casework considers which of the points are special. Note that $v_{12}$ and $v_{21}$ may be simultaneously special. Aside from this pair, no two of $v_{ij}$'s can be special at the same time.

\underline{Case 1.} $v_{11}$ is special. This means that $c_{11}$, $c_{12}$, and $c_{21}$ are all in the DRH, while $c_{22}$, $c_{23}$, $c_{32}$, and $c_{33}$ are not in the DRH. It depends on the shape of the DRH to tell whether $c_{13}$ and $c_{31}$ are in the DRH.

\underline{Case 1.1.} Both $c_{13}$ and $c_{31}$ are not in the DRH. 

\begin{center}
\begin{tikzpicture}
\ytableausetup{notabloids}
\ytableausetup{mathmode, boxsize=1.0 em}
\node (n) {\ytableausetup{nosmalltableaux}
\ytableausetup{notabloids}
\ydiagram[*(cyan)]{2,1,0}
*[*(white)]{2+1,1+2,3}};
\node at (0.15,-1.0) {\tiny Case 1.1};
\end{tikzpicture}
\end{center}

Extend the labels $c_{ij}$ to include $c_{00}$, $c_{10}$, $c_{20}$, $c_{01}$, and $c_{02}$ in the natural way. Also, let $v_{ij}$ be the lower right corner of $c_{ij}$. Note that the five cells $c_{00}$, $c_{10}$, $c_{20}$, $c_{01}$, and $c_{02}$ are in the DRH. From $v_{12}$, we go upwards until we meet the first Step-1 lattice point, and call it $\nu$. From $v_{21}$, we go to the left until we meet the first Step-1 lattice point, and call it $\tau$. Referring to the description of frozen arrows earlier in this section, we find
\[
\fc_{Q_{w_1}}(c_{22}) = \FC(v_{11}) = \frac{\FC(v_{10}) \FC(v_{01})}{\FC(v_{00})}
\]
and
\[
\fc_{Q_{w_1}}(c_{12}) = \frac{1}{\FC(v_{12})} = \frac{1}{\FC(v_{01})\FC(v_{10}) \FC(\nu)}.
\]
Corollary \ref{c:outmuta} yields
\[
\fc_{Q_{w'_1}}(c_{22}) = \fc_{Q_{w_1}}(c_{22}) \cdot \fc_{Q_{w_1}}(c_{12}) = \frac{1}{\FC(\nu) \FC(v_{00})}.
\]
Similarly, we have
\[
\fc_{Q_{w_2}}(c_{32}) = \FC(v_{21}) = \FC(\tau) \FC(v_{10}) \FC(v_{01})
\]
and
\[
\fc_{Q_{w_2}}(c_{22}) = \frac{1}{\FC(v_{22})}.
\]
To compute $\FC(v_{22})$, we use the main decomposition theorem. Let $\mu_{ij}$ denote the subskeleton corresponding to the cell $c_{ij}$ and let $\fq_{ij}$ denote the $2 \times 2$-square centered at $v_{ij}$. Using the notations $B$, $E$, $I$, $\chi^X$, and $\chi^Y$ as before, we see that the theorem gives
\begin{align*}
\FC(v_{22}) &= B(\mu_{22}) E(\mu_{33}) I(\mu_{22}) \chi^X (\fq_{22}) \chi^Y(\fq_{22}) \\
&= B(\mu_{22}) \cdot \left( E(\mu_{33}) \chi^X (\fq_{22}) \chi^Y(\fq_{22}) \right) \cdot I(\mu_{22}) \\
&= \left( \FC(v_{00}) \right) \cdot \left( \FC(\tau) \FC(\nu) \right) \cdot \left( \FC(v_{10}) \FC(v_{01}) \right).
\end{align*}
Therefore,
\[
\fc_{Q_{w_2}}(c_{22}) = \frac{1}{\FC(v_{22})} = \frac{1}{\FC(v_{00}) \FC(v_{10}) \FC(v_{01}) \FC(\tau) \FC(\nu)}.
\]

Lemma \ref{l:inmuta} yields
\[
\fc_{Q_{w'_2}}(c_{22}) = \fc_{Q_{w_2}}(c_{32}) \cdot \fc_{Q_{w_2}}(c_{22}) = \frac{1}{\FC(\nu) \FC(v_{00})}.
\]
Thus, $\fc_{Q_{w'_1}}(c_{22}) = \fc_{Q_{w'_2}}(c_{22})$ as desired.

\underline{Case 1.2.} $c_{13}$ is in the DRH, but $c_{31}$ is not. 

\begin{center}
\begin{tikzpicture}
\def \a {0.21}
\ytableausetup{notabloids}
\ytableausetup{mathmode, boxsize=1.0 em}
\node (n) {\ytableausetup{nosmalltableaux}
\ytableausetup{notabloids}
\ydiagram[*(cyan)]{3,1,0}
*[*(white)]{0,1+2,3}};
\node at (0.15,-1.0) {\tiny Case 1.2};
\end{tikzpicture}
\end{center}

Extend $c_{ij}$ and $v_{ij}$ as we did in the previous case. (When there is no confusion, we will continue to use this convention in other cases as well.) From $v_{21}$, go left until we find the first Step-1 lattice point and call it $\tau$. Similar to the case above, we can directly compute
\[
\fc_{Q_{w_1}}(c_{22}) = \FC(v_{11}) = \frac{\FC(v_{10}) \FC(v_{01})}{\FC(v_{00})}
\]
and
\[
\fc_{Q_{w_1}}(c_{12}) = \frac{1}{\FC(v_{12})} = \frac{1}{\FC(v_{02}) \FC(v_{10})}.
\]
Thus, we have
\[
\fc_{Q_{w'_1}}(c_{22}) = \fc_{Q_{w_1}}(c_{22}) \cdot \fc_{Q_{w_1}}(c_{12}) = \frac{\FC(v_{01})}{\FC(v_{00}) \FC(v_{02})}.
\]
Also, the main decomposition theorem gives
\[
\fc_{Q_{w_2}}(c_{32}) = \FC(v_{21}) = \FC(\tau) \FC(v_{10}) \FC(v_{01})
\]
and
\[
\fc_{Q_{w_2}}(c_{22}) = \frac{1}{\FC(v_{22})} = \frac{1}{\FC(\tau) \FC(v_{10}) \FC(v_{00}) \FC(v_{02})}.
\]
Therefore,
\[
\fc_{Q_{w'_2}}(c_{22}) = \fc_{Q_{w_2}}(c_{32}) \cdot \fc_{Q_{w_2}}(c_{22}) = \frac{\FC(v_{01})}{\FC(v_{00}) \FC(v_{02})}.
\]
This shows that $\fc_{Q_{w'_1}}(c_{22}) = \fc_{Q_{w'_2}}(c_{22})$.

\underline{Case 1.3.} $c_{13}$ is not in the DRH, while $c_{31}$ is. 

\begin{center}
\begin{tikzpicture}
\def \a {0.21}
\ytableausetup{notabloids}
\ytableausetup{mathmode, boxsize=1.0 em}
\node (n) {\ytableausetup{nosmalltableaux}
\ytableausetup{notabloids}
\ydiagram[*(cyan)]{2,1,1}
*[*(white)]{2+1,1+2,1+2}};
\node at (0.15,-1.0) {\tiny Case 1.3};
\end{tikzpicture}
\end{center}

By considering the conjugate DRH, this case is analogous to Case 1.2.

\underline{Case 1.4.} Both $c_{13}$ and $c_{31}$ are in the DRH.

\begin{center}
\begin{tikzpicture}
\def \a {0.21}
\ytableausetup{notabloids}
\ytableausetup{mathmode, boxsize=1.0 em}
\node (n) {\ytableausetup{nosmalltableaux}
\ytableausetup{notabloids}
\ydiagram[*(cyan)]{3,1,1}
*[*(white)]{0,1+2,1+2}};
\node at (0.15,-1.0) {\tiny Case 1.4};
\end{tikzpicture}
\end{center}

In this case, all of $e(c_{22})$, $e(c_{23})$, $e(c_{32})$, and $e(c_{33})$ are cubic polynomials. As usual, we directly compute the frozen coefficients. Direct computations give
\[
\fc_{Q_{w_1}}(c_{22}) = \FC(v_{11}) = \frac{\FC(v_{10}) \FC(v_{01})}{\FC(v_{00})}
\]
and
\[
\fc_{Q_{w_1}}(c_{12}) = \frac{1}{\FC(v_{12})} = \frac{1}{\FC(v_{10}) \FC(v_{02})}.
\]
We then have
\[
\fc_{Q_{w'_1}}(c_{22}) = \fc_{Q_{w_1}}(c_{22}) \cdot \fc_{Q_{w_1}}(c_{12}) = \frac{\FC(v_{01})}{\FC(v_{00}) \FC(v_{02})}.
\]
Next, we look at the worm $w_2$. We have
\[
\fc_{Q_{w_2}}(c_{32}) = \FC(v_{21}) = \FC(v_{20}) \FC(v_{01})
\]
and
\[
\fc_{Q_{w_2}}(c_{22}) = \frac{1}{\FC(v_{22})} = \frac{1}{\FC(v_{20}) \FC(v_{00}) \FC(v_{02})}.
\]
Therefore,
\[
\fc_{Q_{w'_2}}(c_{22}) = \fc_{Q_{w_2}}(c_{32}) \cdot \fc_{Q_{w_2}}(c_{22}) = \frac{\FC(v_{01})}{\FC(v_{00}) \FC(v_{02})}.
\]
This shows that $\fc_{Q_{w'_1}}(c_{22}) = \fc_{Q_{w'_2}}(c_{22})$.

\underline{Case 2.} $v_{12}$ is special. In this case, the cells $c_{11}$, $c_{12}$, $c_{13}$, $c_{21}$, and $c_{22}$ are all in the DRH. The cells $c_{23}$ and $c_{33}$ are in $\cW_{\lambda}$. Depending on the shape of the DRH, $c_{31}$ and $c_{32}$ may be inside the DRH or in $\cW_{\lambda}$. Note that if $c_{32}$ is in the DRH, then $c_{31}$ must also be in the DRH.

\underline{Case 2.1.} Both $c_{31}$ and $c_{32}$ are not in the DRH. 

\begin{center}
\begin{tikzpicture}
\ytableausetup{notabloids}
\ytableausetup{mathmode, boxsize=1.0 em}
\node (n) {\ytableausetup{nosmalltableaux}
\ytableausetup{notabloids}
\ydiagram[*(cyan)]{3,2,0}
*[*(white)]{0,2+1,3}};
\node at (0.15,-1.0) {\tiny Case 2.1};
\end{tikzpicture}
\end{center}

From $v_{21}$, we go left until we meet the first Step-1 lattice point and call it $\tau$. We have
\[
\fc_{Q_{w_1}}(c_{22}) = \FC(v_{11})
\]
and
\[
\fc_{Q_{w_1}}(c_{12}) = \frac{1}{\FC(v_{12})} = \frac{\FC(v_{01})}{\FC(v_{11}) \FC(v_{02})}.
\]
Note that $\FC(v_{11})$ is already a frozen variable, because $v_{11}$ is the center of a $2 \times 2$ square inside the DRH. By quiver mutation, we find
\[
\fc_{Q_{w'_1}}(c_{22}) = \frac{1}{\FC(v_{02})}.
\]
We also have
\[
\fc_{Q_{w_2}}(c_{32}) = \FC(v_{21}) = \FC(\tau) \FC(v_{11})
\]
and
\[
\fc_{Q_{w_2}}(c_{22}) = \frac{1}{\FC(v_{22})} = \frac{1}{\FC(\tau) \FC(v_{11}) \FC(v_{02})}.
\]
Therefore,
\[
\fc_{Q_{w'_2}}(c_{22}) = \frac{1}{\FC(v_{02})}.
\]
This implies $\fc_{Q_{w'_1}}(c_{22}) = \fc_{Q_{w'_2}}(c_{22})$.

\underline{Case 2.2.} $c_{31}$ is in the DRH, but $c_{32}$ is not. 

\begin{center}
\begin{tikzpicture}
\ytableausetup{notabloids}
\ytableausetup{mathmode, boxsize=1.0 em}
\node (n) {\ytableausetup{nosmalltableaux}
\ytableausetup{notabloids}
\ydiagram[*(cyan)]{3,2,1}
*[*(white)]{0,2+1,1+2}};
\node at (0.15,-1.0) {\tiny Case 2.2};
\end{tikzpicture}
\end{center}

This is the case in which both $v_{12}$ and $v_{21}$ are special. Similar to Case 2.1, we have
\[
\fc_{Q_{w'_1}}(c_{22}) = \frac{1}{\FC(v_{02})}.
\]
Direct computations give
\[
\fc_{Q_{w_2}}(c_{32}) = \FC(v_{21}) = \frac{\FC(v_{20}) \FC(v_{11})}{\FC(v_{10})}
\]
and
\[
\fc_{Q_{w_2}}(c_{22}) = \frac{1}{\FC(v_{22})} = \frac{1}{\FC(v_{20}) \FC(v_{11}) \FC(v_{02})}.
\]
Therefore, we have
\[
\fc_{Q_{w'_2}}(c_{22}) = \frac{1}{\FC(v_{02})}.
\]
Thus, $\fc_{Q_{w'_1}}(c_{22}) = \fc_{Q_{w'_2}}(c_{22})$.

\underline{Case 2.3.} Both $c_{31}$ and $c_{32}$ are in the DRH. 

\begin{center}
\begin{tikzpicture}
\ytableausetup{notabloids}
\ytableausetup{mathmode, boxsize=1.0 em}
\node (n) {\ytableausetup{nosmalltableaux}
\ytableausetup{notabloids}
\ydiagram[*(cyan)]{3,2,2}
*[*(white)]{0,2+1,2+1}};
\node at (0.15,-1.0) {\tiny Case 2.3};
\end{tikzpicture}
\end{center}

Similar to Case 2.1, we have
\[
\fc_{Q_{w'_1}}(c_{22}) = \frac{1}{\FC(v_{02})}.
\]
We also have
\[
\fc_{Q_{w_2}}(c_{32}) = \FC(v_{21})
\]
and
\[
\fc_{Q_{w_2}}(c_{22}) = \frac{1}{\FC(v_{22})} = \frac{1}{\FC(v_{21}) \FC(v_{02})}.
\]
Therefore,
\[
\fc_{Q_{w'_2}}(c_{22}) = \frac{1}{\FC(v_{02})}.
\]
Hence, $\fc_{Q_{w'_1}}(c_{22}) = \fc_{Q_{w'_2}}(c_{22})$.

\underline{Case 3.} $v_{21}$ is special. By considering the conjugate DRH, this case is analogous to Case 2.

\underline{Case 4.} $v_{22}$ is special. 

\begin{center}
\begin{tikzpicture}
\ytableausetup{notabloids}
\ytableausetup{mathmode, boxsize=1.0 em}
\node (n) {\ytableausetup{nosmalltableaux}
\ytableausetup{notabloids}
\ydiagram[*(cyan)]{3,3,2}
*[*(white)]{0,0,2+1}};
\node at (0.15,-1.0) {\tiny Case 4};
\end{tikzpicture}
\end{center}

In this case, all $c_{ij}$ with $1 \le i,j \le 3$ except $c_{33}$ are in the DRH. We have
\[
\fc_{Q_{w_1}}(c_{22}) = \FC(v_{11})
\]
and
\[
\fc_{Q_{w_1}}(c_{12}) = \frac{1}{\FC(v_{12})}.
\]
Thus,
\[
\fc_{Q_{w'_1}}(c_{22}) = \frac{\FC(v_{11})}{\FC(v_{12})}.
\]
We also have
\[
\fc_{Q_{w_2}}(c_{32}) = \FC(v_{21})
\]
and
\[
\fc_{Q_{w_2}}(c_{22}) = \frac{1}{\FC(v_{22})} = \frac{\FC(v_{11})}{\FC(v_{12}) \FC(v_{21})}.
\]
Therefore,
\[
\fc_{Q_{w'_2}}(c_{22}) = \frac{\FC(v_{11})}{\FC(v_{12})}.
\]
Hence, $\fc_{Q_{w'_1}}(c_{22}) = \fc_{Q_{w'_2}}(c_{22})$.  We have finished the proof of the lemma.
\end{proof}

Next, we show a lemma similar to the previous one for the case when the quiver mutation happens near the staircase zigzag lines.

\begin{lemma} \label{l:startabcompat}
Suppose that $[c_{ij}]_{1 \le i, j \le 3}$ is a $3 \times 3$ square of cells, with all except $c_{31}$ inside the DRH staircase. Let $w_1$ be a worm that starts as $c_{21} \rightarrow c_{22} \leftarrow c_{12} \rightarrow c_{13}$. Let $w_2$ be a worm that starts as $c_{32} \leftarrow c_{22} \rightarrow c_{23}$. For a worm $w$, let $Q_w$ denote its proposed quiver. Let $(w'_1,Q_{w'_1})$ be the resulting quiver when mutating $w_1$ at $c_{12}$. Let $(w'_2, Q_{w'_2})$ be the resulting quiver when mutating $w_2$ at $c_{32}$. Then,
\[
\fc_{Q_{w'_1}}(c_{22}) = \fc_{Q_{w'_2}}(c_{22}).
\]
\end{lemma}

\begin{center}
\begin{figure}
\begin{tikzpicture}
\def \a {+1.5};
\def \b {+2.0};
\def \s {0.32}
\def \m {-0.34}
\def \n {0.15}
\def \g {0.6}
\def \h {0.1}

\begin{scope}[shift={(-6*\b,0)}]
\ytableausetup{notabloids}
\ytableausetup{mathmode, boxsize=0.7 em}
\node (n) {\ytableausetup{nosmalltableaux}
\ytableausetup{notabloids}
\ydiagram[*(cyan)]{0,0,0}
*[*(white)]{3,3,3}};
\node at (-\g,0) {1)};
\draw[line width = 0.17 em] (\m,\n) -- (\m,{\n-\s}) -- ({\m + \s},{\n - \s}) -- ({\m+\s},{\n-2*\s}) -- ({\m+2*\s},{\n-2*\s});
\end{scope}

\begin{scope}[shift={(-5*\b,0)}]
\ytableausetup{notabloids}
\ytableausetup{mathmode, boxsize=0.7 em}
\node (n) {\ytableausetup{nosmalltableaux}
\ytableausetup{notabloids}
\ydiagram[*(cyan)]{1,0,0}
*[*(white)]{1+2,3,3}};
\node at (-\g,0) {2)};
\draw[line width = 0.17 em] (\m,\n) -- (\m,{\n-\s}) -- ({\m + \s},{\n - \s}) -- ({\m+\s},{\n-2*\s}) -- ({\m+2*\s},{\n-2*\s});
\end{scope}

\begin{scope}[shift={(-4*\b,0)}]
\ytableausetup{notabloids}
\ytableausetup{mathmode, boxsize=0.7 em}
\node (n) {\ytableausetup{nosmalltableaux}
\ytableausetup{notabloids}
\ydiagram[*(cyan)]{2,0,0}
*[*(white)]{2+1,3,3}};
\node at (-\g,0) {3)};
\draw[line width = 0.17 em] (\m,\n) -- (\m,{\n-\s}) -- ({\m + \s},{\n - \s}) -- ({\m+\s},{\n-2*\s}) -- ({\m+2*\s},{\n-2*\s});
\end{scope}

\begin{scope}[shift={(-3*\b,0)}]
\ytableausetup{notabloids}
\ytableausetup{mathmode, boxsize=0.7 em}
\node (n) {\ytableausetup{nosmalltableaux}
\ytableausetup{notabloids}
\ydiagram[*(cyan)]{3,0,0}
*[*(white)]{0,3,3}};
\node at (-\g,0) {4)};
\draw[line width = 0.17 em] (\m,\n) -- (\m,{\n-\s}) -- ({\m + \s},{\n - \s}) -- ({\m+\s},{\n-2*\s}) -- ({\m+2*\s},{\n-2*\s});
\end{scope}

\begin{scope}[shift={(-2*\b,0)}]
\ytableausetup{notabloids}
\ytableausetup{mathmode, boxsize=0.7 em}
\node (n) {\ytableausetup{nosmalltableaux}
\ytableausetup{notabloids}
\ydiagram[*(cyan)]{1,1,0}
*[*(white)]{1+2,1+2,3}};
\node at (-\g,0) {5)};
\draw[line width = 0.17 em] (\m,\n) -- (\m,{\n-\s}) -- ({\m + \s},{\n - \s}) -- ({\m+\s},{\n-2*\s}) -- ({\m+2*\s},{\n-2*\s});
\end{scope}

\begin{scope}[shift={(-\b,0)}]
\ytableausetup{notabloids}
\ytableausetup{mathmode, boxsize=0.7 em}
\node (n) {\ytableausetup{nosmalltableaux}
\ytableausetup{notabloids}
\ydiagram[*(cyan)]{2,1,0}
*[*(white)]{2+1,1+2,3}};
\node at (-\g,0) {6)};
\draw[line width = 0.17 em] (\m,\n) -- (\m,{\n-\s}) -- ({\m + \s},{\n - \s}) -- ({\m+\s},{\n-2*\s}) -- ({\m+2*\s},{\n-2*\s});
\end{scope}

\begin{scope}[shift={(0,0)}]
\ytableausetup{notabloids}
\ytableausetup{mathmode, boxsize=0.7 em}
\node (n) {\ytableausetup{nosmalltableaux}
\ytableausetup{notabloids}
\ydiagram[*(cyan)]{3,1,0}
*[*(white)]{0,1+2,3}};
\node at (-\g,0) {7)};
\draw[line width = 0.17 em] (\m,\n) -- (\m,{\n-\s}) -- ({\m + \s},{\n - \s}) -- ({\m+\s},{\n-2*\s}) -- ({\m+2*\s},{\n-2*\s});
\end{scope}

\begin{scope}[shift={(-6*\b,-\a)}]
\ytableausetup{notabloids}
\ytableausetup{mathmode, boxsize=0.7 em}
\node (n) {\ytableausetup{nosmalltableaux}
\ytableausetup{notabloids}
\ydiagram[*(cyan)]{2,2,0}
*[*(white)]{2+1,2+1,3}};
\node at (-\g,0) {8)};
\draw[line width = 0.17 em] (\m,\n) -- (\m,{\n-\s}) -- ({\m + \s},{\n - \s}) -- ({\m+\s},{\n-2*\s}) -- ({\m+2*\s},{\n-2*\s});
\end{scope}

\begin{scope}[shift={(-5*\b,-\a)}]
\ytableausetup{notabloids}
\ytableausetup{mathmode, boxsize=0.7 em}
\node (n) {\ytableausetup{nosmalltableaux}
\ytableausetup{notabloids}
\ydiagram[*(cyan)]{3,2,0}
*[*(white)]{0,2+1,3}};
\node at (-\g,0) {9)};
\draw[line width = 0.17 em] (\m,\n) -- (\m,{\n-\s}) -- ({\m + \s},{\n - \s}) -- ({\m+\s},{\n-2*\s}) -- ({\m+2*\s},{\n-2*\s});
\end{scope}

\begin{scope}[shift={(-4*\b,-\a)}]
\ytableausetup{notabloids}
\ytableausetup{mathmode, boxsize=0.7 em}
\node (n) {\ytableausetup{nosmalltableaux}
\ytableausetup{notabloids}
\ydiagram[*(cyan)]{3,3,0}
*[*(white)]{0,0,3}};
\node at (-\g-\h,0) {10)};
\draw[line width = 0.17 em] (\m,\n) -- (\m,{\n-\s}) -- ({\m + \s},{\n - \s}) -- ({\m+\s},{\n-2*\s}) -- ({\m+2*\s},{\n-2*\s});
\end{scope}

\begin{scope}[shift={(-3*\b,-\a)}]
\ytableausetup{notabloids}
\ytableausetup{mathmode, boxsize=0.7 em}
\node (n) {\ytableausetup{nosmalltableaux}
\ytableausetup{notabloids}
\ydiagram[*(cyan)]{2,2,2}
*[*(white)]{2+1,2+1,2+1}};
\node at (-\g-\h,0) {11)};
\draw[line width = 0.17 em] (\m,\n) -- (\m,{\n-\s}) -- ({\m + \s},{\n - \s}) -- ({\m+\s},{\n-2*\s}) -- ({\m+2*\s},{\n-2*\s});
\end{scope}

\begin{scope}[shift={(-2*\b,-\a)}]
\ytableausetup{notabloids}
\ytableausetup{mathmode, boxsize=0.7 em}
\node (n) {\ytableausetup{nosmalltableaux}
\ytableausetup{notabloids}
\ydiagram[*(cyan)]{3,2,2}
*[*(white)]{0,2+1,2+1}};
\node at (-\g-\h,0) {12)};
\draw[line width = 0.17 em] (\m,\n) -- (\m,{\n-\s}) -- ({\m + \s},{\n - \s}) -- ({\m+\s},{\n-2*\s}) -- ({\m+2*\s},{\n-2*\s});
\end{scope}

\begin{scope}[shift={(-\b,-\a)}]
\ytableausetup{notabloids}
\ytableausetup{mathmode, boxsize=0.7 em}
\node (n) {\ytableausetup{nosmalltableaux}
\ytableausetup{notabloids}
\ydiagram[*(cyan)]{3,3,2}
*[*(white)]{0,0,2+1}};
\node at (-\g-\h,0) {13)};
\draw[line width = 0.17 em] (\m,\n) -- (\m,{\n-\s}) -- ({\m + \s},{\n - \s}) -- ({\m+\s},{\n-2*\s}) -- ({\m+2*\s},{\n-2*\s});
\end{scope}

\begin{scope}[shift={(-6*\b,-2*\a)}]
\ytableausetup{notabloids}
\ytableausetup{mathmode, boxsize=0.7 em}
\node (n) {\ytableausetup{nosmalltableaux}
\ytableausetup{notabloids}
\ydiagram[*(cyan)]{1,0,0}
*[*(white)]{1+2,3,2}
*[*(green)]{0,0,2+1}};
\node at (-\g-\h,0) {14)};
\draw[line width = 0.17 em] (\m,\n) -- (\m,{\n-\s}) -- ({\m + \s},{\n - \s}) -- ({\m+\s},{\n-2*\s}) -- ({\m+2*\s},{\n-2*\s});
\end{scope}

\begin{scope}[shift={(-5*\b,-2*\a)}]
\ytableausetup{notabloids}
\ytableausetup{mathmode, boxsize=0.7 em}
\node (n) {\ytableausetup{nosmalltableaux}
\ytableausetup{notabloids}
\ydiagram[*(cyan)]{1,0,0}
*[*(white)]{1+2,3,1}
*[*(green)]{0,0,1+2}};
\node at (-\g-\h,0) {15)};
\draw[line width = 0.17 em] (\m,\n) -- (\m,{\n-\s}) -- ({\m + \s},{\n - \s}) -- ({\m+\s},{\n-2*\s}) -- ({\m+2*\s},{\n-2*\s});
\end{scope}

\begin{scope}[shift={(-4*\b,-2*\a)}]
\ytableausetup{notabloids}
\ytableausetup{mathmode, boxsize=0.7 em}
\node (n) {\ytableausetup{nosmalltableaux}
\ytableausetup{notabloids}
\ydiagram[*(cyan)]{1,1,0}
*[*(white)]{1+2,1+2,2}
*[*(green)]{0,0,2+1}};
\node at (-\g-\h,0) {16)};
\draw[line width = 0.17 em] (\m,\n) -- (\m,{\n-\s}) -- ({\m + \s},{\n - \s}) -- ({\m+\s},{\n-2*\s}) -- ({\m+2*\s},{\n-2*\s});
\end{scope}

\begin{scope}[shift={(-3*\b,-2*\a)}]
\ytableausetup{notabloids}
\ytableausetup{mathmode, boxsize=0.7 em}
\node (n) {\ytableausetup{nosmalltableaux}
\ytableausetup{notabloids}
\ydiagram[*(cyan)]{1,1,0}
*[*(white)]{1+2,1+2,1}
*[*(green)]{0,0,1+2}};
\node at (-\g-\h,0) {17)};
\draw[line width = 0.17 em] (\m,\n) -- (\m,{\n-\s}) -- ({\m + \s},{\n - \s}) -- ({\m+\s},{\n-2*\s}) -- ({\m+2*\s},{\n-2*\s});
\end{scope}
\end{tikzpicture}
\caption{The seventeen possibilities in the proof of Lemma \ref{l:startabcompat}.} \label{fig:17casescompat}
\end{figure}
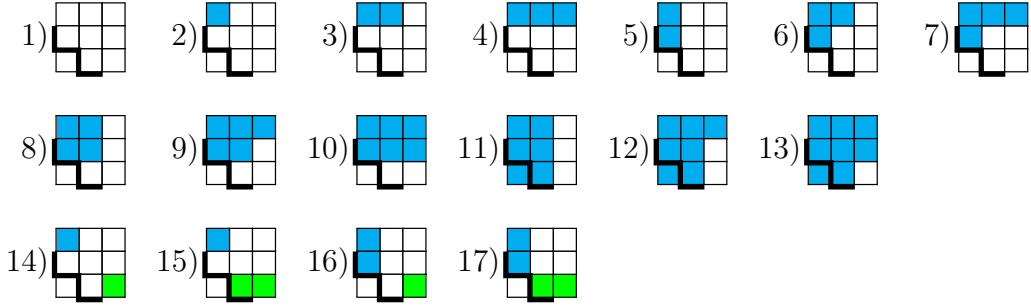
\end{center}

\begin{proof}
There are four major possibilities: (1) no cells in the square are from the DRH or the DRH transpose, (2) there are cells from the DRH, but not from the DRH transpose, (3) there are cells from the DRH transpose, but not from the DRH, and (4) there are cells from both the DRH and the DRH transpose. Observe that (3) can be argued analogously to (2). Thus we will only consider (1), (2), and (4). There are $17$ cases to consider, as shown in Figure \ref{fig:17casescompat}. Note that in Cases 14 - 17, when there are cells from both the DRH and the DRH transpose simultaneously, those cells cannot be in $\{c_{12}, c_{13}, c_{22}, c_{23}\}$. This is as a result of Corollary \ref{c:notbothD&Dt}.

We will deal with the first case by a similar analysis we have been doing in the proofs of previous results. The other $16$ cases, on the other hand, are more straightforward. The fact that the lower DRH staircase path cuts out $c_{31}$ while there is a part of the DRH in the $3 \times 3$-square shows that we are looking at the starting portion of the worm. For these cases, we can compute the frozen coefficients $\fc_{Q_{w'_1}}(c_{22})$ and $\fc_{Q_{w'_2}}(c_{22})$ directly.

Let us start with the first case. Note that because of the DRH staircase path, we know that the middle row label coincides with the left column label and the bottom row label coincides with the middle column label. We may give the labels as follows.

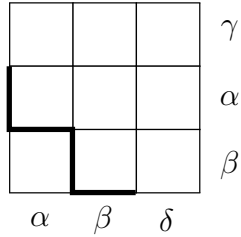
\begin{center}
\begin{tikzpicture}
\def \a {0.84};
\def \s {0.84}
\def \m {-1.13}
\def \n {0.42}

\ytableausetup{notabloids}
\ytableausetup{mathmode, boxsize=2.0 em}
\node (n) {\ytableausetup{nosmalltableaux}
\ytableausetup{notabloids}
\ydiagram[*(cyan)]{0,0,0}
*[*(white)]{3,3,3}};
\draw[line width = 0.17 em] (\m,\n) -- (\m,{\n-\s}) -- ({\m + \s},{\n - \s}) -- ({\m+\s},{\n-2*\s}) -- ({\m+2*\s},{\n-2*\s});

\node at (1.8,0.9) {$\gamma$};
\node at (1.8,0) {$\alpha$};
\node at (1.8,-0.9) {$\beta$};

\node at (1.8-3*\a,0.9-3*\a) {$\alpha$};
\node at (1.8-2*\a,0.9-3*\a) {$\beta$};
\node at (1.8-\a,0.9-3*\a) {$\delta$};
\end{tikzpicture}
\end{center}

Necessarily, $\delta > \beta > \alpha > \gamma$. We also have that the $\alpha$-th and the $\beta$-th letter in $\lambda$ are both $E$. In fact, there are no other $E$'s between $E^{\alpha}$ and $E^{\beta}$ in $\lambda$.

Define $\fq_{ij}$, $\mu_{ij}$, $v_{ij}$, and $[r,s]$ as before. We have that $\mu_{21} = [\alpha, \alpha] = E^{\alpha}$ and $\mu_{32} = [\beta, \beta] = E^{\beta}$.

\underline{Case 1.1.} Suppose there are no other letters between $E^{\alpha}$ and $E^{\beta}$. That is, $\beta = \alpha + 1$. Note that $E^{\beta}$ is not the final letter in $\lambda$, because at least there is $W^{\delta}$ with $\delta > \beta$. \underline{Case 1.1.1.} Suppose that the next letter after $E^{\beta}$ is also an $E$. Then, that letter must be $E^{\delta}$ which gives $\delta = \alpha + 2$. Using the main decomposition theorem, we can compute directly
\[
\fc_{Q_{w'_1}}(c_{22}) = \frac{\Delta_{\alpha+2}}{\Delta_{\alpha+3}}.
\]
We also have, by the main decomposition theorem, that
\[
\fc_{Q_{w_2}}(c_{22}) = \frac{1}{\Delta_\alpha \Delta_{\alpha + 3}}.
\]
By the description in Section \ref{ss:fcdesc}, we have
\[
\fc_{Q_{w_2}}(c_{32}) = \frac{\Delta_\alpha \Delta_{\alpha+2}}{\Delta_{\alpha+1}}.
\]
Hence, after mutating $Q_{w_2}$ at $c_{32}$, we obtain
\[
\fc_{Q_{w'_2}}(c_{22}) = \frac{\Delta_{\alpha+2}}{\Delta_{\alpha+3}} = \fc_{Q_{w'_1}}(c_{22}).
\]
\underline{Case 1.1.2.} Suppose that the next letter after $E^{\beta}$ is $N^{\beta+1}$. In this case, $\delta \ge \alpha + 3$. Like the previous case, we still have
\[
\fc_{Q_{w_2}}(c_{32}) = \frac{\Delta_{\alpha} \Delta_{\alpha+2}}{\Delta_{\alpha+1}}
\]
from the description in Section \ref{ss:fcdesc}. Using the main decomposition theorem, we have
\[
\fc_{Q_{w_2}}(c_{22}) = \frac{1}{\Delta_{\alpha} \Delta_{\alpha+2} \left( \Delta_{\delta+1} \chi^Y(\fq_{22}) \right)}.
\]
Note that $\Delta_{\delta+1} \chi^Y(\fq_{22})$ is either $\Delta_{\delta+1}$ or $\Delta_{l+2} = a_{pq}$ by definition. This means that $\Delta_{\delta+1} \chi^Y(\fq_{22})$ is a single frozen cluster variable that is not adjacent to $c_{32}$ in $w_2$. Mutating $w_2$ at $c_{32}$ gives
\[
\fc_{Q_{w'_2}}(c_{22}) = \frac{1}{\Delta_{\delta+1} \chi^Y(\fq_{22})}.
\]
On the other hand, we directly compute
\[
\fc_{Q_{w'_1}}(c_{22}) = \frac{1}{\Delta_{\delta+1} \chi^Y(\fq_{12})} = \fc_{Q_{w'_2}}(c_{22}).
\]

\underline{Case 1.2.} Suppose there is at least one $N$ between $E^{\alpha}$ and $E^{\beta}$. In this case, we have
\[
\fc_{Q_{w_2}}(c_{32}) = \Delta_\alpha \Delta_{\beta+1}.
\]
Similar to Case 1, the letter $E^\beta$ is not the final one. \underline{Case 1.2.1.} Suppose the next letter after $E^{\beta}$ is an $E$. In this case,
\[
\fc_{Q_{w_2}}(c_{22}) = \frac{1}{\Delta_\alpha \Delta_{\beta} \left( \Delta_{\delta+1} \chi^Y(\fq_{22}) \right)}.
\]
Hence,
\[
\fc_{Q_{w'_2}}(c_{22}) = \frac{\Delta_{\beta+1}}{\Delta_{\beta} \left( \Delta_{\delta+1} \chi^Y(\fq_{22}) \right)}.
\]
On the other hand, it is direct to compute
\[
\fc_{Q_{w'_1}}(c_{22}) = \frac{\Delta_{\beta+1}}{\Delta_{\beta} \Delta_{\delta+1} \chi^Y(\fq_{12})} = \fc_{Q_{w'_2}}(c_{22}).
\]
\underline{Case 1.2.2.} Suppose the next letter after $E^{\beta}$ is an $N$. In this case,
\[
\fc_{Q_{w_2}}(c_{22}) = \frac{1}{\Delta_\alpha \Delta_\beta \Delta_{\beta+1} \left( \Delta_{\delta+1} \chi^Y(\fq_{22}) \right)}.
\]
Thus,
\[
\fc_{Q_{w'_2}}(c_{22}) = \frac{1}{\Delta_{\beta} \left( \Delta_{\delta+1} \chi^Y(\fq_{22}) \right)}.
\]
Using the main decomposition theorem, we see that
\[
\fc_{Q_{w'_1}}(c_{22}) = \frac{1}{\Delta_{\beta} \Delta_{\delta+1} \chi^Y(\fq_{12})} = \fc_{Q_{w'_2}}(c_{22}).
\]

We have finished the analysis of Case 1, in which there are no cells from the DRH or the DRH transpose in the $3 \times 3$-square.  The other $16$ cases are much more straightforward. For each $1 \le r \le l(\lambda)$, we write
\[
\cE(\lambda,r) = \begin{cases}
1 & \text{if the } r^{\text{th}} \text{ letter in } \lambda \text{ is an } E, \\
0 & \text{otherwise.}
\end{cases}
\]
In Table \ref{t:fcequal}, we directly compute the two frozen coefficients in each case and observe that they are equal. Note that in each case, we know what the starting letters in $\lambda$ must be. When appropriate, we use the labels $\alpha$, $\beta$, $\gamma$, and $\delta$ as in the first case. 

We have finished the proof.
\end{proof}

\begin{center}
\setlength{\tabcolsep}{10 pt}
\def\arraystretch{1.5}
\begin{table}
\begin{tabular}{|c|c|c|}
\hline
Case & The lattice path $\lambda$ & $\fc_{Q_{w'_1}}(c_{22})$ and $\fc_{Q_{w'_2}}(c_{22})$ \\ [0.5 ex]
\hline
\hline
2 & $E\underbrace{N \cdots N}_a E \underbrace{\cdots}_b$ \, ($a, b \ge 1$) & $\Delta_{\beta+1}^{\cE(\lambda,\beta+1)} \Delta_\beta^{-1} \left( \Delta_{\delta+1} \chi^Y(\fq_{22}) \right)^{-1}$ \\
\hline
3 & $EEN \cdots$ & $\left(\Delta_{\delta+1} \chi^Y(\fq_{22})\right)^{-1}$ \\
\hline
4 & $EEE \cdots$ & $\Delta_3 \Delta_4^{-1}$ \\
\hline
5 & $\underbrace{N \cdots N}_a E \cdots$ \, ($a \ge 2$) & $\Delta_{\beta+1}^{\cE(\lambda,\beta+1)} \Delta_{\beta}^{-1} \left( \Delta_{\delta+1} \chi^Y(\fq_{22}) \right)^{-1}$ \\
\hline
6 & $NEN \cdots$ & $\Delta_2^{-1} \left( \Delta_{\delta+1} \chi^Y(\fq_{22}) \right)^{-1}$ \\
\hline
7 & $NEE \cdots$ & $\Delta_2^{-1} \Delta_3 \Delta_4^{-1}$ \\
\hline
8 & $ENN \cdots$ & $\left(\Delta_{\delta+1} \chi^Y(\fq_{22})\right)^{-1}$ \\
\hline
9 & $ENE \cdots$ & $\Delta_4^{-1}$ \\
\hline
10 & $EE \underbrace{\cdots}_{a}$ \, ($a \ge 1$) & $\Delta_2 \Delta_3^{-1}$ \\
\hline
11 & $NNN \cdots$ & $\left(\Delta_{\delta+1} \chi^Y(\fq_{22})\right)^{-1}$ \\
\hline
12 & $NNE \cdots$ & $\Delta_4^{-1}$ \\
\hline
13 & $NE\underbrace{\cdots}_{a}$ \, ($a \ge 1$) & $\Delta_2 \Delta_3^{-1}$ \\
\hline
14 & $E\underbrace{N \cdots N}_{a}E$ \, ($a \ge 1$) & $\Delta_{\delta+2} \Delta_{\delta+1}^{-1} a_{pq}^{-1}$ \\
\hline
15 & $E\underbrace{N \cdots N}_{a}$ \, ($a \ge 2$) & $a_{pq}\Delta_{\delta+2}^{-1}$ \\
\hline
16 & $\underbrace{N \cdots N}_{a} E$ \, ($a \ge 2$) & $\Delta_{\delta+2} \Delta_{\delta+1}^{-1} a_{pq}^{-1}$ \\
\hline
17 & $\underbrace{N \cdots N}_{a}$ \, ($a \ge 3$) & $a_{pq}\Delta_{\delta+2}^{-1}$ \\
\hline
\end{tabular}
\bigskip
\caption{Computations of the frozen coefficients in Cases 2 - 17 in the proof of Lemma \ref{l:startabcompat}.} \label{t:fcequal}
\end{table}
\end{center}

The following observation is going to be useful in the proof of the main theorem in the next section.

\begin{obs} \label{o:c1abcompat}
Let $[c_{ij}]_{1\le i,j \le 3}$ be as in Lemma \ref{l:startabcompat}. Let $w_1$ be a worm that starts with $c_{21} \leftarrow c_{11} \rightarrow c_{12}$. Let $w_2$ be a worm that starts with $c_{21} \rightarrow c_{22} \leftarrow c_{12}$. As before, let $Q_w$ denote the proposed quiver for the worm $w$. Let $(w',Q')$ be the resulting quiver after mutating $Q_{w_1}$ at $c_{11}$. Then,
\[
\fc_{Q'}(c_{21}) = \fc_{Q_{w_2}}(c_{21}).
\]
\end{obs}
\begin{proof}
This is a direct consequence the description of frozen coefficients in Section \ref{ss:fcdesc} and the main decomposition theorem.
\end{proof}

\medskip

\subsection{Proof of the main theorem}
In this section, we prove the main theorem, which we restate here.

\main*

\begin{proof}
Following the strategy we outlined earlier, we first need to check that for the initial worm $w = \cM(\lambda)$, the described quiver $Q_w$ agrees with the initial DRH quiver $Q_{\lambda}$. Consider any cell $c$ inside the initial worm $\cM(\lambda$). If $c$ is a bend or an endpoint, the quiver mutation at $c$ brings the initial worm to another worm with $c$ replaced by another cell $c'$ close by that is still in the DRH. We proposed $\fc_{Q_w}(c)$ to be exactly the quiver which respects this worm operation. Since we showed in Lemma \ref{l:insideboa} that the initial DRH quiver $Q_{\lambda}$ also mutates this way, we see that
\[
\fc_{Q_w}(c) = \fc_{Q_{\lambda}}(c)
\]
when $c$ is either a bend or an endpoint. When $c$ is neither a bend nor an endpoint, mutating $Q_w$ at $c$ yields a quiver which does not correspond to any worm inside the DRH staircase. However, we can get around this problem, as we notice that there is a sequence of worm operations which sends $w$ to a worm $w'$ in which $c$ is a bend adjacent to another bend $\ve$ in $w'$ and in which $w'$ is still inside the DRH. We defined
\[
\fc_{Q_w}(c) = \fc_{Q_{w'}}(c) \cdot \fc_{Q_{w'}}(\ve).
\]
The right hand side of the above equation is of the form $\frac{\Delta_i}{\Delta_j}$, where $i$ and $j$ are consecutive indices. It is direct to see that this description coincides with the original description in $Q_{\lambda}$. As a result, the quivers $Q_w$ and $Q_{\lambda}$ agree.

Now, let $w$ be any worm inside the DRH staircase. We have to show that if a worm operation sends $w$ to $w'$, then the corresponding mutation sends the quiver $Q_w$ to $Q_{w'}$. Let the resulting quiver be $Q'$. We will show that $Q' = Q_{w'}$. Worm operations can only transform the worm $w$ at the bends and the endpoints. Suppose we transform at a bend first. Let $w$ contain an $EN$-bend at $c_1 \rightarrow c_2 \leftarrow c_3$. Mutating $Q_w$ at $c_2$ sends the quiver to $Q'$ with $c_1 \leftarrow c'_2 \rightarrow c_3$. Note that $\fc_{Q'}(c'_2) = \fc_{Q_w}(c_2)^{-1}$. Recall that we proposed $\fc_{Q_{w'}}(c'_2) = \FC(u)^{-1} = \fc_{Q_w}(c_2)^{-1}$, where $u$ is the shared vertex between $c_2$ and $c'_2$. Thus, $\fc_{Q'}(c'_2) = \fc_{Q_{w'}}(c'_2)$. Moreover, we have $\fc_{Q'}(c_1) = \fc_{Q_{w'}}(c_1)$ and $\fc_{Q'}(c_3) = \fc_{Q_{w'}}(c_3)$, as a result of Lemmas \ref{l:abcompat} and \ref{l:startabcompat}. Therefore, $Q' = Q_{w'}$. If we mutate $w$ at an endpoint instead, Observation \ref{o:c1abcompat} also ensures $Q' = Q_{w'}$.

Therefore, the collection of all $(w,Q_w)$ correctly describes the mutation behavior of the cluster algebra, and therefore, $c(\mu)$ ($\mu \le \lambda$) together with all $a_{ij}$ that are not previously frozen are all the mutable cluster variables of the cluster algebra $\cA_{\lambda}$.
\end{proof}

\bigskip

\section{Example: the NNEN-DRH}

In this final section, we illustrate an example when $\lambda = NNEN$. We start by drawing the initial NNEN-DRH and its quiver in Figure \ref{fig:NNEN}, and we show the NNEN-DRH staircase in Figure \ref{fig:NNENst}. We fill in the subskeleta corresponding to all the cells in $\cW_{NNEN}$ in Figure \ref{fig:W_NNEN}.

\begin{center}
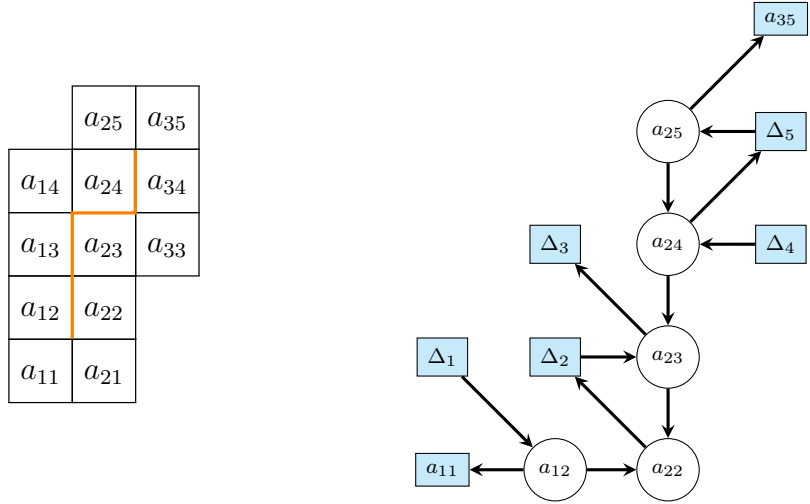
\begin{figure}
\begin{tikzpicture}
\def \a {0.837};

\ytableausetup{notabloids}
\ytableausetup{mathmode, boxsize=2.0em}
\node (n) {\begin{ytableau}
\none & a_{25} & a_{35} \\
a_{14} & a_{24} & a_{34} \\
a_{13} & a_{23} & a_{33} \\
a_{12} & a_{22} & \none \\
a_{11} & a_{21} & \none \\
\end{ytableau}};
\draw[very thick,orange] (-0.42,-0.42 - \a) -- (-0.42,-0.42 + \a) -- (-0.42 + \a,-0.42 + \a) -- (-0.42 + \a,-0.42 + 2*\a);

\begin{scope}[scale = 1.5, shift={(+2.0,-3.0)}, auto=left,every node/.style={circle, fill=cyan!20, minimum width= 1 em, draw = black}]
  \node[fill = none] (M1) at (2,1)  {\tiny $a_{12}$};
  \node[fill = none] (M2) at (3,1)  {\tiny $a_{22}$};
  \node[fill = none] (M3) at (3,2)  {\tiny $a_{23}$};
  \node[fill = none] (M4) at (3,3)  {\tiny $a_{24}$};
  \node[fill = none] (M5) at (3,4)  {\tiny $a_{25}$};
  \node[rectangle] (F1) at (1,1) {\tiny $a_{11}$};
  \node[rectangle] (F2) at (1,2) {\tiny $\Delta_1$};
  \node[rectangle] (F3) at (2,2) {\tiny $\Delta_2$};
  \node[rectangle] (F4) at (2,3) {\tiny $\Delta_3$};
  \node[rectangle] (F5) at (4,3) {\tiny $\Delta_4$};
  \node[rectangle] (F6) at (4,4) {\tiny $\Delta_5$};
  \node[rectangle] (F7) at (4,5) {\tiny $a_{35}$};
  
  \draw[line width = 0.1 em,->,>=stealth] (M5) -- (M4);
  \draw[line width = 0.1 em,->,>=stealth] (M4) -- (M3);
  \draw[line width = 0.1 em,->,>=stealth] (M3) -- (M2);
  \draw[line width = 0.1 em,->,>=stealth] (M1) -- (M2);
  \draw[line width = 0.1 em,->,>=stealth] (F2) -- (M1);
  \draw[line width = 0.1 em,->,>=stealth] (M1) -- (F1);
  \draw[line width = 0.1 em,->,>=stealth] (M2) -- (F3);
  \draw[line width = 0.1 em,->,>=stealth] (F3) -- (M3);
  \draw[line width = 0.1 em,->,>=stealth] (M3) -- (F4);
  \draw[line width = 0.1 em,->,>=stealth] (F5) -- (M4);
  \draw[line width = 0.1 em,->,>=stealth] (M4) -- (F6);
  \draw[line width = 0.1 em,->,>=stealth] (F6) -- (M5);
  \draw[line width = 0.1 em,->,>=stealth] (M5) -- (F7);
\end{scope}
\end{tikzpicture}
\caption{The initial NNEN-DRH and the initial NNEN-DRH quiver.} \label{fig:NNEN}
\end{figure}
\end{center}

\begin{center}
\begin{figure}
\begin{tikzpicture}[scale = 1.5]
\ytableausetup{notabloids}
\ytableausetup{mathmode, boxsize=3.0em}
\node (n) {\ytableausetup{nosmalltableaux}
\ytableausetup{notabloids}
\ydiagram[*(cyan)]{1+2,3,3,2,2}
*[*(green)]{0,0,0,0,0,4+4,3+5,3+3}
*[*(white)]{0,0,3+1,2+3,2+4,2+2}};

\def \a {0.837};

\def \zxs {-3.27};
\def \zys {0.42};

\begin{scope}[shift={(-0.84,+0.42)}]
\draw[red, line width = 1 pt] ({-1.56-2*\a},-0.42-\a) -- ({-1.56+8*\a},-0.42-\a);
\draw[red, line width = 1 pt] ({-1.59+2*\a},{-0.42+5*\a}) -- ({-1.59+2*\a},{-0.42-5*\a});

\draw[line width= 2.0 pt,black] (\zxs+\a, \zys) -- (\zxs+\a, \zys-\a) -- (\zxs+2*\a, \zys-\a) -- (\zxs+2*\a, \zys-2*\a) -- (\zxs+3*\a, \zys-2*\a) -- (\zxs+3*\a, \zys-3*\a) -- (\zxs+4*\a, \zys-3*\a) -- (\zxs+4*\a, \zys-4*\a) -- (\zxs+5*\a, \zys-4*\a) -- (\zxs+5*\a, \zys-5*\a) -- (\zxs+6*\a, \zys-5*\a);
\end{scope}

\begin{scope}[shift = {(-0.84+2*\a,+0.42+3*\a)}]
\draw[line width= 2.0 pt,black] (\zxs, \zys) -- (\zxs+\a, \zys) -- (\zxs+\a, \zys-\a) -- (\zxs+2*\a, \zys-\a) -- (\zxs+2*\a, \zys-2*\a) -- (\zxs+3*\a, \zys-2*\a) -- (\zxs+3*\a, \zys-3*\a) -- (\zxs+4*\a, \zys-3*\a) -- (\zxs+4*\a, \zys-4*\a) -- (\zxs+5*\a, \zys-4*\a) -- (\zxs+5*\a, \zys-5*\a) -- (\zxs+6*\a, \zys-5*\a) -- (\zxs+6*\a, \zys-6*\a) -- (\zxs+7*\a, \zys-6*\a) -- (\zxs+7*\a, \zys-7*\a);
\end{scope}

\def \xs {0.042};
\def \ys {0.15};
\def \frl {0.05};

\begin{scope}
\node[anchor = west] at (\xs-3*\a,\ys) {$\Delta_1$};
\node[anchor = west] at (\xs-3*\a,\ys+\a) {$\Delta_2$};
\node[anchor = west] at (\xs-3*\a,\ys+2*\a) {$\Delta_3$};
\node[anchor = west] at (\xs-2*\a,\ys+2*\a) {$\Delta_4$};
\node[anchor = west] at (\xs-2*\a,\ys+3*\a) {$\Delta_5$};

\node[scale = 0.4, diamond, fill=yellow,minimum width=1 em] at (-2.417+\a,+\a) {\hphantom{o}};
\node[scale = 0.4, diamond, fill=yellow,minimum width=1 em] at (-2.417+3*\a,-2*\a) {\hphantom{o}};
\node[scale = 0.4, diamond, fill=yellow,minimum width=1 em] at (-2.417,+3*\a) {\hphantom{o}};
\node[scale = 0.4, diamond, fill=yellow,minimum width=1 em] at (-2.417+5*\a,-3*\a) {\hphantom{o}};
\end{scope}

\end{tikzpicture}
\caption{The NNEN-DRH staircase.} \label{fig:NNENst}
\end{figure}
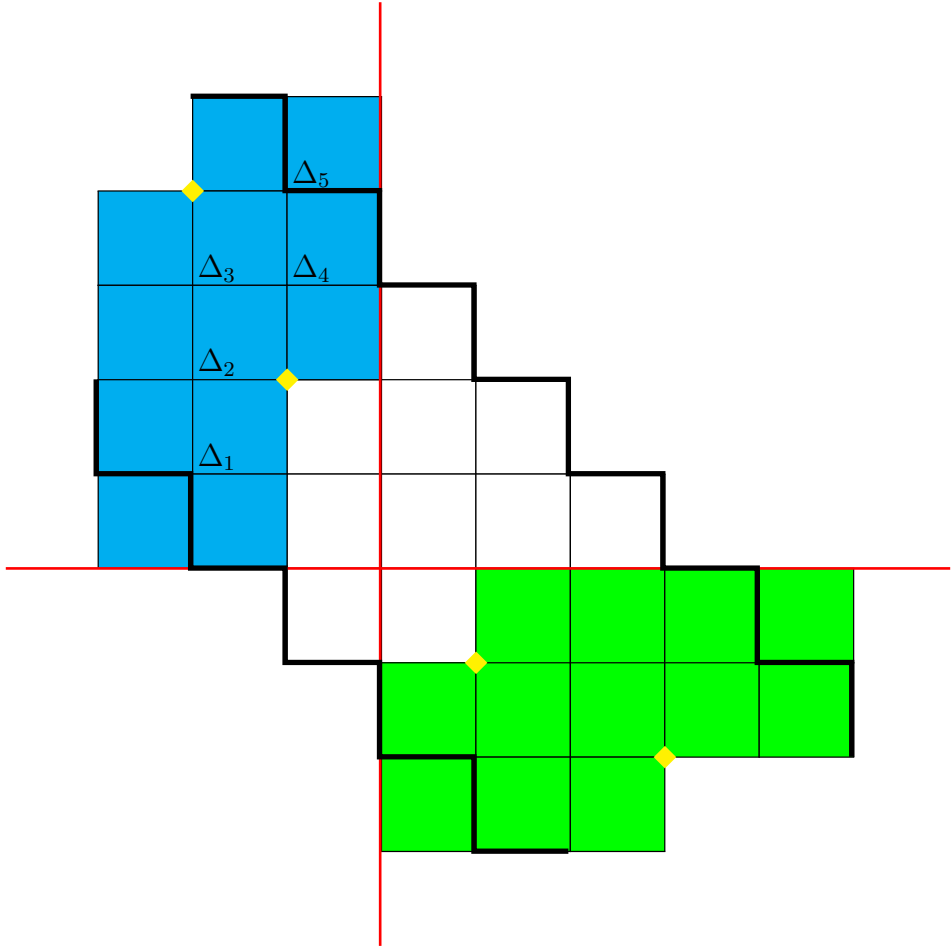
\end{center}

\begin{center}
\begin{figure}
\begin{tikzpicture}
\ytableausetup{notabloids}
\ytableausetup{mathmode, boxsize=5.0em}
\node (n) {\ytableausetup{nosmalltableaux}
\ytableausetup{notabloids}
\ydiagram[*(white)]{1+1,3,4,2}};

\begin{scope}[scale = 2.1, shift = {(-2.45,-2.5)}]
\tiny
\node at (1,1) {$E^3$};
\node at (1,2) {$N^1N^2E^3$};
\node at (1,3) {$N^2E^3$};
\node at (2,1) {$E^3N^4$};
\node at (2,2) {$N^1N^2E^3N^4$};
\node at (2,3) {$N^2E^3N^4$};
\node at (2,4) {$N^4$};
\node at (3,3) {$N^2$};
\node at (3,2) {$N^1N^2$};
\node at (4,2) {$N^1$};
\normalsize

\node at (1,0) {3};
\node at (2,0) {4};
\node at (3,0) {2};
\node at (4,0) {1};

\node at (0,1) {3};
\node at (0,2) {1};
\node at (0,3) {2};
\node at (0,4) {4};
\end{scope}

\begin{scope}[scale = 2.1, shift = {(-2.92,0.02)}]
\draw[line width= 3.0 pt,black] (1,-1) -- (1,-2) -- (2,-2);
\end{scope}

\begin{scope}[scale = 2.1, shift = {(-0.94,2.98)}]
\draw[line width= 3.0 pt,black] (0,-1) -- (1,-1) -- (1,-2) -- (2,-2) -- (2,-3) -- (3,-3) -- (3,-4);
\end{scope}

\end{tikzpicture}
\caption{$\Subsk \left(\cW_{NNEN}\right)$} \label{fig:W_NNEN}
\end{figure}
\end{center}

To obtain all the mutable cluster variables, we standardize the subskeleta as follows.

\begin{multicols}{3}
    \begin{itemize}
        \item $N^1 \mapsto \begin{vmatrix}
a_{13} & a_{23} \\
a_{11} & a_{21}
\end{vmatrix}$,
        \item $N^2 \mapsto \begin{vmatrix}
a_{14} & a_{24} \\
a_{12} & a_{22}
\end{vmatrix}$,
        \item $E^3 \mapsto \begin{vmatrix}
a_{14} & a_{34} \\
a_{13} & a_{33}
\end{vmatrix}$,
    \end{itemize}
    \end{multicols}

\begin{multicols}{2}
    \begin{itemize}
        \item $N^4 \mapsto \begin{vmatrix}
a_{25} & a_{35} \\
a_{23} & a_{33}
\end{vmatrix}$,
		\item $N^1N^2 \mapsto \begin{vmatrix}
a_{14} & a_{24} \\
a_{11} & a_{12}
\end{vmatrix}$,
    \end{itemize}
    \end{multicols}
    
\begin{multicols}{2}
    \begin{itemize}
		\item $N^2E^3 \mapsto - \begin{vmatrix}
a_{14} & a_{24} & a_{34} \\
a_{13} & a_{23} & a_{33} \\
a_{12} & a_{22} & 0
\end{vmatrix}$,
		\item $E^3N^4 \mapsto - \begin{vmatrix}
0 & a_{25} & a_{35} \\
a_{14} & a_{24} & a_{34} \\
a_{13} & a_{23} & a_{33}
\end{vmatrix}$,
    \end{itemize}
    \end{multicols}
    
\begin{itemize}
\item $N^1N^2E^3 \mapsto - \begin{vmatrix}
a_{14} & a_{24} & a_{34} \\
a_{13} & a_{23} & a_{33} \\
a_{11} & a_{21} & 0
\end{vmatrix}$,
\end{itemize}

\begin{itemize}
\item $N^2E^3N^4 \mapsto - \begin{vmatrix}
0 & 0 & a_{25} & a_{35} \\
a_{14} & a_{24} & a_{24} & a_{34} \\
a_{13} & a_{23} & a_{23} & a_{33} \\
a_{12} & a_{22} & 0 & 0
\end{vmatrix}$, and
\end{itemize}

\begin{itemize}
\item $N^1N^2E^3N^4 \mapsto - \begin{vmatrix}
0 & 0 & a_{25} & a_{35} \\
a_{14} & a_{24} & a_{24} & a_{34} \\
a_{13} & a_{23} & a_{23} & a_{33} \\
a_{11} & a_{21} & 0 & 0
\end{vmatrix}$.
\end{itemize}

Therefore, all the mutable cluster variables for the cluster algebra $\cA_{NNEN}$ are $a_{12}$, $a_{13}$, $a_{14}$, $a_{21}$, $a_{22}$, $a_{23}$, $a_{24}$, $a_{25}$, $a_{33}$, $a_{34}$, $\begin{vmatrix}
a_{13} & a_{23} \\
a_{11} & a_{21}
\end{vmatrix}$, $\begin{vmatrix}
a_{14} & a_{24} \\
a_{12} & a_{22}
\end{vmatrix}$, $\begin{vmatrix}
a_{14} & a_{34} \\
a_{13} & a_{33}
\end{vmatrix}$, $\begin{vmatrix}
a_{25} & a_{35} \\
a_{23} & a_{33}
\end{vmatrix}$, $\begin{vmatrix}
a_{14} & a_{24} \\
a_{11} & a_{12}
\end{vmatrix}$, $- \begin{vmatrix}
a_{14} & a_{24} & a_{34} \\
a_{13} & a_{23} & a_{33} \\
a_{12} & a_{22} & 0
\end{vmatrix}$, $- \begin{vmatrix}
0 & a_{25} & a_{35} \\
a_{14} & a_{24} & a_{34} \\
a_{13} & a_{23} & a_{33}
\end{vmatrix}$, $-\begin{vmatrix}
a_{14} & a_{24} & a_{34} \\
a_{13} & a_{23} & a_{33} \\
a_{11} & a_{21} & 0
\end{vmatrix}$, 
\newline $- \begin{vmatrix}
0 & 0 & a_{25} & a_{35} \\
a_{14} & a_{24} & a_{24} & a_{34} \\
a_{13} & a_{23} & a_{23} & a_{33} \\
a_{12} & a_{22} & 0 & 0
\end{vmatrix}$, and $- \begin{vmatrix}
0 & 0 & a_{25} & a_{35} \\
a_{14} & a_{24} & a_{24} & a_{34} \\
a_{13} & a_{23} & a_{23} & a_{33} \\
a_{11} & a_{21} & 0 & 0
\end{vmatrix}$.

\bigskip

\section{Acknowledgments}
This research project was conducted in the summer 2015 Combinatorics REU at the University of Minnesota, Twin Cities, and in the summer of 2016 at the same university. We would like to thank the university for hosting us both times. We are immensely grateful to Professor Pavlo Pylyavskyy and Professor Victor Reiner for constantly motivating us to finish this project. We would like to thank all the mentors and the TAs in the program for their help and motivation. In particular, we thank Professor Pavlo Pylyavskyy for introducing the Double Rim Hook algebras to us, and our project TA, Dr. Thomas McConville for many helpful ideas. PJ would also like to thank all the supports from his parents and his friends during the completion of this project.

MC was partially supported through the NSF grants DMS--1503119 and DMS--1148634.  PJ's undergraduate studies at Harvard University were supported by King's Scholarship (Thailand), while his visits to UMN were supported through the grant DMS--1351590. JS's visit, as well as the REU program in general, was supported through the grant DMS--1148634.

\bigskip

\bibliographystyle{alpha}
\bibliography{snakes_ref}

\begin{thebibliography}{BFZ05}

\bibitem[BFZ05]{ca3}
Arkady Berenstein, Sergey Fomin, and Andrei Zelevinsky.
\newblock Cluster algebras. {III}. {U}pper bounds and double {B}ruhat cells.
\newblock {\em Duke Math. J.}, 126(1):1--52, 2005.

\bibitem[CAP]{CAP}
{C}luster {A}lgebras {P}ortal.
\newblock \url{http://www.math.lsa.umich.edu/~fomin/cluster.html}.

\bibitem[Dod67]{Dod1866}
Charles~L Dodgson.
\newblock {IV}. {C}ondensation of determinants, being a new and brief method
  for computing their arithmetical values.
\newblock {\em Proceedings of the Royal Society of London}, (15):150--155,
  1867.

\bibitem[FR07]{fomin-reading}
Sergey Fomin and Nathan Reading.
\newblock Root systems and generalized associahedra.
\newblock In {\em Geometric combinatorics}, volume~13 of {\em IAS/Park City
  Math. Ser.}, pages 63--131. Amer. Math. Soc., Providence, RI, 2007.

\bibitem[FZ02]{ca1}
Sergey Fomin and Andrei Zelevinsky.
\newblock Cluster algebras. {I}. {F}oundations.
\newblock {\em J. Amer. Math. Soc.}, 15(2):497--529, 2002.

\bibitem[Lec93]{Lec93}
Bernard Leclerc.
\newblock On identities satisfied by minors of a matrix.
\newblock {\em Adv. Math.}, 100(1):101--132, 1993.

\bibitem[Wil14]{Wil14}
Lauren~K. Williams.
\newblock Cluster algebras: an introduction.
\newblock {\em Bull. Amer. Math. Soc. (N.S.)}, 51(1):1--26, 2014.

\end{thebibliography}

\end{document}